\documentclass[a4paper,10pt]{amsart}
\usepackage{txfonts}
\usepackage{amsfonts,amsmath,amssymb,amscd,bbm,amsthm,mathrsfs,bm,enumitem}
\usepackage{extarrows}
\usepackage{latexsym}
\usepackage{graphicx}
\usepackage{color}
\usepackage{xcolor}
\usepackage{xy}
\usepackage{leftidx}
\usepackage{pifont}
\input xy
\input xypic
\xyoption{all}
\usepackage{mathtools}

\allowdisplaybreaks
\newtheorem{theorem}{Theorem}[section]
\newtheorem{lemma}[theorem]{Lemma}
\newtheorem{proposition}[theorem]{Proposition}
\newtheorem{corollary}[theorem]{Corollary}

\theoremstyle{definition}
\newtheorem{definition}[theorem]{Definition}

\newtheorem{remark}[theorem]{Remark}

\numberwithin{equation}{section}

\newtheorem{example}[theorem]{Example}
 \DeclareMathOperator{\Ker}{Ker}
\DeclareMathOperator{\Image}{Im} \DeclareMathOperator{\Ext}{Ext}
 
\DeclareMathOperator{\Hom}{Hom}

\DeclareMathOperator{\uExt}{\underline{Ext}}

\DeclareMathOperator{\uHom}{\underline{Hom}}

\DeclareMathOperator{\id}{id}

\DeclareMathOperator{\Gr}{Gr}

\DeclareMathOperator{\hdet}{hdet}
\def\bt{\begin{theorem}}
\def\et{\end{theorem}}
\def\bl{\begin{lemma}}
\def\el{\end{lemma}}
\def\br{\begin{remark}}
\def\er{\end{remark}}
\def\bc{\begin{corollary}}
\def\ec{\end{corollary}}

\begin{document}

\title[Nakayama automorphisms of graded double Ore extensions]{Nakayama automorphisms of graded double Ore extensions of Koszul Artin-Schelter regular algebras with nontrivial skew derivations}

\author{Yan Cao}
\address{Cao: Department of Mathematics, Zhejiang Sci-Tech University, Hangzhou 310018, China}
\email{caoyannnnn@163.com}

\author{Yuan Shen}
\address{Shen: Department of Mathematics, Zhejiang Sci-Tech University, Hangzhou 310018, China}
\email{yuanshen@zstu.edu.cn}

\author{Xin Wang}
\address{Wang: School of Science, Shandong Jianzhu University, Jinan 250101, China}
\email{wangxin19@sdjzu.edu.cn}


\date{}

\begin{abstract}
Let $A$ be a Koszul Artin-Schelter regular algebra and  $B=A_P[y_1,y_2;\varsigma,\nu]$ be a graded double Ore extension of $A$ where $\varsigma:A\to M_{2\times 2}(A)$ is a graded algebra homomorphism and $\nu:A\to A^{\oplus 2}$ is a degree one $\varsigma$-derivation. We construct a minimal free resolution for the trivial module of $B$, and it implies that $B$ is still Koszul. We introduce a homological invariant called $\varsigma$-divergence of $\nu$, and with its aid,  we obtain a precise description of the Nakayama automorphism of $B$. A twisted superpotential $\hat{\omega}$ for $B$ with respect to the Nakayama automorphism is constructed so that $B$ is isomorphic to the derivation quotient algebra of $\hat{\omega}$.
\end{abstract}



\maketitle

\section*{Introduction}
Nakayama automorphism is an important homological invariant for skew Calabi-Yau algebras. It is closely related to noncommutative invariant theory, Zariski cancellation problem, and Batalin-Vilkovisky algebra structures on the Hochschild cohomology of skew Calabi-Yau algebras, and so on (see e.g. \cite{CWZ, KK, LMZ}). However,  it is not easy to compute. A great deal of effort has been made to describe Nakayama automorphisms of skew Calabi-Yau algebras in ungraded and graded cases (see e.g. \cite{HVZ,LM,LWW,LMZ,LMZ1,RRZ1,SL,SZL,WZ,ZSL,ZVZ}, and references therein).

Artin-Schelter regular algebras, as connected graded skew Calabi-Yau algebras, originate from the study of noncommutative projective algebraic geometry. The classification, ring-theoretic and homological properties of Artin-Schelter regular algebras have undergone decades of study. Notably, algebraic extensions, such as regular normal extensions, graded Ore extensions and graded double Ore extensions, etc., have built a bridge between low-dimensional and high-dimensional Artin-Schelter regular algebras. Hence, to study the behavior of various properties under different algebraic extensions is of great help in understanding Artin-Schelter regular algebras. This approach also applies to the characterization of Nakayama automorphisms. There is a precise homological identity for Nakayama automorphisms of Artin-Schelter regular algebras under regular normal extensions (see \cite{RRZ1,ZSL}), and the alteration of Nakayama automorphisms of Artin-Schelter regular algebras under graded Ore extensions is addressed in \cite{SG,WZ}  for most cases.

Double Ore extension was introduced by Zhang and Zhang, primarily for classification of $4$-dimensional Koszul Artin-Schelter regular algebras (see \cite{ZZ1,ZZ2}). Roughly speaking, a double Ore extension of an algebra $A$ is adding two new variables $y_1,y_2$ to $A$ simultaneously, and there are some relations among $y_1$, $y_2$ and elements of $A$ associated with a parameter $P=\{p_{12},p_{11}\}$ constituted by two elements of the ground field, a tail $\kappa=\{\kappa_0,\kappa_1,\kappa_2\}\subseteq A$, an algebra homomorphism $\varsigma=(\varsigma_{ij}):A\to M_{2\times 2}(A)$ and a $\varsigma$-derivation $\nu=(\nu_i):A\to A^{\oplus 2}$. Typically, a double Ore extension is denoted by $A_P[y_1,y_2;\varsigma,\nu,\kappa]$. Compared with a length two iterated Ore extension, a double Ore extension may produce much more classes of algebras, while it appears more mysterious since the images of $\varsigma$ and $\nu$ are formed by matrices.

Zhang and Zhang proved that graded double Ore extensions preserve Artin-Schelter regularity in \cite{ZZ1}, yet whether many other properties are preserved remains open, such as domain, noetherian property, Auslander regularity, etc. Fortunately, Zhu, Van Oystaeyen and Zhang proved that trimmed graded double Ore extensions, whose tails and skew derivations are both trivial, preserve Koszulity, and they compute Nakayama automorphisms of trimmed graded double Ore extensions of any Koszul Artin-Schelter regular algebras in \cite{ZVZ}. In this paper, nontrivial skew derivations will be taken into account, and we aim to describe the Nakayama automorphism of graded double Ore extension $A_P[y_1,y_2;\varsigma,\nu]$ precisely when $A$ is a Koszul Artin-Schelter regular algebra and $\nu$ is nontrivial. 

Let $A=T(V)/(R)$ be a Koszul algebra, and $B=A_P[y_1,y_2;\varsigma,\nu]$ be a graded double Ore extension of $A$, where the element $p_{12}$ in $P$ is nonzero. In this case, $\varsigma$ is invertible in the sense of Definition \ref{def: inverse}, and we find that $\varsigma^T$ is also invertible whose inverse is denoted by $\varsigma^{-T}$ (see Proposition \ref{prop: invertible means T-invertible}). To deal with the $\varsigma$-derivations $\nu$, we make use of a method comparable to the one used for graded Ore extensions in \cite{SG}. We construct a quadruple consisting of four classes of linear maps for $\nu$:
$$
\begin{array}{ll}
\delta_{i,r}:W_i\to (W_i\otimes V)^{\oplus 2}, &\delta_{i,l}:W_i\to (V\otimes W_i)^{\oplus 2}, \\
\upsilon_{i,r}:W_i\to W_{i+1}\otimes V,&\upsilon_{i,l}:W_i\to V\otimes W_{i+1},
\end{array}
$$
where $W_1=V$ and $W_j=\bigcap_{s=0}^{j-2} V^{\otimes s}\otimes R\otimes V^{\otimes j-2-s}$ where $j\geq 2$ and $i\geq 1$ (see Lemma \ref{lemma: construction of delta_r and delta_l}, Lemma \ref{lem: construction of upsilon} and Lemma \ref{lem: construction of upsilon l}). It is worth noting that $\{\upsilon_{i,r}\}$ and two left ones $\{\delta_{i,l}\}$ and $\{\upsilon_{i,l}\}$ are all related with  some homotopies. 

The first main result is that from the Koszul complex for the trivial module of $A$, we construct a minimal free resolution for the trivial module of $B$ whose differentials have been given concrete expression involving $\{\delta_{i,r}\}$ and $\{\upsilon_{i,r}\}$ (see Theorem \ref{thm: minimial free resolution}). The technique we used is taking a length two iterated mapping cone construction. As a consequence, we directly prove that $B=A_P[y_1,y_2;\varsigma,\nu]$ with nontrivial skew derivation is still Koszul (see Corollary \ref{cor: double ore extension preserve Koszul}).

Assume $A$ is also an Artin-Schelter regular algebra of dimension $d$. Then $W_d$ is $1$-dimensional, and let $\omega$ be a basis of $W_d$. As a generalization of \cite[Theorem 1.2]{MS}, we have the following result.
\begin{theorem}(Theorem \ref{thm: hdet is sigma d})
    Let $A=T(V)/(R)$ be a Koszul Artin-Schelter regular algebra of dimension $d$ and $\varsigma:A\to M_{2\times 2}(A)$ be an invertible graded algebra homomorphism. Write $\sigma=\varsigma_{\mid V}$, Then
    $$
    \sigma^{\boxtimes d}(\omega)=(\hdet\varsigma)\omega.
    $$
\end{theorem}
\noindent The definition of notation ``$\boxtimes$'' can be found in Section \ref{subsec: notations}, and the homological determinant $\hdet$ of $\varsigma$ is defined firstly in \cite{ZVZ} (see Section \ref{subsection: homological determinant}, or \cite{SZL} for a more general case). Besides, there is a pair  of elements $\delta_r=(\delta_{r;i})$ and $\delta_l=(\delta_{l;i})$ in $V^{\oplus 2}$ such that
$$
\delta_{d,r}(\omega)=\begin{pmatrix}
        \omega\otimes \delta_{r;1}\\
    \omega\otimes\delta_{r;2}
\end{pmatrix},\qquad
\delta_{d,l}(\omega)=\begin{pmatrix}
       \delta_{l;1}\otimes \omega\\
    \delta_{l;2}\otimes\omega
\end{pmatrix}.
$$
Although the construction of the quadruple $(\{\delta_{i,r}\},\{\delta_{i,l}\},\{\upsilon_{i,r}\},\{\upsilon_{i,l}\})$ is not unique, the element 
$$
\delta_r+\mu_A^{\oplus 2}(\sigma^{-T}\cdot\delta_l)
$$
in $V^{\oplus 2}$ is independent on the choices of quadruples for $\nu$, where $\sigma=\varsigma_{\mid V}$ and $\mu_A$ is the Nakayama automorphism of $A$ (see Proposition \ref{coro: relation between delta_r and delta l}). It is an analogous to the case of skew derivations in Ore extensions discussed in \cite{LM,SG}, we also call it the \emph{$\varsigma$-divergence} of $\nu$, denoted by $\mathrm{div}_{\varsigma} \nu$.

Van den Bergh showed that describing the Nakayama automorphism of $B$ can be achieved by computing the Yoneda product of $E(B)$ in \cite{V}. Based on the  preparations above, we compute the Yoneda product of the Ext-algebra $E(B)$, and then obtain another main result.

\begin{theorem}(Theorem \ref{thm: nakayama automorphism of double Ore extension}) Let $B=A_P[y_1,y_2;\varsigma,\nu]$ be a graded double Ore extension of a Koszul Artin-Schelter regular algebra $A$. Then the Nakayama automorphism $\mu_B$ of $B$ satisfies
\begin{align*}
{\mu_{B}}_{\mid A}=(\det\varsigma)^{-1}\mu_A,\qquad
\mu_B\begin{pmatrix}
    y_1\\y_2
\end{pmatrix}
=-(\mathbb{J}^T)^{-1}\mathbb{J}\hdet\varsigma
\begin{pmatrix}
    y_1\\y_2
\end{pmatrix}
-(\mathbb{J}^T)^{-1}\mathbb{J} \mathrm{div}_{\varsigma}\nu,
\end{align*}
where $\mu_A$ is the Nakayama automorphism of $A$,  $\det\varsigma$ is the determinant of $\varsigma$ (see \cite{ZZ1}, or Definition \ref{def: determinant of varsigma}) and $$\mathbb{J}=\begin{pmatrix} -p_{11}&-p_{12}\\
1&0
\end{pmatrix}.$$
\end{theorem}

Moreover, any Koszul Artin-Schelter regular algebra is isomorphic to a derivation quotient algebra defined by a twisted superpotential with respect to the Nakayama automorphism (see \cite{BSW,DV,HVZ}). Due to this fact, our last main result is the construction of a twisted superpotential $\hat{\omega}$ for $B=A_P[y_1,y_2;\varsigma,\nu]$ such  that $B$ is a derivation quotient algebra defined by $\hat{\omega}$ (see Theorem \ref{thm: twisted superpotential}). A few more words, the graded double Ore extension $B=A_P[y_1,y_2;\varsigma,\nu]$ we discussed can be viewed as a twisted tensor product of $A$ and a $2$-dimensional Artin-Schelter regular algebra $J=k\langle y_1,y_2\rangle/(y_2y_1-p_{12}y_1y_2-p_{11}y_1^2)$. It is well known $J$ is Koszul and there is a twisted superpotential $\omega_J=y_2y_1-p_{12}y_1y_2-p_{11}y_1^2$ for $J$. Although the expression of the twisted superpotential $\hat{\omega}$ is complicated, it is something like a ``twisted tensor product'' of $\omega_J$ and $\omega_A$ satisfying skew symmetry property, where $\omega_A$ is  a twisted superpotential for $A$. This fact also provides us some perspectives for constructing and classifying high-dimensional Koszul Artin-Schelter regular algebras. 

The paper is organized as follows. In Section 1, we recall some definitions and results including some homological properties of Koszul Artin-Schelter regular algebras and homological determinant for graded algebra homomorphisms from an algebra to $2\times 2$ matrix algebra over itself, and also fix some notations. In Section 2, we construct minimal free resolutions for trivial modules of graded double Ore extensions of Koszul algebras, by introducing a quadruple consisting of four classes of linear maps from skew derivations. In Section 3, we describe Nakayama automorphisms and construct twisted superpotentials for graded double Ore extensions of Koszul Artin-Schelter regular algebras. Section 4 is devoted to show the proofs of results about the Yoneda products (Lemma \ref{lemma: Yoneda product}) and the twisted superpotentials for graded double Ore extensions (Theorem \ref{thm: twisted superpotential}).

\section{Preliminaries}\label{Section preliminaries}
Let $A=\bigoplus_{i\in\mathbb{Z}} A_i$ be a  graded algebra. If $A_i=0$ for all $i<0$ and $\dim A_0=1$, $A$ is called \emph{connected}. Write $\Gr A$ for the category of right graded $A$-modules with morphisms consisting of right $A$-module homomorphisms preserving degrees. For any $M, N\in \Gr A$, the right graded $A$-module $M(n)$ is an $n$-th \emph{shift} of $M$ whose homogeneous space $M(n)_i=M_{n+i}$ for all $i\in\mathbb{Z}$, and  the functor $\uHom$ and its derived functor defined by
$$
\uHom_A(M,N)=\bigoplus\nolimits_{i\in \mathbb{Z}}\Hom_{\Gr A}(M,N(i))\quad\text{and}\quad  \uExt^n_A(M,N)=\bigoplus\nolimits_{i\in \mathbb{Z}}\Ext^n_{\Gr A}(M,N(i)).
$$
Let $L$ be a graded $A$-bimodule and $\mu$ be a graded automorphism of $A$. The graded twisted  $A$-bimodule $L^{\mu}$ is a graded $A$-bimodule with the $A$-bimodule structure $a\cdot x\cdot b=ax\mu(b)$ for any $a,b\in A,x\in L$.

Let $A$ be a connected graded algebra. Write $\varepsilon_A$ for the augmentation from $A$ to $k$. There exists a minimal graded free resolution of the right graded trivial module $k_A$ of $A$,
\begin{equation}\label{general Resolution of k_A}
(\cdots\xlongrightarrow{} P_{d}\xlongrightarrow{\partial_d} \cdots\xlongrightarrow{\partial_3} P_{2}\xlongrightarrow{\partial_2} P_1\xlongrightarrow{\partial_1} P_0\xlongrightarrow{} 0\to \cdots)\ \xlongrightarrow{\varepsilon_A}\ k_A,
\end{equation}
that is, $\Ker \partial_i\subseteq P_{i}A_{\geq 1}$ for any $i\geq 0$. The graded vector space $E(A)=\bigoplus_{i\geq0}\uExt^i_A(k,k)=\bigoplus_{i\geq0}\uHom_A(P_i,k)$ with the Yoneda product ``$\ast$'' is a connected graded algebra, called the \emph{Ext-algebra} of $A$.

\begin{definition}
Let $A=T(V)/(R)$ be a connected graded algebra where $T(V)$ is the tensor algebra over a finite dimensional vector space $V$ and vector space $R\subseteq V^{\otimes 2}$. If each right graded $A$-module $P_i$ in the minimal graded free resolution (\ref{general Resolution of k_A}) of $k_A$ is generated by degree $i$ for any $i\geq0$, we say $A$ is a \emph{Koszul algebra}.
\end{definition}

Let $V$ be a finite dimensional vector space and $A=T(V)/(R)$ be a Koszul algebra. Write $W_0=k$, $W_1=V$ and
$
W_{i}=\bigcap_{0\leq s\leq i-2} V^{\otimes s}\otimes R\otimes V^{\otimes i-s-2}
$  for $i\geq 2$, and  $\pi_A$ for the canonical projection from $T(V)$ to $A$. 
Then the following Koszul complex is  a minimal graded free resolution of $k_A$ via $\varepsilon_A$,
$$
(\cdots\xlongrightarrow{} W_d\otimes A\xlongrightarrow{\partial^A_d}W_{d-1}\otimes A\xlongrightarrow{\partial^A_{d-1}}\cdots\xlongrightarrow{\partial^A_{3}}W_2\otimes A\xlongrightarrow{\partial^A_{2}}W_1\otimes A\xlongrightarrow{\partial^A_{1}}A\xlongrightarrow{}0\to\cdots)\ \xlongrightarrow{\varepsilon_A}\  k_A,
$$
where $\partial^A_{i}=(\id_V^{\otimes i-1}\otimes m_A)(\id_V^{\otimes i-1}\otimes{\pi_A}_{\mid V}\otimes\id_A)=\id_V^{\otimes i-1}\otimes m_A$ and $m_A$ is the multiplication of $A$ for $i\geq 1$. Hence, the Ext-algebra 
$
E(A)=\bigoplus_{i\geq0}\uHom_A(W_i\otimes A,k_A)\cong \bigoplus_{i\geq0} W_i^{*}
$ as vector spaces.

Since the homogeneous space $A_1=V$ as vector spaces, we use $V$ to represent the vector space $V$ or the homogeneous space $A_1$ freely in the sequel. Also, we write $m_A$ shortly for the linear map $m_A({\pi_A}_{\mid V}\otimes \id_A):V\otimes A\to A$.

Write $\tau:V\otimes V\to V\otimes V$ for the flipping map. We adopt the notation in \cite{HVZ} of a sequence of linear endomorphisms of $V^{\otimes d}$ for any $d\geq 2$:
$$
\tau_{d}^0=\id^{\otimes d},\quad 
\tau_d^{i}=(\id^{\otimes i-1}\otimes \tau\otimes\id^{\otimes d-i-1 })\tau_{d}^{i-1},\quad\text{for any }~ 1\leq i\leq d-1.
$$

\begin{definition}
Let $V$ be a finite dimensional vector space and $\sigma$ be a linear automorphism of $V$. If an element $\omega\in V^{\otimes d}$ for some $d\geq 2$ satisfies
$$
\omega=(-1)^{d-1}\tau_{d}^{d-1}(\sigma\otimes\id^{\otimes d-1})(\omega),
$$
then $\omega$ is called a \emph{$\sigma$-twisted superpotential}. 
\end{definition}

\begin{definition}
Let $V$ be a finite dimensional vector space and $\omega\in V^{\otimes d}$ be a $\sigma$-twisted superpotential for some linear automorphism $\sigma$ of $V$ and $d\geq 2$. The \emph{$i$-th derivation quotient algebra} $\mathcal{D}(\omega,i)$ of $\omega$ is
$$
\mathcal{D}(\omega,i)=T(V)/(\partial_{\psi}(\omega),\psi\in (V^*)^{\otimes i}),
$$
where $V^*$ is the linear dual of $V$ and  $
\partial_{\psi}(\omega)=(\id^{\otimes d-i}\otimes \psi)(\omega)
$ is the \emph{partial derivation} of $\omega$.
\end{definition}

\begin{definition}
A connected graded algebra $C$ is \emph{Artin-Schelter regular}  (\emph{AS-regular}, for short) of dimension $d$, if it  has finite global dimension $d$, $\uExt_{C}^{i}(k, C)=0$ for $i\neq d$ and $\dim(\uExt_{C}^{d}(k, C))=1$. 
\end{definition}

By \cite[Lemma 1.2]{RRZ1}, an AS-regular algebra $C$ of dimension $d$ is endowed with the \emph{Nakayama automorphism} $\mu_C$, that is, $\mu_C$ is a  graded algebra automorphism of $C$ such that 
$$\uExt^i_{C^e}(C,C^e)\cong\left\{\begin{array}{ll} 0&i\neq d,\\C^{\mu_C}(l)&i=d,   \end{array}\right. \quad \text{as graded $C$-bimodules},$$
where $C^e=C\otimes C^{op}$ is the enveloping algebra of $C$ for some $l\geq 0$.

Let $E=\bigoplus_{i\in \mathbb{Z}}E^i$ be a finite dimensional graded algebra. We say $E$ is a \emph{graded Frobenius algebra}, if there exists a nondegenerate associative graded bilinear form $\langle -,-\rangle: E\otimes E\to k(d)$ for some $d\in \mathbb{Z}$. In particular, there exists a graded automorphism $\mu_E$ of $E$ such that
$$
\langle \alpha,\beta\rangle=(-1)^{ij}\langle \beta,\mu_E(\alpha)\rangle,
$$
for any $\alpha\in E^i,\beta\in E^j$. We call the automorphism $\mu_E$ is the \emph{Nakayama automorphism} of the graded Frobenius algebra $E$ (see \cite{Sm} for details).

We list some important results for Koszul AS-regular algebras below.

\begin{theorem}\label{thm: properties of Koszul regular algebras}Let $A$ be a Koszul AS-regular algebra of dimension $d$ with the Nakayama automorphism $\mu_A$.
\begin{enumerate}
\item \cite[Proposition 5.10]{Sm} The Ext-algebra $E(A)$ of $A$ is a graded Frobenius algebra.
\item\cite[Proposition 3]{V} Let $\mu_E$ be the Nakayama automorphism of $E(A)$, then
    $
    {\mu_E}_{|E^1(A)}=({\mu_A}_{|A_1})^*,
    $ 
    where $E^1(A)$ is identified with $V^*=A_1^*$.
 \item \cite[Lemma 4.3 and Theorem 4.4(i)]{HVZ}
 Any nonzero element $\omega\in W_d$ is a ${\mu_A}_{\mid V}$-twisted superpotential, and
 $
A\cong \mathcal{D}(\omega,d-2).
$
\end{enumerate}
\end{theorem}

\subsection{Notations}\label{subsec: notations}
Throughout the paper, $k$ is a fixed field. All vector spaces and algebras are over $k$. Unless otherwise stated, the tensor product $\otimes$ means $\otimes_k$. We write $M_{r\times s}(V)$ for $r\times s$ matrix space over some vector space $V$. 

In this paper, we mainly deal with maps from some vector space (or algebra) to the  matrix space (or matrix algebra) over itself. It is necessary to fix some notations to keep us in the setting of such maps.

Let $f=(f_{ij}):V_1\to M_{r\times s}(V_2), g=(g_{ij}): W_1\to M_{s\times t}(W_2)$ and $h=(h_{ij}):V_2\to M_{t\times r}(V_3)$ be linear maps for some vector spaces $V_1,V_2,V_3, W_1$ and $W_2$. We write the ``tensor product'' and ``composition'' of those maps as follows
\begin{align*}
f\boxtimes g=&\left(\sum_{l=1}^s f_{il}\otimes g_{lj}\right): V_1\otimes W_1\to  M_{r\times t}(V_2\otimes W_2),\\
h\bullet f=&\left(\sum_{l=1}^r h_{il}f_{lj}\right):V_1\to M_{t\times s}(V_3).
\end{align*}

In particular, the vector space $\Hom_k(V,M_{r\times r}(V))$ consisting of all linear maps from $V$ to $ M_{r\times r}(V)$ with the multiplication $\bullet$ becomes an algebra. Then there is a left $\Hom_k(V,M_{r\times r}(V))$-module structure on  $M_{r\times s}(V)$, that is for any $f=(f_{ij})\in \Hom_k(V,M_{r\times r}(V))$, there is a linear map defined by
$$
\begin{array}{cclc}
f\cdot-: & M_{r\times s}(V)&\to &M_{r\times s}(V)\\
&v=(v_{ij})&\mapsto&f\cdot v=\left(\sum_{l=1}^rf_{il}(v_{lj})\right).
\end{array}
$$

Besides, we define the transpose and a kind of evaluation map for $f:V_1\to M_{r\times s}(V_2)$ as follows:
\begin{align*}
    &f^T=(f^T_{ji}):V_1\to  M_{s\times r}(V_2), \  \text{where }f^T_{ji}=f_{ij},    \\
&\underline{f}:M_{r\times s}(V_1)\to  V_2, \ (v_{ij})\mapsto  \sum_{i=1}^r\sum_{j=1}^s f_{ij}(v_{ij}).
\end{align*}
For any $X=(x_{ij})\in M_{r\times s}(k)$, it defines a natural map 
$$
\begin{array}{cclc}
    X: &V&\to &M_{r\times s}(V) \\
     ~& v& \mapsto & (x_{ij}v).
\end{array}
$$

For convenience, we write $V^{\oplus r}$ for $M_{r\times 1}(V)$. In keeping with the usual, for a linear map $\varphi:V\to V$, we denote
$$
\varphi^{\oplus r}=\begin{pmatrix}
    \varphi&&\\
    &\ddots&\\
    &&\varphi
\end{pmatrix}\cdot-:V^{\oplus r}\to V^{\oplus r}.
$$

\begin{remark}\label{rem: tensor is not unusual}
\begin{enumerate}
\item It is clear that $\boxtimes$ and $\bullet$ satisfy associative law.

\item In general, 
$$(f'\boxtimes g')\bullet(f\boxtimes g )\neq (f'\bullet f)\boxtimes (g'\bullet g),$$ 
for some linear maps $f:V_1\to M_{r\times r}(V_2),g:W_1\to M_{r\times t}(W_2)$, $f':V_2\to M_{s\times r}(V_3)$, $g':W_2\to M_{r \times r}(W_3)$. In particular, if $g'$ is a diagonal type of
$$
\begin{pmatrix}
    \phi & & \\
    & \ddots &\\
    &&\phi
\end{pmatrix}:W_2\to M_{r\times r}(W_3).
$$
Then for any linear maps $f:V_1\to M_{r\times r}(V_2),g:W_1\to M_{r\times t}(W_2)$, $f':V_2\to M_{s\times r}(V_3)$,
$$
(f'\boxtimes g')\bullet(f\boxtimes g )= (f'\bullet f)\boxtimes (g'\bullet g).
$$

\end{enumerate}
    
\end{remark}

\begin{definition}\cite[Definition 1.8]{ZZ1}\label{def: inverse}
Let $V$ be a vector space (resp. $A$ be a connected graded algebra), and $\varsigma=(\varsigma_{ij})$ and $\varphi=(\varphi_{ij})$ be two linear maps (resp. graded algebra homomorphisms) from $V$ to $M_{2\times2}(V)$ (resp. $A$ to $M_{2\times2}(A)$). If
	$$
	\varsigma^T\bullet \varphi=\varphi\bullet\varsigma^T=\begin{pmatrix}
	    \id &\\
        &\id
	\end{pmatrix},
	$$
we say $\varsigma$ is \emph{invertible} and $\varphi$ is \emph{T-invertible}, $\varphi$ is an \emph{inverse} of $\varsigma$ and $\varsigma^T$ is a \emph{T-inverse} of $\varphi$.

In this case, it is not hard to obtain that the inverse of $\varsigma$ and $T$-inverse of $\varphi$ are both unique. Write $\varphi=\varsigma^{-1}$ and $\varphi^{-T}=\varsigma^T$.
\end{definition}

\subsection{Homological determinant}\label{subsection: homological determinant}
Let $A=T(V)/(R)$ be a Koszul AS-regular algebra of dimension $d$, where $V$ is a vector space with a basis $\{x_1,x_2,\cdots,x_n\}$. It is well-known (see \cite[Proposition 3.1.4]{SZ} for example) 
$$
\dim W_i=
\left\{
\begin{array}{ll}
\dim W_{d-i}, & \text{if } 0\leq i\leq d,\\
0,            & \text{if } i>d.
\end{array}
\right.
$$
Write $\omega$ for a basis of $W_d$. 

Let $\varsigma:A\to M_{2\times 2}(A)$ be an invertible graded algebra homomorphism, which plays an important role in the discussion of graded double Ore extensions.  The homological determinant $\hdet \varsigma$ of $\varsigma$ has been introduced in \cite[Definition 2.5]{ZVZ} (also see \cite{SZL} for a general case). We recall the way used in \cite{SZL} to define $\hdet \varsigma$.

It is not hard to check that the right graded $A$-module $\left((W_i\otimes A)^{\oplus 2}\right)^{\varsigma}$ is also free, where the right graded $A$-module structure is given by
$$
\begin{pmatrix}
    \alpha_1\\
    \alpha_2
\end{pmatrix}\star a:=\left((\alpha_1,\alpha_2)\varsigma(a)\right)^T,
$$
for any $\alpha_1,\alpha_2\in W_i\otimes A$, $a\in A$ and $i=1,\cdots,d$. In fact, right graded $A$-module homomorphisms
$$
A ^{\oplus  2}
{\xleftrightharpoons[\varsigma^T\cdot-]{\varsigma^{-1}\cdot-}}\ \left( A ^{\oplus 2} \right)^\varsigma
$$
are inverse of each other.

Let $\sigma=(\sigma_{ij})=\varsigma_{\mid V}:V\to M_{2\times 2}(V)$. By Comparison Theorem, we have a commutative diagram between two free resolutions of $(k_A)^{\oplus 2}$ as follows. 
{\footnotesize
$$
\xymatrix@C=15px{
(\cdots\ar[r]&0\ar[r] &(W_d\otimes A)^{\oplus 2}\ar[r]^{(\partial^A_d)^{\oplus 2}}  \ar[d]^{\widetilde{\sigma}_d}&(W_{d-1}\otimes A)^{\oplus 2}\ar[r]^(0.7){(\partial^A_{d-1})^{\oplus 2}}\ar[d]^{\widetilde{\sigma}_{d-1}}  &\cdots\ar[r]^(0.35){\partial_{2}} &
(W_1\otimes A)^{\oplus 2}\ar[r]^(0.6){(\partial^A_1)^{\oplus 2}}\ar[d]^{\widetilde{\sigma}_1}
&A^{\oplus 2}\ar[r]\ar[d]^{\widetilde{\sigma}_0}&0\ar[r]&\cdots)\ar[r]^{(\varepsilon_A)^{\oplus 2}}&
\\
(\cdots\ar[r]&0\ar[r] & \left((W_d\otimes A)^{\oplus 2}\right)^{\varsigma}\ar[r]^{(\partial^A_d)^{\oplus 2}}&   \left((W_{d-1}\otimes A)^{\oplus 2}\right)^{\varsigma}\ar[r]^(0.75){(\partial^A_{d-1})^{\oplus 2}}&\cdots\ar[r]^(0.3){{{\partial^{\oplus 2}_{2}}}} &\left((W_1\otimes A)^{\oplus 2}\right)^{\varsigma}\ar[r]^(0.6){(\partial^A_1)^{\oplus 2}}
&\left(A^{\oplus 2}\right)^{\varsigma}\ar[r]&0\ar[r]&\cdots)\ar[r]^(0.55){(\varepsilon_A)^{\oplus 2}}&
}
\xymatrix{
\hspace{-3.5mm}(k_A)^{\oplus 2}\ar@{=}@<-3.5mm>[d]\\
\hspace{-3.5mm}(k_A)^{\oplus 2},
}
$$
}

\noindent where
$
\widetilde{\sigma}_i=(\sigma^{\boxtimes i}\boxtimes \varsigma)^T\cdot-
$ for all $i=1,\cdots,d$.  It provides us a graded algebra homomorphism $E(\widetilde{\sigma}):E(A)^{\oplus 2}\to E(A)^{\oplus 2}$. In particular,
$$
({E^d(A)})^{\oplus 2}\cong \uHom_A((W_d\otimes A)^{\oplus 2},k)\cong  \uHom_A\left(\left((W_d\otimes A)^{\oplus 2}\right)^{\varsigma},k\right)\cong (W_d^*)^{\oplus 2},
$$
as vector spaces, and we write $e_1=(\omega^*,0)$ and $e_2=(0,\omega^*)$ for a basis of $\left(E^d(A)\right)^{\oplus 2}$, where $\omega^*$ is a dual basis. Then there exists a matrix $\widetilde{H}\in M_{2\times 2}(k)$ such that
$$
E(\widetilde{\sigma})(e_1,e_2)=(e_1,e_2)\widetilde{H}.
$$
The matrix $\widetilde{H}$ is the \emph{homological determinant} of $\varsigma$, denoted by $\hdet \varsigma$.

On the other hand, there exists a matrix $H=(h_{ij})\in M_{2\times 2}(k)$ such that
$$
\sigma^{\boxtimes d}(\omega)=H\omega,
$$
and 
$$
E(\widetilde{\sigma})(e_1,e_2)=(e_1,e_2)H.
$$
In fact, for any $i=1,2$ and $c_1,c_2\in k$, one obtains that
$$
\left(E(\widetilde{\sigma})e_i\right)
\begin{pmatrix}
    \omega\otimes c_1\\
    \omega\otimes c_2
\end{pmatrix}=
e_i\widetilde{\sigma}_d
\begin{pmatrix}
    \omega\otimes c_1\\
    \omega\otimes c_2
\end{pmatrix}=
e_i
\begin{pmatrix}
    h_{11}\omega\otimes c_1+h_{21}\omega\otimes c_2\\
h_{12}\omega\otimes c_1+h_{22}\omega\otimes c_2
\end{pmatrix}=h_{1i}c_1+h_{2i}c_2=\left(h_{1i}e_1+h_{2i}e_2\right)\begin{pmatrix}
    \omega\otimes c_1\\
    \omega\otimes c_2
\end{pmatrix}.
$$
In conclusion, we obtain the following result, which is a generalization of \cite[Theorem 1.2]{MS}.
\begin{theorem}\label{thm: hdet is sigma d}
    Let $A=T(V)/(R)$ be a Koszul AS-regular algebra of dimension $d$ and $\varsigma:A\to M_{2\times 2}(A)$ be an invertible graded algebra homomorphism. Let $\sigma=\varsigma_{\mid V}$ and $\omega$ be a basis of $W_d$. Then
    $$
    \sigma^{\boxtimes d}(\omega)=(\hdet\varsigma)\omega.
    $$
\end{theorem}

\section{Graded double Ore extensions of Koszul algebras}

In this section, we discuss graded double Ore extensions of Koszul algebras. There are two goals, one is to construct a quadruple consisting of four classes of linear maps from skew derivations in graded double Ore extensions, and the other one is to construct a minimal free resolution for the trivial module of a graded double Ore extension of a Koszul algebra. We recall some basic facts about graded double Ore extensions first (see \cite{ZZ1} for more details).

\begin{definition}\cite[Definition 1.3]{ZZ1}
We say $B$ is a \emph{right double Ore extension} of an algebra $A$ if $B$ is generated by $A$ and two new variables $y_1,y_2$  satisfying
\begin{enumerate}
\item a relation $y_2y_1=p_{12}y_1y_2+p_{11}y^2_1+\kappa_1y_1+\kappa_2y_2+\kappa_0$ for some $p_{12},p_{11}\in k$ and $\kappa_0,\kappa_1,\kappa_2\in A$,

\item $B$ is a free left $A$-module with a basis $\{y_1^iy_2^j\mid i,j\geq0\}$,

\item $y_1A+y_2A\subseteq Ay_1+A y_2+A$.
\end{enumerate}
\end{definition}
Similarly, one can define a \emph{left  double Ore extension}. An algebra $B$ is called a \emph{double Ore extension} of $A$ if it is a left and a right double Ore extension of $A$ with the same generating set $\{y_1, y_2\}.$

The condition (c) of the definition of right double Ore extension implies the existence of an algebra homomorphism $\varsigma=(\varsigma_{ij}):A\to M_{2\times 2}(A)$ and a $\varsigma$-derivation $\nu=(\nu_i):A\to A^{\oplus2}$ such that 
$$
y_ia=\varsigma_{i1}(a)y_1+\varsigma_{i2}(a)y_2+\nu_i(a),
$$
for any $a\in A$ and $i=1,2$. We write $A_{P}[y_1,y_2;\varsigma,\nu,\kappa]$ for a right double Ore extension where $P=\{p_{12},p_{11}\}$ and $\kappa=\{\kappa_0,\kappa_1,\kappa_2\}$, and $A_{P}[y_1,y_2;\varsigma,\nu]$ for convenience if $\kappa=\{0,0,0\}$. In the sequel, write 
$$
\mathbb{J}=\begin{pmatrix}
-p_{11}&-p_{12}\\1&0
 \end{pmatrix}.
$$

We can define an algebra $J=k\langle Y\rangle /(r_J)$, where $r_J=Y^T\mathbb{J}Y$, 
$Y=\left(y_1,y_2\right)^T$. In fact, $A_{P}[y_1,y_2;\varsigma,\nu]$ is a twisted tensor product of $A$ and $J$. In this case, there is a surjective graded algebra homomorphism
\begin{equation}\label{eq: algebra homomorphism B to J}
 \begin{array}{cclc}
    p_J:& B &\to & J  \\
     & \sum_{i,j}a_{ij}y_1^iy_2^j&\mapsto&\sum_{i,j}\varepsilon_A(a_{ij})y_1^iy_2^j,
\end{array}   
\end{equation}
and so $J$ is a right graded $B$-module. It is clear that $\Ker p_J=A_{\geq 1} B$, so  $J\cong A_{\geq 1} B$ as right graded $B$-modules. However, there is no such natural graded algebra homomorphism from $B$ to $A$ since $\nu$ is nontrivial.

Different to the definition of homological determinant for the algebra homomorphism $\varsigma$, there is another determinant defined for $\varsigma$ in \cite{ZZ1}.

\begin{definition}\label{def: determinant of varsigma}
Let $B=A_P[y_1,y_2;\varsigma,\nu,\kappa]$ be a right double Ore extension of an algebra $A$. Then the endomorphism of $A$
$$\det\varsigma=\varsigma_{22}\varsigma_{11}-p_{12}\varsigma_{12}\varsigma_{21}-p_{11}\varsigma_{12}\varsigma_{11}$$
is called the \emph{determinant} of $\varsigma$.
\end{definition}

For a right double Ore extension $B=A_P[y_1,y_2;\varsigma,\nu]$,  $\varsigma$ and $\nu$ should satisfy some conditions (see \cite[(R3.1)--(R3.6)]{ZZ1}). We list them below.

\begin{proposition}\label{prop: conditions for varsigma and nu}
Let $B=A_P[y_1,y_2;\varsigma,\nu]$ be a right double Ore extension of an algebra $A$. Then
\begin{align*}
\varsigma_{21}\varsigma_{11}-p_{12}\varsigma_{11}\varsigma_{21}-p_{11}\varsigma_{11}\varsigma_{11}+p_{11}\det\varsigma&=0,\\
\varsigma_{21}\varsigma_{12}-p_{12}\varsigma_{11}\varsigma_{22}-p_{11}\varsigma_{11}\varsigma_{12}+p_{12}\det\varsigma&=0,\\
\varsigma_{22}\varsigma_{12}-p_{12}\varsigma_{12}\varsigma_{22}-p_{11}\varsigma_{12}\varsigma_{12}&=0,\\
\nu_2\varsigma_{11}+\varsigma_{21}\nu_1-p_{11}\left(
\nu_1\varsigma_{11}+\varsigma_{11}\nu_1
\right)-p_{12}\left(
\nu_1\varsigma_{21}+\varsigma_{11}\nu_2
\right)&=0,\\
\nu_2\varsigma_{12}+\varsigma_{22}\nu_1-p_{11}\left(
\nu_1\varsigma_{12}+\varsigma_{12}\nu_1
\right)-p_{12}\left(
\nu_1\varsigma_{22}+\varsigma_{12}\nu_2
\right)&=0,\\
\nu_2\nu_1-p_{12}\nu_1\nu_2-p_{11}\nu_1\nu_1&=0.
\end{align*}

\begin{remark} The results in Proposition \ref{prop: conditions for varsigma and nu}  can be written in the matrix form as follows.
    \begin{align}
\varsigma^T\bullet\mathbb{J}\bullet \varsigma-\mathbb{J}\bullet \begin{pmatrix}
    \det\varsigma&\\
    &\det \varsigma
\end{pmatrix}=\varsigma^T\bullet\mathbb{J}\bullet \varsigma- \begin{pmatrix}
    \det\varsigma&\\
    &\det \varsigma
\end{pmatrix}\bullet\mathbb{J}&=0,\label{eq: Theta}\\
\nu^T\bullet\mathbb{J}\bullet\varsigma+(\varsigma^T\bullet\mathbb{J}\bullet\nu)^T&=0,\label{eq: Gamma}\\
\underline{\nu}(\mathbb{J}\bullet\nu)=\nu^T\bullet \mathbb{J}\bullet\nu&=0.\label{eq: Nu}
\end{align}
\end{remark}

\end{proposition}

\begin{proposition}\label{prop: invertible means T-invertible}
Let $B=A_P[y_1,y_2;\varsigma,\nu]$ be a right double Ore extension of an algebra $A$ with $p_{12}\neq0$. If $\varsigma$  is invertible, then 
\begin{enumerate}
    \item $\det\varsigma$ is an automorphism of $A$;
    \item $\varsigma$ is $T$-invertible and
    $$
    \varsigma^{-T}=\mathbb{J}^{-1}\bullet\begin{pmatrix}
        (\det\varsigma)^{-1} & \\
        & (\det\varsigma)^{-1}
    \end{pmatrix}\bullet\varsigma^T\bullet \mathbb{J}.
    $$
\end{enumerate}
\end{proposition}
\begin{proof}
The first result is by \cite[Proposition 2.1]{ZZ1}. By  (\ref{eq: Theta}), one obtains that 
\begin{align*}
\varsigma=\mathbb{J}^{-1}\bullet\varsigma^{-1}\bullet\begin{pmatrix}
        \det\varsigma & \\
        &\det\varsigma
    \end{pmatrix}\bullet \mathbb{J}.
\end{align*}
The last result follows.
\end{proof}

Similarly, we can define the (right or left) graded double Ore extension in the categories of connected graded algebras. That is, $B=A_P[y_1,y_2;\varsigma,\nu]$ is a (right or left) double Ore extension of a connected graded algebra $A$ where $\deg y_1=\deg y_2=1$, $\varsigma:A\to M_{2\times 2}(A)$ is a graded algebra homomorphism and $\nu:A\to A^{\oplus 2}$ is a degree one $\varsigma$-derivation.

\begin{proposition}\cite[Lemma 1.9, Proposition 1.13 and Theorem 0.2]{ZZ1}.
\label{prop: invertible iff double ore extension}
Let $B=A_P[y_1,y_2;\varsigma,\nu]$ be a right graded double Ore extension of a connected graded algebra $A$ with $p_{12}\neq0$. Then 
\begin{enumerate}
    \item $B$ is a graded double Ore extension of $A$ if and only if $\varsigma$ is invertible.
\item $B$ is an AS-regular algebra of dimension $d+2$ if $A$ is an AS-regular algebra of dimension of $d$ and $\varsigma$ is invertible.
\end{enumerate}
\end{proposition}

\subsection{Graded double Ore extensions of Koszul algebras}
Let $B=A_P[y_1,y_2;\varsigma,\nu]$ be a graded double Ore extension of a Koszul algebra  $A=T(V)/(R)$, where $P=\{p_{12}\neq 0,p_{11}\}\subseteq k$, $\varsigma:A\to M_{2\times 2}(A)$ is an invertible graded algebra homomorphism and $\nu:A\to A^{\oplus 2}$ is a degree one $\varsigma$-derivation. Write $\lambda_{Y}$ for the left multiplication of $Y$ on $B$ and some quotients of $B$. 

We construct three kinds of maps for the tensor algebra $T(V)$ from $\varsigma$ and $\nu$. The first two classes is similar to ones constructed for the case of graded Ore extensions in \cite{SG}, and there is one more class of maps for graded double Ore extensions.

\noindent(I). Write $\sigma=(\sigma_{ij})=\varsigma_{\mid V}:V\to M_{2\times 2}(V)$. Then $\sigma_{T}:=\bigoplus_{i=0}^{\infty}\sigma^{\boxtimes i}$ is an invertible graded algebra homomorphism from $T(V)$ to $M_{2\times 2}(T(V))$ which induces $\varsigma$. It is not hard to obtain that for any $i\geq 0$, 
$$\sigma^{\boxtimes i}(W_i)\subseteq M_{2\times 2}(W_i).$$   

Similar to Definition \ref{def: determinant of varsigma}, we define a linear automorphism of $V^{\otimes i}$ by
\begin{align*}
   \det\sigma^{\boxtimes i}&=\sigma^{\boxtimes i}_{22}\sigma^{\boxtimes i}_{11}-p_{12}\sigma^{\boxtimes i}_{12}\sigma^{\boxtimes i}_{21}-p_{11}\sigma^{\boxtimes i}_{12}\sigma^{\boxtimes i}_{11},
\end{align*}
for any $i\geq 0$. By Proposition \ref{prop: invertible means T-invertible},  $\det\sigma=\det\varsigma_{\mid V}$ is an automorphism of the vector space $V$ and $\sigma$ is $T$-invertible with the $T$-inverse
\begin{equation}\label{eq: t-inverse of sigma}
    \sigma^{-T}=\mathbb{J}^{-1}\bullet\begin{pmatrix}
        (\det\sigma)^{-1} & \\
        & (\det\sigma)^{-1}
    \end{pmatrix}\bullet\sigma^T\bullet \mathbb{J}.
\end{equation}
The following result for $\sigma$ is similar to the condition \eqref{eq: Theta} for $\varsigma$.

\begin{lemma}\label{lem: relation between Theta^{(i)}}
For any $i\geq 0$,
\begin{enumerate}
\item  
$
    (\sigma^{\boxtimes i})^T\bullet \mathbb{J}\bullet \sigma^{\boxtimes i}+\mathbb{J}\bullet\begin{pmatrix}
    (\det\sigma)^{\otimes i}&\\
    &(\det\sigma)^{\otimes i}
\end{pmatrix}=0;$

\item $\det\sigma^{\boxtimes i}=(\det \sigma)^{\otimes i};
$

\item right graded $B$-module homomorphism $
\underline{
\sigma^{\boxtimes i}\boxtimes  \lambda_Y}\left(\mathbb{J}\bullet\left(\sigma^{\boxtimes i}\boxtimes \lambda_Y\right)\right):W_i\otimes B(-2)\to W_i\otimes B
$ is a zero map.
\end{enumerate}
\end{lemma}
\begin{proof} Write $\Theta_i=(\Theta_{i;st})=(\sigma^{\boxtimes i})^T\bullet \mathbb{J}\bullet \sigma^{\boxtimes i}:V^{\otimes i}\to M_{2\times 2}(V^{\otimes i})$, that is $\Theta_{i;st}=\sigma^{\boxtimes i}_{2s}\sigma^{\boxtimes i}_{1t}-p_{12}\sigma^{\boxtimes i}_{1s}\sigma^{\boxtimes i}_{2t}-p_{11}\sigma^{\boxtimes i}_{1s}\sigma^{\boxtimes i}_{1t},$
for all $s,t=1,2$ and $i\geq 0$. In particular, $\Theta_{i;21}=\det\sigma^{\boxtimes i}.$

Obviously, the results hold if $i=0$. (a) and (b) hold if $i=1$ by \eqref{eq: Theta}. 

For any $i\geq 2$, it is not hard to obtain that 
$
\Theta_{i;st}=\sum_{p,q=1}^2\Theta_{i-1;pq}\otimes \sigma_{ps}\sigma_{qt}$
for all $s,t=1,2$. Inductively, one obtains that
$$
\Theta_{i;st}=\Theta_{i-1;21}\otimes \Theta_{1;st}+\Theta_{i-1;22}\otimes \sigma_{2s}\sigma_{2t}=\Theta_{i-1;21}\otimes \Theta_{1;st},
$$
for all $i\geq 2$ and $s,t=1,2$. Then (a) and (b) follow.

To prove (c), we note that the two left multiplications $\lambda_{y_2}\lambda_{y_1}$ and $p_{12}\lambda_{y_1}\lambda_{y_2}+p_{11}\lambda^2_{y_1}$ on $B$ are the same. By (a), one obtains that
\begin{align*}
\underline{
\sigma^{\boxtimes i}\boxtimes  \lambda_Y}\left(\mathbb{J}\bullet\left(\sigma^{\boxtimes i}\boxtimes \lambda_Y\right)\right)&=\left(\Theta_{i;11}+p_{11}\Theta_{i;21}\right)\otimes \lambda_{y_1}^2+\left(\Theta_{i;12}+p_{12}\Theta_{i;21}\right)\otimes \lambda_{y_1} \lambda_{y_2}+\Theta_{i;22}\otimes \lambda_{y_2}^2=0.\qedhere
\end{align*}
\end{proof}

The lemma above implies that $\det\varsigma$ can be induced by $\bigoplus_{i=0}^{\infty}(\det \sigma)^{\otimes i}$.

\medskip
\noindent(II). Choose a linear map $\delta=(\delta_{i}):V\to (V\otimes V)^{\oplus 2}$ such that $\pi^{\oplus 2}_A\delta=\nu_{\mid V}$. One can extend $\delta$ to a degree one $\sigma_T$-derivation from $T(V)$ to $(T(V)\otimes T(V))^{\oplus 2}$ satisfying
\begin{equation}\label{condition for delta}
\delta(R)\subseteq (R\otimes V+V\otimes R)^{\oplus 2},
\end{equation}
also denoted by $\delta$. Clearly, $\nu$ can be induced by $\delta$. By the condition (\ref{condition for delta}), there exist two linear maps
$$\delta_{2,r}=(\delta_{2,r;j}): R\to (R\otimes V)^{\oplus 2},\quad\delta_{2,l}=(\delta_{2,l;j}): R\to (
V\otimes R)^{\oplus 2},
$$ such that
\begin{equation*}
\delta_{\mid R}=\delta_{2,r}+\delta_{2,l}.
\end{equation*}

\medskip

\noindent(III). By \eqref{eq: Nu}, one obtains that
$
\underline{\delta}(\mathbb{J}\bullet \delta)=\left(\delta_2\delta_1-p_{12}\delta_1\delta_2-p_{11}\delta_1\delta_1\right)(V)\subseteq R\otimes V+V\otimes R.
$
We define two linear maps
$$\upsilon_{1,r}=(\upsilon_{1,r;j}): V\to R\otimes V,\quad\upsilon_{1,l}=(\upsilon_{1,l;j}): V\to 
V\otimes R,
$$ such that
\begin{equation*}
\underline{\delta}(\mathbb{J}\bullet \delta)_{\mid V}=\upsilon_{1,r}+\upsilon_{1,l}.
\end{equation*}
For convenience, we write $\delta_{0,r}=\delta_{0,l}=\upsilon_{0,r}=\upsilon_{0,l}=0$, and $\delta_{1,r}=(\delta_{1,r;j})={\delta}_{|V}=(\delta_{1,l;j})=\delta_{1,l}$.

It should be noted that pairs $(\delta_{2,r},\delta_{2,l})$ and $(\upsilon_{1,r},\upsilon_{1,l})$ for $\nu$ are not unique.

\subsection{Minimal free resolution}

Now we begin to construct a minimal free resolution of $k_B$. 

Applying $-\otimes_AB$ to the Koszul complex of $ k_{A}$, one obtains a free resolution $Q_\centerdot$ of $B/A_{\geq1}B$ as right graded $B$-modules,
\begin{equation*}
Q_\centerdot=(\cdots\xlongrightarrow{} W_d\otimes B\xlongrightarrow{\partial_d}W_{d-1}\otimes B\xlongrightarrow{\partial_{d-1}}\cdots\xlongrightarrow{\partial_{3}}W_2\otimes B\xlongrightarrow{\partial_{2}}W_1\otimes B\xlongrightarrow{\partial_{1}}B\to0\to\cdots)\xlongrightarrow{\varepsilon_A\otimes_AB}B/A_{\geq1}B,
\end{equation*}
where $\partial_{i}=\partial^A_{i}\otimes_AB=\id_V^{\otimes i-1}\otimes m_B$ and $m_B$ is the multiplication of $B$ for $i\geq1$.

\begin{lemma}\label{lemma: construction of delta_r and delta_l}
There exists a pair $\left(\{\delta_{i,r}\}_{i\geq 0}, \{\delta_{i,l}\}_{i\geq 0}\right)$ of linear maps for $\nu$, where  $$\delta_{i,r}=(\delta_{i,r;j}):W_i\to( W_i\otimes V)^{\oplus
2}\quad\text{and}\quad\delta_{i,l}=(\delta_{i,l;j}):W_i\to( V\otimes W_i)^{\oplus
2}$$ 
for all $i\geq 0$, satisfying the following two conditions:
\begin{enumerate}
    \item the following diagram is commutative 
{\footnotesize
$$
~\hspace{-3mm}\xymatrix{
Q_\centerdot(-2)=\ar[d]^{f_\centerdot}
\\
Q_\centerdot(-1)^{\oplus2}=
}
\xymatrix@C=18px{
(\cdots\ar[r] &W_d\otimes B(-2) \ar[r]^{\partial_d}\ar[d]^{f_d}&\cdots\ar[r]^(0.35){\partial_{2}} &W_1\otimes B(-2)\ar[r]^(0.55){\partial_{1}}\ar[d]^{f_1}&B(-2)\ar[r]\ar[d]^{f_0}
\ar[r]&0\ar[d]^0\ar[r]&\cdots)\ar[r]^{\varepsilon_A\otimes_AB}&
\\
(\cdots\ar[r] & (W_d\otimes B(-1) )^{\oplus 2}\ar[r]^(0.6){{\partial^{\oplus 2}_d}}&\cdots\ar[r]^(0.3){{{\partial^{\oplus 2}_{2}}}} &(W_1\otimes B(-1))^{\oplus 2}\ar[r]^(0.55){{\partial^{\oplus 2}_{1}}} &B(-1)^{\oplus 2}\ar[r]
&0\ar[r]&\cdots)\ar[r]^(0.65){(\varepsilon_A\otimes_AB)^{\oplus2}}&
}
\hspace{-3mm}\xymatrix{
B/A_{\geq1}B(-2)\ar[d]^{\mathbb{J}\bullet\lambda_Y}
\\
(B/A_{\geq1}(-1))^{\oplus 2},
}
$$
}
where each right graded $B$-module homomorphism
\begin{align*}
f_i
=\mathbb{J}\bullet\left(\sigma^{\boxtimes i}\boxtimes \lambda_Y
+({\id}_V^{\otimes i}\otimes m_B)^{\oplus2}(\delta_{i,r}\boxtimes \id_B)
\right);
\end{align*} 
    \item for any $i\geq 1$,
    $$\delta_{i,r}+(-1)^i\delta_{i,l}=\sigma\boxtimes\delta_{i-1,r}+(-1)^i\delta_{i-1,l}\boxtimes {\id}_V.$$
\end{enumerate}
Moreover, as linear maps from $W_i$ to  $(W_{i-1}\otimes B)^{\oplus2}$,
\begin{equation}\label{eq: induction relations on delta_i,r}
(\id_V^{\otimes i-1} \otimes m_B)^{\oplus2}\delta_{i,r}=(\id_V^{\otimes i-1}\otimes m_B)^{\oplus2}\left(\sigma^{\boxtimes i-1}\boxtimes \delta+\delta_{i-1,r} \boxtimes {\id}_V\right),\quad i\geq 2.
\end{equation}
\end{lemma}
\begin{proof}
(a) Clearly, $
(\mathbb{J}\bullet\lambda_Y)(\varepsilon_A\otimes_A B)=(\varepsilon_A\otimes_A B)^{\oplus 2}f_0
$. It is not hard to check that 
$f_0\partial_1=\partial_1^{\oplus2}f_1$, since
\begin{align*}
f_0\partial_1&= 
(\mathbb{J}\bullet\lambda_Y)m_B
=m_B^{\oplus2}\left(\mathbb{J}\bullet\left(\sigma\boxtimes\lambda_Y+(\id_V\otimes m_B)^{\oplus2}(\delta\boxtimes{\id}_B)\right)\right)=
\partial_1^{\oplus2}f_1.
\end{align*}
Note that $(\id_V\otimes m_B)^{\oplus 2}\delta_{2,r}=(\id_V\otimes m_B)^{\oplus 2}(\delta_{2,r}+\delta_{2,l})=(\id_V\otimes m_B)^{\oplus 2}(\sigma\boxtimes \delta+\delta\boxtimes\id_V)$ as linear maps from $W_2=R$ to $(W_1\otimes B)^{\oplus 2}$, we have
\begin{align*}
\partial_2^{\oplus2}f_2&=(\id_V\otimes m_B)^{\oplus2}\left(\mathbb{J}\bullet
\left(\sigma^{\boxtimes 2}\boxtimes\lambda_Y+({\id}^{\otimes 2}_V\otimes m_B)^{\oplus2}\left(\delta_{2,r}\boxtimes{\id}_B\right)\right)\right)\\
&=\mathbb{J}\bullet\left((\id_V\otimes m_B)^{\oplus2}
\left(\sigma^{\boxtimes 2}\boxtimes\lambda_Y+\sigma\boxtimes \delta\boxtimes{\id}_B+\delta\boxtimes \id_V\boxtimes{\id}_B\right)\right)\\
&=\left(\mathbb{J}\bullet\left(\sigma\boxtimes\lambda_Y+(\id_V\otimes m_B)^{\oplus 2}(\delta\boxtimes {\id}_B)\right)\right)(\id_V\otimes m_B)=f_1\partial_2.
\end{align*}

It turns to construct linear maps $\delta_{i,r}$ for all $i\geq 3$. By the Comparison Theorem, there exist right graded $B$-module homomorphisms $f_i: W_i\otimes B(-2)\to (W_i\otimes B(-1))^{\oplus2}$ such that $f_{i-1}\partial_{i}=\partial_{i}^{\oplus2}f_{i}$  for all $i\geq 3$. One obtains that 
$\left(f_i-\mathbb{J}\bullet\left(\sigma^{\boxtimes i}\boxtimes\lambda_Y\right)
\right)(W_i\otimes B_0)\subseteq(W_i\otimes V)^{\oplus2}$, and so we have linear maps
$$
\delta_{i,r}=(\delta_{i,r;j}):=\left(\mathbb{J}^{-1}\bullet f_i-\left(\sigma^{\boxtimes i}\boxtimes\lambda_Y\right)
\right){\mid_{W_i}}:W_i\to (W_i\otimes V)^{\oplus 2},
$$
and the right graded $B$-module homomorphisms 
$$
f_i=\mathbb{J}\bullet\left(\sigma^{\boxtimes i}\boxtimes\lambda_Y+({\id}_V^{\otimes i}\otimes m_B)^{\oplus2}(\delta_{i,r}\boxtimes \id_B)\right),
$$
for all $i\geq 3$. Since $\partial_i^{\oplus2}f_i=f_{i-1}\partial_i$, that is
\begin{align*}\partial_i^{\oplus2}f_i&=(\id_V^{\otimes i-1}\otimes m_B)^{\oplus2}\left(\mathbb{J}\bullet
\left(\sigma^{\boxtimes i}\boxtimes\lambda_Y+({\id}_V^{\otimes i}\otimes m_B)^{\oplus2}(\delta_{i,r}\boxtimes \id_B)\right)\right),
\\
f_{i-1}\partial_i&=
\left(\mathbb{J}\bullet\left(\sigma^{\boxtimes i-1}\boxtimes\lambda_Y+({\id}_V^{\otimes i}\otimes m_B)^{\oplus2}(\delta_{i-1,r}\boxtimes \id_B)\right)\right)({\id}_V^{\otimes i-1}\otimes m_B)\\
&=({\id}_V^{\otimes i-1}\otimes m_B)^{\oplus2}\left(\mathbb{J}\bullet\left(\sigma^{\boxtimes i}\boxtimes \lambda_Y+\sigma^{\boxtimes i-1}\boxtimes\delta\boxtimes\id_B+\delta_{i-1,r}\boxtimes{\id}_V\boxtimes{\id}_B
\right)\right).
\end{align*}
Hence, 
$$
    (\id_V^{\otimes i-1} \otimes m_B)^{\oplus2}\delta_{i,r}=(\id_V^{\otimes i-1}\otimes m_B)^{\oplus2}\left(\sigma^{\boxtimes i-1}\boxtimes \delta+\delta_{i-1,r} \boxtimes {\id}_V\right),
$$
as linear maps $W_i\to (W_i\otimes B)^{\oplus 2}$ for all $i\geq 2$.

(b) By (a), one obtains that the following diagram is commutative (without dotted arrows) .
\begin{scriptsize}
$$
\xymatrix{
\hspace{-1mm}\cdots\ar[r]&W_{d+1}\otimes B(-2)\ar[r]^{\partial_{d+1}}\ar[d]^{\rho_{d+1}}\ar@{-->}[ld]_{s_{d+1}}  &\cdots\ar[r]^(0.3){} &W_2\otimes B(-2)\ar[r]^(0.5){{\partial_{2}}} \ar[d]^{\rho_2}\ar@{-->}[ld]_{s_2}&W_1\otimes B(-2)\ar[r]^(0.55){{{\partial_1}}}\ar[d]^{\rho_1}\ar@{-->}[ld]_{s_1}&{\rm Im}\partial_1(-2)\ar[r]\ar[d]^{0}&0\\
\hspace{-1mm}\cdots\ar[r]&(V\otimes W_d\otimes B(-1))^{\oplus2}\ar[r]^(0.7){-({\id}_V\otimes\partial_d)^{\oplus2}}&\cdots
\ar[r] &(V\otimes W_1\otimes B(-1))^{\oplus2}\ar[r]^(0.55){{-({\id}_V\otimes \partial_{1})^{\oplus2}}} &(V\otimes B(-1))^{\oplus2}\ar[r]^(0.47){{-({\id}_V\otimes\varepsilon_A\otimes B)^{\oplus2}}}&(V\otimes B/A_{\geq1}B)^{\oplus2}(-1)\ar[r]&0,
}
$$
\end{scriptsize}
where the top row is a truncation of $Q_\centerdot(-2)$, the bottom row is the complex  $\left(V^{\oplus 2}\otimes Q_\centerdot(-1)\right)[-1]$ and $$\rho_i:=(-1)^i({\id}_V^{\otimes i}\otimes m_B)^{\oplus2}\left(\delta_{i,r}-\sigma\boxtimes\delta_{i-1,r}\right)\boxtimes {\id}_B,$$
for any $i\geq 1$. So the chain homomorphism $\{\rho_i\}$ is null homotopic, and  there exist right graded $B$-module homomorphisms $s_i:W_i\otimes B(-2)\to (V\otimes W_i\otimes B(-1))^{\oplus 2}$ such that $\rho_i=s_{i-1}\partial_i-(\id_V\otimes \partial_i)^{\oplus 2}s_i$ for all $i\geq 1$ where $s_0=0.$ By an argument on the grading, one obtains linear maps
$$
\delta_{i,l}=(\delta_{i,l;j})={s_i}_{\mid W_i}: W_i\to (V\otimes W_i)^{\oplus 2}
$$
for all $i\geq 0$. Then we can write $s_i=\delta_{i,l}\boxtimes\id_B$. Since
\begin{align*}
(-1)^i({\id}_V^{\otimes i}\otimes m_B)^{\oplus2}\left(\delta_{i,r}-\sigma\boxtimes\delta_{i-1,r}\right)\boxtimes {\id}_B=&
\left(\delta_{i-1,l}\boxtimes\id_B\right)\left(\id_V^{\otimes i-1}\otimes m_B\right)
-(\id_V^{\otimes i}\otimes m_B)^{\oplus 2}\left(\delta_{i-1,l}\boxtimes\id_B\right) \\
=&(\id_V^{\otimes i}\otimes m_B)^{\oplus 2}\left(\delta_{i-1,l}\boxtimes\id_V\boxtimes\id_B-\delta_{i,l}\boxtimes\id_B\right)
\end{align*}
as right graded $B$-module homomorphisms from $W_i\otimes B(-2)$ to $\left(V\otimes W_i\otimes B(-1)\right)^{\oplus 2}$, we have
$$
\delta_{i,r}+(-1)^i\delta_{i,l}=\sigma\boxtimes \delta_{i-1,r}+(-1)^i\delta_{i-1,l}\boxtimes \id_V,
$$
by restricted on $W_i$ for all $i\geq 1$.
\end{proof}

Let $\left(\{\delta_{i,r}\}, \{\delta_{i,l}\}\right)$ be a pair of linear maps for $\nu$ constructed as in Lemma \ref{lemma: construction of delta_r and delta_l}. Following \eqref{eq: Gamma} and \eqref{eq: Nu}, we define two classes of linear maps:
\begin{align*}
\Gamma_{i,\star}&=(\Gamma_{i,\star;s})=\left(\left(\sigma^{\boxtimes i+1}\right)^T\bullet\mathbb{J}\bullet\delta_{i,\star}\right)^T
+
\delta^T_{i,\star}\bullet \mathbb{J}\bullet
\sigma^{\boxtimes i}:W_i\to M_{1\times2}(V^{\otimes i+1}),\\
\Delta_{i}&=\underline{\sigma^{\boxtimes i}\boxtimes \delta+\delta_{i,r}\boxtimes\id_V}(\mathbb{J}\bullet\delta_{i,r}):W_i\to W_i\otimes V^{\otimes 2},
\end{align*}
for all $i\geq 0$ and $\star=$``$r$'' or  ``$l$''. To be precise, 
\begin{align*}
&\Gamma_{i,\star;s}=\left(\delta_{i,\star;2}\sigma^{\boxtimes i}_{1s}+\sigma^{\boxtimes i+1}_{2s}\delta_{i,\star;1}\right)-p_{12}\left(\delta_{i,\star;1}\sigma^{\boxtimes i}_{2s}+\sigma^{\boxtimes i+1}_{1s}\delta_{i,\star;2}\right)-p_{11}\left(\delta_{i,\star;1}\sigma^{\boxtimes i}_{1s}+\sigma^{\boxtimes i+1}_{1s}\delta_{i,\star;1}\right),\\
&\Delta_i=\left(\sum_{u=1}^2 \sigma^{\boxtimes i}_{2u}\otimes \delta_{u}+\delta_{i,r;2}\otimes \id_V\right)\delta_{i,r;1}-\sum_{v=1}^2p_{1v}\left(\sum_{u=1}^2 \sigma^{\boxtimes i}_{1u}\otimes \delta_{u}+\delta_{i,r;1}\otimes \id_V\right)\delta_{i,r;v},
\end{align*}
for all $s=1,2$, $i\geq 0$ and $\star=$``$r$'' or  ``$l$''.

By  \eqref{eq: Gamma} and the definitions of $\delta_{i,r}$ and $\delta_{i,l}$ for all $i\geq 0$, we have the following result immediately.
\begin{lemma}\label{lem: prop for Gamma}
\begin{enumerate}
    \item $\Image \Gamma_{1,r}=\Image \Gamma_{1,l}\subseteq M_{1\times2}(R)$.
    \item  $\Image \Gamma_{i,r}\subseteq M_{1\times2}(W_i\otimes V)$ and  $\Image \Gamma_{i,l}\subseteq M_{1\times2}(V\otimes W_i)$ for any $i\geq 2$.
\end{enumerate}

\end{lemma}

\begin{lemma}\label{lem: realtions between Gamma}
For any $i\geq 1$,
\begin{align}
    \Gamma_{i,r}+(-1)^i\Gamma_{i,l}&=\det\sigma\boxtimes \Gamma_{i-1,r}+(-1)^i\Gamma_{i-1,l}\boxtimes \sigma\label{eq: relations between Gamma r and l}\\
    &=\sum_{u=1}^{i-1}(-1)^{u+1}(\det \sigma)^{\boxtimes i-u-1}\boxtimes \left(\Gamma_{u,l}\boxtimes \sigma+\det \sigma \boxtimes \Gamma_{u,l}\right).\label{eq: Gamma r and l to l}
\end{align}
Moreover, for all $i\geq 1$,
\begin{enumerate}
    \item $\sum_{u=1}^i(-1)^u\Gamma_{u,r}\boxtimes\sigma^{\boxtimes i-u}=\sum_{u=1}^{i}(-1)^{i+u+1}(\det \sigma)^{\boxtimes i-u}\boxtimes \Gamma_{u,l}$;

    \item $\left(\sum_{u=1}^i(-1)^u\Gamma_{u,r}\boxtimes\sigma^{\boxtimes i-u}\right)(W_i)=\left(\sum_{u=1}^{i}(-1)^{u+1}(\det \sigma)^{\boxtimes i-u}\boxtimes \Gamma_{u,l}\right)(W_i)\subseteq M_{1\times 2}(W_{i+1})$;

    \item If $W_{d+1}=0$,  
    \begin{align*}
             & \sum_{u=1}^d(-1)^u
     \left( \left(\sigma^{\boxtimes d+1}\right)^T\bullet\mathbb{J}\bullet(\delta_{u,r}\boxtimes\id^{\boxtimes d-u})\right)^T
      =
        \sum_{u=1}^d(-1)^{u+1}(
      \delta_{u,r}\boxtimes \id_V^{\boxtimes d-u})^T\bullet \mathbb{J}\bullet
\sigma^{\boxtimes d}.
    \end{align*}
\end{enumerate}
\end{lemma}
\begin{proof}
Write $\Xi^{(i)}_{\star;st;j}=\delta_{i,\star;s}\sigma^{\boxtimes i}_{tj}+\sigma^{\boxtimes i+1}_{sj}\delta_{i,\star;t}$, and then $\Gamma_{i,\star;j}=\Xi^{(i)}_{\star;21;j}-p_{11}\Xi^{(i)}_{\star;11;j}-p_{12}\Xi^{(i)}_{\star;12;j}$ for all $i\geq 0$, $s,t,j=1,2$ and $\star=$``$r$'' or  ``$l$''.

By Lemma \ref{lemma: construction of delta_r and delta_l}, one obtains that
\begin{equation}\label{eq: relations between Xi}
    \Xi^{(i)}_{r;st;j}+(-1)^i\Xi^{(i)}_{l,st;j}=\sum_{u,v=1}^2 \sigma_{su}\sigma_{tv}\otimes \Xi^{(i-1)}_{r;uv;j}+(-1)^{i}\sum_{u=1}^2\Xi^{(i-1)}_{l;st;u}\otimes \sigma_{uj},
\end{equation}
for any $i\geq 1$ and $s,t,j=1,2$. Then
\begin{align*}
&\Gamma_{i,r;j}+(-1)^i\Gamma_{i,l;j}\\
=&\sum_{u,v=1}^2 \sigma_{2u}\sigma_{1v}\otimes \Xi^{(i-1)}_{r;uv;j}+(-1)^{i}\sum_{u=1}^2\Xi^{(i-1)}_{l;21;u}\otimes \sigma_{uj}-p_{11}\left(
\sum_{u,v=1}^2 \sigma_{1u}\sigma_{1v}\otimes \Xi^{(i-1)}_{r;uv;j}+(-1)^{i}\sum_{u=1}^2\Xi^{(i-1)}_{l;11;u}\otimes \sigma_{uj}
\right)\\
&-p_{12}\left(
\sum_{u,v=1}^2 \sigma_{1u}\sigma_{2v}\otimes \Xi^{(i-1)}_{r;uv;j}
+(-1)^{i}\sum_{u=1}^2\Xi^{(i-1)}_{l;12;u}\otimes \sigma_{uj}
\right)\\
=&\sum_{u,v=1}^2\det\sigma\otimes \Xi^{(i-1)}_{r;uv;j}+(-1)^{i}\sum_{u=1}^2\left(\Xi^{(i-1)}_{l;21;u}-p_{11}\Xi^{(i-1)}_{l;11;u}-p_{12}\Xi^{(i-1)}_{l;12;u}\right)\otimes \sigma_{uj}\\
=&\det\sigma\otimes \Gamma_{i-1,r;j}+(-1)^{i}\sum_{u=1}^2\Gamma_{i-1,l;u}\otimes \sigma_{uj},
\end{align*}
where the first equation is by \eqref{eq: relations between Xi} and the last equation is by Lemma \ref{lem: relation between Theta^{(i)}} for all $i\geq 1$ and $j=1,2$.

Another equality holds when $i=1$ since both sides are zero, and we show it also holds when $i>1$ inductively.
\begin{align*}
\Gamma_{i,r}+(-1)^i\Gamma_{i,l}=&\det\sigma\boxtimes \Gamma_{i-1,r}+(-1)^i\Gamma_{i-1,l}\boxtimes \sigma\\
=&\det\sigma\boxtimes \left(
(-1)^{i}\Gamma_{i-1,l}+\sum_{u=1}^{i-2}(-1)^{u+1}(\det \sigma)^{\boxtimes i-u-2}\boxtimes \left(\Gamma_{u,l}\boxtimes \sigma+\det \sigma \boxtimes \Gamma_{u,l}\right)
\right)+(-1)^i\Gamma_{i-1,l}\boxtimes \sigma\\
=&\sum_{u=1}^{i-2}(-1)^{u+1}(\det \sigma)^{\boxtimes i-u-1}\boxtimes \left(\Gamma_{u,l}\boxtimes \sigma+\det \sigma \boxtimes \Gamma_{u,l}
\right)+(-1)^{i}\left(\Gamma_{i-1,l}\boxtimes \sigma+\det\sigma\boxtimes
\Gamma_{i-1,l}\right)\\
=&\sum_{u=1}^{i-1}(-1)^{u+1}(\det \sigma)^{\boxtimes i-u-1}\boxtimes \left(\Gamma_{u,l}\boxtimes \sigma+\det \sigma \boxtimes \Gamma_{u,l}
\right).
\end{align*}

By \eqref{eq: relations between Gamma r and l}, one obtains that
\begin{align*}
\sum_{u=1}^i (-1)^u\Gamma_{u,r}\boxtimes\sigma^{\boxtimes i-u}-\sum_{u=1}^i(-1)^u\det\sigma\boxtimes\Gamma_{u-1,r}\boxtimes\sigma^{\boxtimes i-u}=-\Gamma_{i,l}.
\end{align*}
Then (a) holds by induction on $i$.

Write $\Upsilon_{i}$ for $\sum_{u=1}^{i}(-1)^{u+1}(\det \sigma)^{\boxtimes i-u}\boxtimes \Gamma_{u,l}$. It is clear that $\Upsilon_{1}(W_1)=\Gamma_{1,l}(W_1)\subseteq W^{\oplus 2}_2$. Suppose $\Upsilon_{i-1}(W_u)\subseteq W^{\oplus2}_{u+1}$ for some $i\geq 2$. Since
$
\Upsilon_{i}=\det\sigma\otimes \Upsilon_{i-1}+(-1)^{i+1}\Gamma_{i,l}
$, $\Upsilon_{i}(W_i)\subseteq (V\otimes W_i)^{\oplus 2}.$

On the other hand, on obtains that $\Upsilon_{i}=\Gamma_{i,r}-(-1)^{u}\Upsilon_{i-1}\otimes \sigma
$ by rewriting \eqref{eq: Gamma r and l to l}. So 
$\Upsilon_{i}(W_i)\subseteq (W_i\otimes V)^{\oplus 2}.$ 
Then (b) follows by (a).

(c) follows by (b) and the definition of $\Gamma_{u,r}$.
\end{proof}

\begin{lemma}\label{lem: construction of upsilon} Let $\left(\{\delta_{i,r}\}, \{\delta_{i,l}\}\right)$ be a pair of linear maps for $\nu$ constructed as in Lemma \ref{lemma: construction of delta_r and delta_l}. Then there exist linear maps $$\upsilon_{i,r}:W_i\to W_{i+1}\otimes V$$ 
for all $i\geq 1$ such that the following diagram  commutes (without dotted arrows) and 
$
g_if_i=h_{i-1}\partial_i+\partial_{i+1}h_i
$ for all $i\geq 0:$
{\footnotesize
\begin{equation*}
\hspace{-3.7mm}\xymatrix{
Q_\centerdot(-2)=\ar[d]^{f_\centerdot}\\
Q_\centerdot(-1)^{\oplus2}=\ar[d]^{g_\centerdot}\\
Q_\centerdot=
}
\hspace{-2.6mm}\xymatrix@C=20px{
(\cdots\ar[r] &W_d\otimes B(-2) \ar[r]^{\partial_d}\ar[d]^{f_d}\ar@{-->}[ldd]^(0.6){h_{d}} &\cdots\ar[r]^(0.35){\partial_{2}} &W_1\otimes B(-2)\ar[r]^(0.55){\partial_{1}}\ar[d]^{f_1}\ar@{-->}[ldd]^(0.6){h_{1}} &B(-2)\ar@{-->}[ldd]^(0.6){h_0=0}\ar[d]^{f_0}\ar[r]&0\ar@{-->}[ldd]^(0.6){h_{-1}=0}\ar[r]\ar[d]^0&\cdots)\ar[r]^{\varepsilon_A\otimes_AB}&
\\
(\cdots\ar[r] & (W_d\otimes B(-1) )^{\oplus 2}\ar[r]^(0.65){{\partial^{\oplus 2}_d}}\ar[d]^{g_d} &\cdots\ar[r]^(0.3){{{\partial^{\oplus 2}_{2}}}} &(W_1\otimes B(-1))^{\oplus 2}\ar[r]^(0.5){{\partial^{\oplus 2}_{1}}} \ar[d]^{g_1}&B(-1)^{\oplus 2}\ar[d]^{g_0}\ar[r]&0\ar[r]\ar[d]^0&\cdots)\ar[r]^{(\varepsilon_A\otimes_AB)^{\oplus 2}}&
\\
(\cdots\ar[r] & W_d\otimes B\ar[r]^{\partial_d} &\cdots\ar[r]^(0.41){{\partial_{2}}} &W_1\otimes B\ar[r]^(0.6){{\partial_{1}}} &B\ar[r]&0\ar[r]&\cdots)\ar[r]^{\varepsilon_A\otimes_AB}&
}
\xymatrix{
\hspace{-3.7mm}B/A_{\geq1}B(-2)\ar[d]^{\mathbb{J}\bullet\lambda_Y}\\
\hspace{-3.7mm}(B/A_{\geq1}B(-1))^{\oplus2}\ar[d]^{\underline{\lambda_Y}}\\
\hspace{-3.7mm}B/A_{\geq1}B,
}
\end{equation*}
}
where for any $i\geq 0$, 
right graded $B$-module homomorphisms
\begin{align*}
g_i
&=  \underline{\sigma^{\boxtimes i}\boxtimes\lambda_Y}+({\id}_V^{\otimes i}\otimes m_B)
\underline{\delta_{i,r}\boxtimes \id_B},
\\
h_i&=\sum_{u=1}^i(-1)^{i+u}\Gamma_{u,r}\boxtimes\sigma^{\boxtimes i-u}\boxtimes\lambda_{Y}+({\id}_V^{\otimes i+1}\otimes m_B)(\upsilon_{i,r}\otimes {\id}_B)
\\
&=\sum_{u=1}^i(-1)^{u+1}({\det}\sigma)^{\boxtimes i-u}\boxtimes\Gamma_{u,l}\boxtimes\lambda_{Y}+({\id}_V^{\otimes i+1}\otimes m_B)(\upsilon_{i,r}\otimes {\id}_B).
\end{align*}

Moreover, 
 as graded linear maps from $W_i$ to $W_{i}\otimes B$ for any $i\geq 2$,
\begin{align*}
(\id_V^{\otimes i}\otimes m_B)(\upsilon_{i,r}+\upsilon_{i-1,r}\otimes\id_V)=(\id_V^{\otimes i}\otimes m_B)\left(\Delta_{i}+\sum_{u=1}^{i-1}(-1)^{u}({\det}\sigma)^{\boxtimes i-1-u}\boxtimes\Gamma_{u,l}\boxtimes\delta \right).
\end{align*}
\end{lemma}
\begin{proof}
By Lemma \ref{lemma: construction of delta_r and delta_l}, the top two rows are commutative. It is not hard to obtain $\underline{\lambda_Y}(\varepsilon_A\otimes_A B)^{\oplus 2}=(\varepsilon_A\otimes_A B)g_0$. For all $i\geq 1$,
\begin{align*} 
\partial_ig_i&=({\id}_V^{\otimes i-1}\otimes m_B)\left(\underline{\sigma^{\boxtimes i}\boxtimes\lambda_Y}+({\id}_V^{\otimes i}\otimes m_B)\left(\underline{\delta_{i,r}\boxtimes{\id}_B}\right)\right)\\
&=({\id}_V^{\otimes i-1}\otimes m_B)\left(\underline{\sigma^{\boxtimes i}\boxtimes\lambda_Y}+\underline{\delta_{i,r}\boxtimes{\id}_B}\right)\\
&=({\id}_V^{\otimes i-1}\otimes m_B)\left(\underline{\sigma^{\boxtimes i}\boxtimes\lambda_Y}+\left(\underline{\sigma^{\boxtimes i-1}\boxtimes \delta+\delta_{i-1,r} \boxtimes {\id}_V}\right)\otimes{\id}_B\right)\\
&=\left(\underline{\sigma^{\boxtimes i-1}\boxtimes\lambda_Y}+(\id^{\otimes i-1}_V\otimes m_B)(\underline{\delta_{i-1,r} \boxtimes {\id}_V})\right)({\id}_V^{\otimes i-1}\otimes m_B)^{\oplus 2}\\
&=g_{i-1}\partial_i^{\oplus2},
\end{align*}
where the third equality is by \eqref{eq: induction relations on delta_i,r}. Hence, the bottom two rows commute.

Since $\underline{\lambda_Y}(\mathbb{J}\bullet\lambda_Y)$ is a zero map from $B/A_{\geq 1}B(-2)$ to $B/A_{\geq 1}B$, $\{g_if_i\}$ is null homotopic and there exist right graded $B$-module homomorphisms $h_i:W_i\otimes B(-2)\to W_{i+1}\otimes B$ such that $g_if_i=h_{i-1}\partial_i+\partial_{i+1}h_i$ for all $i\geq 0$.

Firstly, we have
\begin{align*}
g_if_i=&\left(
 \underline{\sigma^{\boxtimes i}\boxtimes\lambda_Y}+({\id}_V^{\otimes i}\otimes m_B)
\underline{\delta_{i,r}\boxtimes \id_B}
\right)\left(
\mathbb{J}\bullet\left(\sigma^{\boxtimes i}\boxtimes\lambda_Y+({\id}_V^{\otimes i}\otimes m_B)^{\oplus2}(\delta_{i,r}\boxtimes \id_B)\right)
\right)\\
=&
\underline{\sigma^{\boxtimes i}\boxtimes\lambda_Y}
 \left(\mathbb{J}\bullet
 (
 \sigma^{\boxtimes i}\boxtimes\lambda_Y
 )\right)
 +({\id}_V^{\otimes i}\otimes m_B)
\underline{\delta_{i,r}\boxtimes \id_B}\left(\mathbb{J}\bullet
({\id}_V^{\otimes i}\otimes m_B)^{\oplus2}(\delta_{i,r}\boxtimes \id_B)\right)\\
&
+
 \underline{\sigma^{\boxtimes i}\boxtimes\lambda_Y}\left(\mathbb{J}\bullet({\id}_V^{\otimes i}\otimes m_B)^{\oplus2}(\delta_{i,r}\boxtimes \id_B)\right)
+
({\id}_V^{\otimes i}\otimes m_B)
\underline{\delta_{i,r}\boxtimes \id_B}\left(\mathbb{J}\bullet
(\sigma^{\boxtimes i}\boxtimes\lambda_Y)
\right)\\
=&
 ({\id}_V^{\otimes i}\otimes m_B)
\left(\left(
\underline{\delta_{i,r}\boxtimes \id_V}\left(\mathbb{J}\bullet
\delta_{i,r}\right)\right)\otimes \id_B)
\right)
+
({\id}_V^{\otimes i}\otimes m_B)
\underline{\delta_{i,r}\boxtimes \id_B}\left(\mathbb{J}\bullet
(\sigma^{\boxtimes i}\boxtimes\lambda_Y)
\right)\\
&
+
 ({\id}_V^{\otimes i}\otimes m_B)\underline{\sigma^{\boxtimes i+1}\boxtimes\lambda_Y}\left(\mathbb{J}\bullet(\delta_{i,r}\boxtimes \id_B)\right)
+
 ({\id}_V^{\otimes i}\otimes m_B)\left(\left(\underline{\sigma^{\boxtimes i}\boxtimes\delta}\left(\mathbb{J}\bullet\delta_{i,r}\right)\right)\otimes \id_B\right)
\\
=&(\id_V^{\otimes i}\otimes m_B)\left(\Gamma_{i,r}\boxtimes \lambda_{Y}+\Delta_{i}\otimes \id_B\right),
\end{align*}
where the third equality holds by Lemma \ref{lem: relation between Theta^{(i)}}. Then we construct $\{\upsilon_{i,r}\}_{i\geq 1}$ inductively.

\begin{enumerate}
    \item[$i=0$.] Clearly, $g_0f_0=0$ and so we can choose $h_0=0$. 
    \item[$i=1$.] It is clear that $ \Delta_{1}=\underline{\delta}(\mathbb{J}\bullet\delta)=\upsilon_{1,r}+\upsilon_{1,l}$ and $\Gamma_{1,r}=\Gamma_{1,l}$. So, 
    \begin{align*}
       \hspace{6mm} g_1f_1=(\id_V\otimes m_B)\left(\Gamma_{1,r}\boxtimes \lambda_{Y}+\Delta_{1}\otimes \id_B\right)=(\id_V\otimes m_B)\left(\Gamma_{1,l}\boxtimes \lambda_{Y}+ (\id_V^{\otimes 2}\otimes m_B)(\upsilon_{1,r}\otimes \id_B)\right).
    \end{align*}
    Since $\left(\Gamma_{1,l}\boxtimes \lambda_{Y}+ (\id_V^{\otimes 2}\otimes m_B)\left(\upsilon_{1,r}\otimes \id_B\right)\right)\left(W_1\otimes B(-2)\right)\subseteq W_2\otimes B$ by Lemma \ref{lem: realtions between Gamma} and the construction of $\upsilon_{1,r}$, we can choose 
    $$h_1=\Gamma_{1,l}\boxtimes \lambda_{Y}+ (\id_V^{\otimes 2}\otimes m_B)(\upsilon_{1,r}\otimes \id_B).$$
    \item[$i\geq 2$.] Suppose the desired linear maps $\upsilon_{s,r}$ are obtained for all $s<i$. Then
\begin{align*}
h_{i-1}\partial_{i}=&\left(\sum_{u=1}^{i-1}(-1)^{u+1}({\det}\sigma)^{\boxtimes i-1-u}\boxtimes\Gamma_{u,l}\boxtimes\lambda_{Y}+({\id}_V^{\otimes i}\otimes m_B)(\upsilon_{i-1,r}\otimes {\id}_B)\right)(\id_V^{\otimes i-1}\otimes m_B)\\
=&({\id}_V^{\otimes i}\otimes m_B)\left(\sum_{u=1}^{i-1}(-1)^{u+1}({\det}\sigma)^{\boxtimes i-1-u}\boxtimes\Gamma_{u,l}\boxtimes\sigma\boxtimes \lambda_{Y}\right)\\
&+({\id}_V^{\otimes i}\otimes m_B)\left(\sum_{u=1}^{i-1}(-1)^{u+1}({\det}\sigma)^{\boxtimes i-1-u}\boxtimes\Gamma_{u,l}\boxtimes\delta\boxtimes\id_B+\upsilon_{i-1,r}\otimes\id_V \otimes {\id}_B\right).
\end{align*}
By Lemma \ref{lem: realtions between Gamma}, one obtains that     
     \begin{equation}\label{eq: gifi-hi-1partiali}
        \begin{aligned}
         g_if_i-h_{i-1}\partial_i=&(\id_V^{\otimes i}\otimes m_B)\sum_{u=1}^i(-1)^{u+1}(\det\sigma)^{\boxtimes i-u}\boxtimes\Gamma_{u,l}\boxtimes \lambda_{Y}\\
            &+({\id}_V^{\otimes i}\otimes m_B)\left(\left(\Delta_i+\sum_{u=1}^{i-1}(-1)^{u}({\det}\sigma)^{\boxtimes i-1-u}\boxtimes\Gamma_{u,l}\boxtimes\delta+\upsilon_{i-1,r}\otimes\id_V \right)\otimes \id_B\right)\\
            &=\partial_{i+1}h_{i}=(\id_V^{\otimes i}\otimes m_B)h_i.
        \end{aligned}
     \end{equation}

By an argument on the grading, ${h_i}(W_i\otimes B_0)\subseteq W_{i+1}\otimes y_1+W_{i+1}\otimes y_2+W_{i+1}\otimes V$. However, 
$$
\left(\left(\Delta_i+\sum_{u=1}^{i-1}(-1)^{u}({\det}\sigma)^{\boxtimes i-1-u}\boxtimes\Gamma_{u,l}\boxtimes\delta+\upsilon_{i-1,r}\otimes\id_V \right)\otimes \id_B\right)(W_i\otimes B_0)\subseteq W_i\otimes V^{\otimes 2}\otimes B_0,
$$
by Lemma \ref{lem: realtions between Gamma}. It implies that
$$
\left(h_i-\sum_{u=1}^i(-1)^{u+1}(\det\sigma)^{\boxtimes i-u}\boxtimes\Gamma_{u,l}\boxtimes \lambda_{Y}\right)(W_i\otimes B_0)\subseteq W_{i+1}\otimes V,
$$
by \eqref{eq: gifi-hi-1partiali}. Write $\upsilon_{i,r}$ for the restricted linear map $\left(h_i-\sum_{u=1}^i(-1)^{u+1}(\det\sigma)^{\boxtimes i-u}\boxtimes\Gamma_{u,l}\boxtimes \lambda_{Y}\right){\mid_{W_i}}$, and then
$$
h_i=\sum_{u=1}^i(-1)^{u+1}(\det\sigma)^{\boxtimes i-u}\boxtimes\Gamma_{u,l}\boxtimes \lambda_{Y}+(\id_V^{\otimes i+1}\otimes m_B)(\upsilon_{i,r}\otimes \id_B).
$$
\end{enumerate}

Another representation for $h_i$ comes from Lemma \ref{lem: realtions between Gamma}(a).  The last assertion follows by \eqref{eq: gifi-hi-1partiali} immediately.
\end{proof}

By now, we have constructed a triple  $(\{\delta_{i,r}\}, \{\delta_{i,l}\},\{\upsilon_{i,r}\})$ for the $\varsigma$-derivation $\nu$. Then we use them to obtain a minimal free resolution of the trivial module $k_B$ of the graded double Ore extension $B$.

\begin{theorem}\label{thm: minimial free resolution}
Let $B=A_P[y_1,y_2;\varsigma,\nu]$ be a graded double Ore extension of a Koszul algebra  $A=T(V)/(R)$ with $p_{12}\neq 0$. Let $(\{\delta_{i,r}\}, \{\delta_{i,l}\},\{\upsilon_{i,r}\})$ be a triple for $\nu$ constructed as in Lemma \ref{lemma: construction of delta_r and delta_l} and Lemma \ref{lem: construction of upsilon}. Then the following complex $F_\centerdot$ is a minimal free resolution of the trivial module $k_B$:
{\footnotesize
\begin{equation*}
    \begin{aligned}
&\cdots
\xlongrightarrow{}
W_{d}\otimes B(-2)\oplus (W_{d+1}\otimes B(-1))^{\oplus2}\oplus W_{d+2}\otimes B
\xlongrightarrow
{\left(
\begin{array}{ccc}
\partial_{d} & 0 &0  \\
-f_{d}&  -\partial_{d+1}^{\oplus2}&0\\
h_{d}&   g_{d+1}&\partial_{d+2}
\end{array}
\right)}
\cdots\cdots\\
&\quad\xlongrightarrow{}
W_1\otimes B(-2)\oplus (W_{2}\otimes B(-1))^{\oplus2}\oplus W_3\otimes B 
\xlongrightarrow{
\left(
\begin{array}{ccc}
\partial_{1}& 0&  0\\ 
-f_1& -\partial_{2}^{\oplus2}&  0\\
h_1 &   g_2&   \partial_3
\end{array}
\right)
}
B(-2)\oplus (W_1\otimes B(-1))^{\oplus2}\oplus W_2\otimes B\\
&\qquad\qquad\qquad\qquad\xlongrightarrow{
\left(
\begin{array}{ccc}
-f_0&  -\partial_1^{\oplus2}&  0\\
0&   g_1&  \partial_2
\end{array}
\right)
}
B(-1)^{\oplus2}\oplus W_1\otimes B
\xlongrightarrow{
\left(
\begin{array}{cc}
g_0& \partial_1
\end{array}
\right)
} 
B\xlongrightarrow{} 0\to \cdots.
\end{aligned}
\end{equation*}
}
\end{theorem}
\begin{proof}
Recall that $B/A_{\geq 1}B\cong J$ as right graded $B$-modules. By the construction of $f_\centerdot$  in Lemma \ref{lemma: construction of delta_r and delta_l}, the mapping cone $$D_\centerdot=\mathrm{Cone}\left(Q_\centerdot(-2)\xlongrightarrow{f_\centerdot}Q^{\oplus 2}_\centerdot(-1)\right)$$ 
is a free resolution of $J^{\oplus 2}(-1)/\Image (\mathbb{J}\bullet\lambda_Y)$. By Lemma \ref{lem: construction of upsilon}, one obtains a chain homomorphism 
$$
D_\centerdot\xlongrightarrow{(h_\centerdot[-1],g_\centerdot)}Q_\centerdot,$$ 
and it is also induced by the right graded $B$-module homomorphism $\underline{\lambda_Y}:J^{\oplus 2}(-1)/\Image (\mathbb{J}\bullet \lambda_Y)\to J$. The complex $F_\centerdot$ is exactly the mapping cone of $(h_\centerdot[-1],g_\centerdot)$. Then the results follows by the following commutative diagram
$$
\xymatrix{
0\ar[r]& D_\centerdot\ar[r]^{(h_\centerdot[-1],g_\centerdot)}\ar[d]^{\simeq}&Q_\centerdot\ar[r]\ar[d]^{\simeq}  & F_\centerdot\ar[r]\ar[d]^{\simeq} &0\\
0\ar[r]& J^{\oplus 2}(-1)/\Image(\mathbb{J}\bullet\lambda_Y)\ar[r]^(0.7){\underline{\lambda_Y}}&J\ar[r]  & k_B\ar[r] &0.
}
$$
\end{proof}

The minimal free resolution also implies that the graded double Ore extension $B=A_P[y_1,y_2;\varsigma,\nu]$ with nontrivial $\nu$  preserves Koszulity. It is a generalization of \cite[Theorem 1]{ZVZ}. 
\begin{corollary}\label{cor: double ore extension preserve Koszul}
Let $B=A_P[y_1,y_2;\varsigma,\nu]$ be a graded double Ore extension of a Koszul algebra  $A=T(V)/(R)$. Then $B$ is Koszul.
\end{corollary}

\begin{lemma}\label{lem: construction of upsilon l} Let $B=A_P[y_1,y_2;\varsigma,\nu]$ be a graded double Ore extension of a Koszul algebra $A=T(V)/(R)$ with $p_{12}\neq 0$. Let $(\{\delta_{i,r}\}, \{\delta_{i,l}\},\{\upsilon_{i,r}\})$ be a triple for $\nu$ constructed as in Lemma \ref{lemma: construction of delta_r and delta_l} and Lemma \ref{lem: construction of upsilon}. 
Then there exist linear maps 
$$\upsilon_{i,l}:W_i\to V\otimes W_{i+1}$$ 
such that for any $i\geq 1$,
$$\upsilon_{i,l}+\upsilon_{i-1,l}\otimes{\id}_V=(-1)^i(\upsilon_{i,r}-\det\sigma\otimes \upsilon_{i-1,r})+\underline{\sigma\boxtimes \delta_{i,r}} \left(\mathbb{J}\bullet \delta_{i,l}\right)+\underline{\delta_{i,l}\boxtimes\id_V}\left(\mathbb{J}\bullet\delta_{{i},r}\right).$$
\end{lemma}

\begin{proof}
One can obtain the following commutative diagram (without dotted arrows),
{
$$
\xymatrix{
\hspace{-1mm}\cdots\ar[r]&W_{d}\otimes B(-2)\ar[r]^{\partial_{d}}\ar[d]^{\Phi'_{d}}\ar@{-->}[ld]_{s'_{d}}  &\cdots\ar[r]^(0.3){} &W_2\otimes B(-2)\ar[r]^(0.5){{\partial_{2}}} \ar[d]^{\Phi'_2}\ar@{-->}[ld]_{s'_2}&W_1\otimes B(-2)\ar[r]^(0.55){{{\partial_1}}}\ar[d]^{\Phi'_1}\ar@{-->}[ld]_{s'_1}&{\rm Im}\partial_1(-2)\ar[r]\ar[d]^{0}&0\\
\hspace{-1mm}\cdots\ar[r]&V\otimes W_d\otimes B\ar[r]^(0.7){{\id}_V\otimes\partial_d}&\cdots
\ar[r] &V\otimes W_2\otimes B\ar[r]^(0.55){{\id}_V\otimes \partial_{1}} &V\otimes W_1\otimes B\ar[r]^(0.47){{\id}_V\otimes\partial_1}&\Image ({\id}_V\otimes\partial_1)\ar[r]&0,
}
$$
}
where the top row is by the truncation of ${Q_\centerdot}(-2)$, the bottom row is a truncation of $V\otimes Q_\centerdot$, and 
$$
\Phi'_i:=({\id}_V^{\otimes i+1}\otimes m_B)\left(\left((-1)^i\upsilon_{i,r}-(-1)^i\det\sigma\otimes \upsilon_{i-1,r}+\underline{\sigma\boxtimes \delta_{i,r}} \left(\mathbb{J}\bullet \delta_{i,l}\right)+\underline{\delta_{i,l}\boxtimes\id_V}\left(\mathbb{J}\bullet\delta_{{i},r}\right)
\right)\otimes{\id}_B\right),
$$
for any $i\geq 1$. In fact, $(\id_V\otimes\partial_1)\Phi'_1=0$ follows by the definition of $\upsilon_{1,r}$ and if $i\geq 2$
\begin{align*}
&(\id_V\otimes\partial_i)\Phi'_i-\Phi'_{i-1}\partial_i\\
=&
(-1)^i({\id}_V^{\otimes i}\otimes m_B)\left(\left(\upsilon_{i,r}+\upsilon_{i-1,r}\otimes\id_V-
\det\sigma\otimes\left(\upsilon_{i-1,r}+\upsilon_{i-2,r}\otimes \id_V\right)\right)\otimes\id_B
\right)
\\
&
+({\id}_V^{\otimes i}\otimes m_B)\left(
\left(
\underline{\sigma\boxtimes \delta_{i,r}} \left(\mathbb{J}\bullet \delta_{i,l}\right)+\underline{\delta_{i,l}\boxtimes\id_V}\left(\mathbb{J}\bullet\delta_{{i},r}\right)
\right)\otimes\id_B
\right)\\
&-({\id}_V^{\otimes i}\otimes m_B)\left(\left(\underline{\sigma\boxtimes \delta_{i-1,r}\boxtimes\id_V} \left(\mathbb{J}\bullet \left(\delta_{i-1,l}\boxtimes\id_V\right)\right)+\underline{\delta_{i-1,l}\boxtimes\id_V^{\otimes 2}}\left(\mathbb{J}\bullet\left(\delta_{{i-1},r}\boxtimes \id_V\right)\right)
\right)\otimes{\id}_B\right)\\
=&
(-1)^i({\id}_V^{\otimes i}\otimes m_B)\left(\left(\Delta_i-
\det\sigma\otimes\Delta_{i-1}+(-1)^{i-1}\Gamma_{i-1,l}\boxtimes\delta\right)\otimes\id_B
\right)
\\
&
+(-1)^i({\id}_V^{\otimes i}\otimes m_B)\left(
\left(
\left(\underline{\sigma^{\boxtimes i}\boxtimes\delta}+\underline{\sigma\boxtimes\delta_{i-1,r}\boxtimes\id_V }\right) \left(\mathbb{J}\bullet \left(
\sigma\boxtimes\delta_{i-1,r}+(-1)^i\delta_{i-1,l}\boxtimes\id_V-\delta_{i,r}
\right)
\right)
\right)\otimes\id_B
\right)\\
&
+(-1)^i({\id}_V^{\otimes i}\otimes m_B)\left(
\left(
\underline{\left(
\sigma\boxtimes\delta_{i-1,r}+(-1)^i\delta_{i-1,l}\boxtimes\id_V-\delta_{i,r}
\right)\boxtimes\id_V}\left(\mathbb{J}\bullet\delta_{{i},r}\right)
\right)\otimes\id_B
\right)\\
&-({\id}_V^{\otimes i}\otimes m_B)\left(\left(\underline{\sigma\boxtimes\delta_{i-1,r}\boxtimes\id_V}\left(\mathbb{J}\bullet \left(\delta_{i-1,l}\boxtimes\id_V\right)\right)+\underline{\delta_{i-1,l}\boxtimes\id_V^{\otimes 2}}\left(\mathbb{J}\bullet\left(\delta_{{i-1},r}\boxtimes \id_V\right)\right)
\right)\otimes{\id}_B\right)\\
=&
(-1)^i({\id}_V^{\otimes i}\otimes m_B)\left(\left(-
\det\sigma\otimes\Delta_{i-1}+(-1)^{i-1}\Gamma_{i-1,l}\boxtimes\delta\right)\otimes\id_B
\right)
\\
&
+(-1)^i({\id}_V^{\otimes i}\otimes m_B)\left(\left(
\underline{\sigma^{\boxtimes i}\boxtimes\delta}(\mathbb{J}\bullet (\sigma\boxtimes\delta_{i-1,r}))
\right)\otimes\id_B\right)\\
&+(-1)^i({\id}_V^{\otimes i}\otimes m_B)
\left(\left(\underline{\sigma\boxtimes\delta_{i-1,r}\boxtimes\id_V }(\mathbb{J}\bullet (\sigma\boxtimes\delta_{i-1,r}))
\right)\otimes\id_B\right)\\
&+({\id}_V^{\otimes i}\otimes m_B)
\left(\left(
\underline{\sigma^{\boxtimes i}\boxtimes\delta}(\mathbb{J}\bullet
(\delta_{i-1,l}\boxtimes\id_V))
\right)\otimes\id_B\right)+({\id}_V^{\otimes i}\otimes m_B)\left(\left( \underline{\delta_{i-1,l}\boxtimes\id_V^{\otimes 2}}\left(\mathbb{J}\bullet\left(\sigma^{\boxtimes i-1}\boxtimes \delta\right)\right)
\right)\otimes{\id}_B\right)\\
=&
(-1)^i({\id}_V^{\otimes i}\otimes m_B)\left(\left(-
\det\sigma\otimes\Delta_{i-1}+(-1)^{i-1}\Gamma_{i-1,l}\boxtimes\delta\right)\otimes\id_B
\right)
\\
&
+(-1)^i({\id}_V^{\otimes i}\otimes m_B)\left(\det\sigma\otimes\left(
\underline{\sigma^{\boxtimes i-1}\boxtimes\delta}(\mathbb{J}\bullet \delta_{i-1,r})\right)\otimes\id_B
\right)\\
&+(-1)^i({\id}_V^{\otimes i}\otimes m_B)\left(
\det\sigma\otimes\left(\underline{\delta_{i-1,r}\boxtimes\id_V}(\mathbb{J}\bullet\delta_{i-1,r})\right)\otimes\id_B
\right)\\
&+({\id}_V^{\otimes i}\otimes m_B)
\left(
\left(
((\sigma^{\boxtimes i})^T\bullet\mathbb{J}\bullet\delta_{i-1,l})^T\boxtimes\delta
\right)
\otimes\id_B
\right)
+({\id}_V^{\otimes i}\otimes m_B)\left(
\left(
(\delta_{i-1,l}^T\bullet\mathbb{J}\bullet\sigma^{\boxtimes i-1})\boxtimes \delta
\right)
\otimes{\id}_B\right)\\
=&0,
\end{align*}
where the second equality follows by Lemma \ref{lemma: construction of delta_r and delta_l} and Lemma \ref{lem: construction of upsilon}, the third equality holds by Lemma \ref{lemma: construction of delta_r and delta_l}, and the fourth equality holds by Lemma \ref{lem: relation between Theta^{(i)}}. 

Hence, the chain homomorphism $\Phi'=\{\Phi'_i\}$ is null homotopic.  Then the rest of proof is similar to the one of Lemma \ref{lemma: construction of delta_r and delta_l}.
\end{proof}

Up to now, we construct a quadruple $(\{\delta_{i,r}\}, \{\delta_{i,l}\},\{\upsilon_{i,r}\},\{\upsilon_{i,l}\})$ for the $\varsigma$-derivation of $\nu$ in a graded double Ore extension $B=A_P[y_1,y_2;\varsigma,\nu]$ of a Koszul algebra $A$. It should be noted that such quadruple is not unique for $\nu$. We have shown that the first classes of maps in a quadruple are used to construct a minimal resolution for $k_B$. In next section, we will use such quadruple to give a homological invariant for $\nu$ and compute the Yoneda product of Ext-algebra $E(B)$ in the following two sections.

To make the construction in this section to be acceptable,  we give an example to explain the whole procedure. 

\begin{example}\label{exa: example 1}
In this example, we assume the field $k$ is algebraically closed. Let $A=T( V)/(R)$ where $V$ is a vector space with a basis $\{x_1,x_2\}$ and $R$ is spanned by $\{x_2\otimes x_1-px_1\otimes x_2\}$, where $p=\sqrt{-1}$. Let $B=A_{\{p,0\}}[y_1,y_2,\varsigma,\nu]$ be a graded double Ore extension, where
$$
\begin{array}{llll}
\varsigma(x_1)=\begin{pmatrix}
    0& px_2\\
    -px_2 &0
\end{pmatrix},& 
\varsigma(x_2)=\begin{pmatrix}
    0& px_1\\
    px_1 &0
\end{pmatrix},&
\nu(x_1)=\begin{pmatrix}
    0\\
    -px_1\otimes x_2 
\end{pmatrix},&
\nu(x_2)=\begin{pmatrix}
    px_1\otimes x_2\\
        0
\end{pmatrix}.
\end{array}
$$ The matrix
$$
\mathbb{J}=\begin{pmatrix}
    0&-p\\
    1&0
\end{pmatrix}.
$$
The linear maps $\sigma=\varsigma_{\mid V}:V\to M_{2\times 2}(V)$ and $\delta_{\mid V}=\nu_{\mid V}:V\to (V\otimes V)^{\oplus 2}$ are 
$$
\begin{array}{llll}
\sigma(x_1)=\begin{pmatrix}
    0& px_2\\
    -px_2 &0
\end{pmatrix},& 
\sigma(x_2)=\begin{pmatrix}
    0& px_1\\
    px_1 &0
\end{pmatrix},&
\delta(x_1)=\begin{pmatrix}
    0\\
    -px_1\otimes x_2 
\end{pmatrix},&
\delta(x_2)=\begin{pmatrix}
   p x_1\otimes x_2\\
        0
\end{pmatrix}.
\end{array}
$$
One obtains that 
\begin{align*}
\delta(x_2\otimes x_1-px_1\otimes x_2)=&\begin{pmatrix}
x_1\otimes x_1\otimes x_2+px_1\otimes x_2\otimes x_1\\
-px_2\otimes x_1\otimes x_2-x_1\otimes x_2\otimes x_2
\end{pmatrix}, \\
    \underline{\delta}(\mathbb{J}\bullet \delta)(x_1)=&0,\\
    \underline{\delta}(\mathbb{J}\bullet \delta)(x_2)=&px_2\otimes x_1\otimes x_2+x_1\otimes x_2\otimes x_2 ,
\end{align*}
and linear maps $\delta_{2,r}:R\to (R\otimes V)^{\oplus 2}$, $\delta_{2,l}:R\to (V\otimes R)^{\oplus 2}$, $\upsilon_{1,r}:V\to R\otimes V$ and $\upsilon_{1,l}:V\to V\otimes R$ satisfy
\begin{align*}
    \delta_{2,r}(x_2\otimes x_1-px_1\otimes x_2)&=\begin{pmatrix}
0\\
(x_2\otimes x_1-px_1\otimes x_2)\otimes (-px_2)
\end{pmatrix},\\
\delta_{2,l}(x_2\otimes x_1-px_1\otimes x_2)&=\begin{pmatrix}
px_1\otimes( x_2\otimes x_1-p x_1\otimes x_2)\\
0
\end{pmatrix}, \\
\upsilon_{1,r}(x_1)&=0,\\
\upsilon_{1,r}(x_2)&=(x_2\otimes x_1-px_1\otimes x_2)\otimes (px_2),\\
\upsilon_{1,l}&=0.
\end{align*}
Write $\delta_{0,r}=\delta_{0,l}=0$, $\delta_{1,r}=\delta_{1,l}=\delta_{\mid V}$.
The Koszul complex for $k_A$ is
$$
(\cdots\xlongrightarrow{} 0\xlongrightarrow{}W_2\otimes A\xlongrightarrow{\partial^A_{2}}W_1\otimes A\xlongrightarrow{\partial^A_{1}}A\xlongrightarrow{}0\to\cdots)\ \xlongrightarrow{\varepsilon_A}\  k_A,
$$
where $W_1=V$ and $W_2=R$, and then
\begin{equation*}
Q_\centerdot=(\cdots\xlongrightarrow{} 0\xlongrightarrow{}W_2\otimes B\xlongrightarrow{\partial_{2}}W_1\otimes B\xlongrightarrow{\partial_{1}}B\to0\to\cdots)\xlongrightarrow{\varepsilon_A\otimes_AB}B/A_{\geq1}B.
\end{equation*}
The commutative diagram (without dotted arrows) in Lemma \ref{lem: construction of upsilon} becomes
\begin{equation*}
\xymatrix@C=20px{
\cdots\ar[r]&0\ar[r] &W_2\otimes B(-2) \ar[r]^{\partial_2}\ar[d]^{f_2}\ar@{-->}[ldd]^(0.6){h_{2}=0}  &W_1\otimes B(-2)\ar[r]^(0.55){\partial_{1}}\ar[d]^{f_1}\ar@{-->}[ldd]^(0.6){h_{1}} &B(-2)\ar@{-->}[ldd]^(0.6){h_0=0}\ar[d]^{f_0}\ar[r]&0\ar@{-->}[ldd]^(0.6){h_{-1}=0}\ar[r]\ar[d]^0&\cdots
\\
\cdots\ar[r]&0\ar[r]  & (W_2\otimes B(-1) )^{\oplus 2}\ar[r]^(0.5){{\partial^{\oplus 2}_2}}\ar[d]^{g_2} &(W_1\otimes B(-1))^{\oplus 2}\ar[r]^(0.65){{\partial^{\oplus 2}_{1}}} \ar[d]^{g_1}&B(-1)^{\oplus 2}\ar[d]^{g_0}\ar[r]&0\ar[r]\ar[d]^0&\cdots
\\
\cdots\ar[r]&0\ar[r]  & W_2\otimes B\ar[r]^{\partial_2} &W_1\otimes B\ar[r]^(0.6){{\partial_{1}}} &B\ar[r]&0\ar[r]&\cdots,
}
\end{equation*}
where 
\begin{align*}
    f_i
&=\mathbb{J}\bullet\left(\sigma^{\boxtimes i}\boxtimes \lambda_Y
+({\id}_V^{\otimes i}\otimes m_B)^{\oplus2}(\delta_{i,r}\boxtimes \id_B)
\right),\quad i=0,1,2,\\
g_i
&=  \underline{\sigma^{\boxtimes i}\boxtimes\lambda_Y}+({\id}_V^{\otimes i}\otimes m_B)
\underline{\delta_{i,r}\boxtimes \id_B},\quad i=0,1,2,
\\
h_1&=\Gamma_{1,r}\boxtimes\lambda_{Y}+({\id}_V^{\otimes 2}\otimes m_B)(\upsilon_{1,r}\otimes {\id}_B),
\end{align*}
and the linear map
$$
\begin{array}{cclc}
\Gamma_{1,r}:&W_1&\to &M_{1\times2}(R)\\
~&x_1&\mapsto&(x_2\otimes x_1-px_1\otimes x_2,0),\\
~&x_2&\mapsto&(0,p(x_2\otimes x_1-px_1\otimes x_2)).
\end{array}
$$
By Theorem \ref{thm: minimial free resolution}, we have a minimal free resolution of $k_B$:
\begin{equation*}
    \begin{aligned}
&0\xlongrightarrow{}
W_1\otimes B(-2)\oplus (W_{2}\otimes B(-1))^{\oplus2}\oplus W_3\otimes B 
\xlongrightarrow{
\left(
\begin{array}{ccc}
\partial_{1}& 0&  0\\ 
-f_1& -\partial_{2}^{\oplus2}&  0\\
h_1 &   g_2&   \partial_3
\end{array}
\right)
}
B(-2)\oplus (W_1\otimes B(-1))^{\oplus2}\oplus W_2\otimes B\\
&\qquad\qquad\qquad\qquad\xlongrightarrow{
\left(
\begin{array}{ccc}
-f_0&  -\partial_1^{\oplus2}&  0\\
0&   g_1&  \partial_2
\end{array}
\right)
}
B(-1)^{\oplus2}\oplus W_1\otimes B
\xlongrightarrow{
\left(
\begin{array}{cc}
g_0& \partial_1
\end{array}
\right)
} 
B\xlongrightarrow{} 0\to \cdots.
\end{aligned}
\end{equation*}
\end{example}

\subsection{Two properties}
We show two properties of quadruples for skew derivations in graded double Ore extensions, which will be used later.

\begin{proposition}\label{prop: relation between delta_r and delta_l}
Let $(\{\delta_{i,r}\}, \{\delta_{i,l}\},\{\upsilon_{i,r}\},\{\upsilon_{i,l}\})$ be a quadruple of linear maps for $\nu$ constructed as in Lemma \ref{lemma: construction of delta_r and delta_l}, Lemma \ref{lem: construction of upsilon} and Lemma \ref{lem: construction of upsilon l}.
\begin{enumerate}
\item For any $d\geq 1$,
$$
\sum_{i=1}^d (-1)^{i}\delta_{i,r}\boxtimes \id_V^{\boxtimes d-i}=(-1)^{d+1}\sum_{i=1}^{d}(-1)^{i}\sigma^{\boxtimes d-i}\boxtimes\delta_{i,l}.
$$
\item For any $d\geq 2$, 
$$
\sum_{i=1}^{d-1}(-1)^{i+1}\sigma^{\boxtimes d-i-1}\boxtimes\delta_{i,l}\boxtimes \id_V=(-1)^{d+1}\sum_{i=1}^{d-1}(-1)^i\delta_{i,r}\boxtimes \id_V^{\boxtimes d-i}.
$$

\item For any $d\geq 1$,
\begin{align*}
   \sum_{i=1}^d\upsilon_{i,r}\otimes\id^{\otimes d-i}_V=&
\sum_{i=1}^d(-1)^i(\det\sigma)^{\otimes d-i}\otimes\upsilon_{i,l}
-\sum_{i=1}^d\sum_{u=1}^i(-1)^u(\det\sigma)^{\otimes d-i}\otimes\\&\left(\underline{\sigma\boxtimes \delta_{u,r}\boxtimes\id^{\boxtimes i-u}_V} \left(\mathbb{J}\bullet (\delta_{u,l}\boxtimes\id^{\boxtimes i-u}_V)\right)+\underline{\delta_{u,l}\boxtimes\id^{\boxtimes i-u+1}_V}\left(\mathbb{J}\bullet(\delta_{{u},r}\boxtimes\id^{\boxtimes i-u}_V)\right)\right).
\end{align*}

\item For any $d\geq 2$,
\begin{align*}
\sum_{i=1}^{d-1}\upsilon_{i,r}\otimes\id^{\otimes d-i}_V
=&\sum_{i=1}^{d-1}(-1)^{i}(\det\sigma)^{\otimes d-i-1}\otimes\upsilon_{i,l}\otimes\id_V-\sum_{i=1}^{d-1}\sum_{u=1}^{i}(-1)^u(\det\sigma)^{\otimes d-i-1}\otimes\\
&\left(\underline{\sigma\boxtimes \delta_{u,r}\boxtimes\id^{\boxtimes i-u+1}_V} \left(\mathbb{J}\bullet (\delta_{u,l}\boxtimes\id^{\boxtimes i-u+1}_V)\right)+\underline{\delta_{u,l}\boxtimes\id^{\boxtimes i-u+2}_V}\left(\mathbb{J}\bullet(\delta_{{u},r}\boxtimes\id^{\boxtimes i-u+1}_V)\right)\right).
\end{align*}
\end{enumerate}

\end{proposition}

\begin{proof}
(a,b) By Lemma \ref{lemma: construction of delta_r and delta_l}, one obtains that
\begin{align*}
    \sum_{i=1}^d (-1)^{i}\delta_{i,r}\boxtimes \id_V^{\boxtimes d-i}&=\sum_{i=1}^d (-1)^{i}\sigma\boxtimes\delta_{i-1,r}\boxtimes \id_V^{\boxtimes d-i}-\delta_{d,l},\\
(-1)^{d+1}\sum_{i=1}^{d-1}(-1)^i\delta_{i,r}\boxtimes \id_V^{\boxtimes d-i}&=(-1)^{d+1}\sum_{i=1}^{d-1} (-1)^{i}\sigma\boxtimes\delta_{i-1,r}\boxtimes \id_V^{\boxtimes d-i}+(-1)^d\delta_{d-1,l}\boxtimes\id_V.
\end{align*}
Then the two results follow by induction on $d$.

(c) By Lemma \ref{lem: construction of upsilon l}, one obtains that 
\begin{align*}
    &\sum_{u=1}^i\upsilon_{u,r}\otimes\id^{\otimes i-u}_V-\det\sigma\otimes \upsilon_{u-1,r} \otimes\id^{\otimes i-u}_V\\
=&\sum_{u=1}^i(-1)^u\left(\upsilon_{u,l}\otimes\id^{\otimes i-u}_V+\upsilon_{u-1,l}\otimes\id^{\otimes i-u+1}_V\right) \\
&-\sum_{u=1}^i(-1)^u\left(\underline{\sigma\boxtimes \delta_{u,r}\boxtimes\id^{\boxtimes i-u}_V} \left(\mathbb{J}\bullet (\delta_{u,l}\boxtimes\id^{\boxtimes i-u}_V)\right)+\underline{\delta_{u,l}\boxtimes\id^{\boxtimes i-u+1}_V}\left(\mathbb{J}\bullet(\delta_{{u},r}\boxtimes\id^{\boxtimes i-u}_V)\right)\right)\\
=&(-1)^i\upsilon_{i,l}-\sum_{u=1}^i(-1)^u\left(\underline{\sigma\boxtimes \delta_{u,r}\boxtimes\id^{\boxtimes i-u}_V} \left(\mathbb{J}\bullet (\delta_{u,l}\boxtimes\id^{\boxtimes i-u}_V)\right)+\underline{\delta_{u,l}\boxtimes\id^{\boxtimes i-u+1}_V}\left(\mathbb{J}\bullet(\delta_{{u},r}\boxtimes\id^{\boxtimes i-u}_V)\right)\right),
\end{align*}
for any $i\geq 1$. Then 
\begin{align*}
&\sum_{i=1}^d(-1)^i(\det\sigma)^{\otimes d-i}\otimes\upsilon_{i,l}-\sum_{i=1}^d\sum_{u=1}^i(-1)^u(\det\sigma)^{\otimes d-i}\otimes\\
&~\qquad\qquad\qquad\qquad\qquad\left(\underline{\sigma\boxtimes \delta_{u,r}\boxtimes\id^{\boxtimes i-u}_V} \left(\mathbb{J}\bullet (\delta_{u,l}\boxtimes\id^{\boxtimes i-u}_V)\right)+\underline{\delta_{u,l}\boxtimes\id^{\boxtimes i-u+1}_V}\left(\mathbb{J}\bullet(\delta_{{u},r}\boxtimes\id^{\boxtimes i-u}_V)\right)\right)\\
=&\sum_{i=1}^d\sum_{u=1}^i(\det\sigma)^{\otimes d-i}\otimes\upsilon_{u,r}\otimes\id^{\otimes i-u}_V-\sum_{i=1}^d\sum_{u=1}^i(\det\sigma)^{\otimes d-i+1}\otimes \upsilon_{u-1,r} \otimes\id^{\otimes i-u}_V\\
=&\sum_{i=1}^d\upsilon_{i,r}\otimes\id^{\otimes d-i}.
\end{align*}
(d) By Lemma \ref{lem: construction of upsilon l}, we have
\begin{align*}
    &\sum_{u=1}^{i-1}\upsilon_{u,r}\otimes\id^{\otimes i-u}_V-\det\sigma\otimes \upsilon_{u-1,r} \otimes\id^{\otimes i-u}_V\\
=&\sum_{u=1}^{i-1}(-1)^u\left(\upsilon_{u,l}\otimes\id^{\otimes i-u}_V+\upsilon_{u-1,l}\otimes\id^{\otimes i-u+1}_V\right) \\
&-\sum_{u=1}^{i-1}(-1)^u\left(\underline{\sigma\boxtimes \delta_{u,r}\boxtimes\id^{\boxtimes i-u}_V} \left(\mathbb{J}\bullet (\delta_{u,l}\boxtimes\id^{\boxtimes i-u}_V)\right)+\underline{\delta_{u,l}\boxtimes\id^{\boxtimes i-u+1}_V}\left(\mathbb{J}\bullet(\delta_{{u},r}\boxtimes\id^{\boxtimes i-u}_V)\right)\right)\\
=&(-1)^{i-1}\upsilon_{i-1,l}\otimes\id_V-\sum_{u=1}^{i-1}(-1)^u\left(\underline{\sigma\boxtimes \delta_{u,r}\boxtimes\id^{\boxtimes i-u}_V} \left(\mathbb{J}\bullet (\delta_{u,l}\boxtimes\id^{\boxtimes i-u}_V)\right)+\underline{\delta_{u,l}\boxtimes\id^{\boxtimes i-u+1}_V}\left(\mathbb{J}\bullet(\delta_{{u},r}\boxtimes\id^{\boxtimes i-u}_V)\right)\right),
\end{align*}
for any $i\geq 2$. The rest of the proof is similar to (c).
\end{proof}

\begin{proposition}
Let $\delta,\delta'$ be two linear maps from $V$ to $(V\otimes V)^{\oplus2}$ such that $\pi_A^{\oplus2}\delta=\pi_A^{\oplus2}\delta'=\nu_{|_V}$. Let $(\{\delta_{i,r}\}, \{\delta_{i,l}\})$ and $(\{\delta'_{i,r}\}, \{\delta'_{i,l}\})$ be two pairs of linear maps for $\nu$ constructed from $\delta$  and $\delta'$  respectively as in Lemma \ref{lemma: construction of delta_r and delta_l}. Then for any $i\geq1$,
\begin{enumerate}
\item $
\left(\sum_{j=0}^{i-1}(-1)^j(\delta_{i-j,r}-\delta'_{i-j,r})\boxtimes\id_V^{\boxtimes j}\right)(W_i)\subseteq W^{\oplus2}_{i+1}.
$
\item $
\left(\delta_{i,l}-\delta'_{i,l}+(-1)^{i}\sum_{j=1}^{i-1}(-1)^j\sigma\boxtimes(\delta_{i-j,r}-\delta'_{i-j,r})\boxtimes\id_V^{\boxtimes j-1}\right)(W_i)\subseteq W_{i+1}^{\oplus2}.
$
\end{enumerate}
As a consequence, if $W_{d+1}=0$ for some $d\geq1$, then
\begin{align}
\sum_{j=0}^{d-1}(-1)^{j}\left(
    \delta_{d-j,r}
\boxtimes\id_V^{\boxtimes j}\right)&=\sum_{j=0}^{d-1}(-1)^{j}\left(
    \delta'_{d-j,r}\boxtimes\id_V^{\boxtimes j}\right),\label{eq: delta r in W_d+1=0}\\
(-1)^{d}
    \delta_{d,l}+\sum_{j=1}^{d-1}(-1)^{j}\left(\sigma\boxtimes
    \delta_{d-j,r} \boxtimes\id_V^{\boxtimes j-1}\right)&=(-1)^{d}
    \delta'_{d,l}+\sum_{j=1}^{d-1}(-1)^{j}\left(\sigma\boxtimes 
    \delta'_{d-j,r}\boxtimes\id_V^{\boxtimes j-1}\right).\label{eq: delta r and delta l in W_d+1=0}
\end{align}

\end{proposition}
\begin{proof}
(a)
Let $\{f_{i}\}$ and $\{f'_{i}\}$ be two lifts of $\mathbb{J}\bullet\lambda_Y:B/A_{\geq 1}B(-2)\to \left(B/A_{\geq1}B\right)^{\oplus2}(-1)$, and $\{\delta_{i,r}
\}$ and $\{\delta'_{i,r}
\}$ be the associated linear maps for $\{f_{i}\}$ and $\{f'_{i}\}$  as in Lemma \ref{lemma: construction of delta_r and delta_l}  respectively. 

Clearly, $\{f_{i}\}$ and $\{f'_{i}\}$ are homotopic, so there exist right graded $B$-module homomorphisms $s_i:W_i\otimes B(-2)\to \left(W_{i+1}\otimes B(-1)\right)^{\oplus2}$ 
such that 
    $$
     f_i-f'_i=(\id_V^{\otimes i}\otimes m_B)^{\oplus2}\left(
     \mathbb{J}\bullet
    \left(\delta_{i,r}-\delta'_{i,r}\right)\boxtimes\id_B
    \right)=s_{i-1}\partial_i+\partial_{i+1}^{\oplus2}s_{i},
    $$
for all $i\geq 1$. In particular, $s_0=0$. 
\begin{scriptsize}$${
\xymatrix{
\cdots\ar[r] & W_i\otimes B(-2) \ar[r]^{\partial_i}\ar@{-->}[ld]_{s_i}\ar[d]^{f_i-f'_i} &W_{i-1}\otimes B(-2)\ar[r]^(0.65){\partial_{i-1}}\ar@{-->}[ld]_{s_{i-1}}\ar[d]^{f_{i-1}-f'_{i-1}}&\cdots\ar[r]^(0.35){\partial_{3}}&W_2\otimes B(-2)\ar@{-->}[ld]_{s_2}\ar[r]^(0.55){\partial_2}\ar[d]^{f_2-f'_2} &W_1\otimes B(-2)\ar@{-->}[ld]_{s_1}\ar[r]^(0.55){\partial_{1}}\ar[d]^{f_1-f'_1} &B(-2)\ar[d]^{f_0-f'_0}\ar@{-->}[ld]_{0}
\\
\cdots\ar[r] & (W_i\otimes B(-1))^{\oplus2} \ar[r]^{\partial_i^{\oplus2}} &(W_{i-1}\otimes B(-1))^{\oplus2}\ar[r]^(0.6){{\partial_{i-1}^{\oplus2}}}&\cdots\ar[r]^(0.41){{\partial_{3}^{\oplus2}}}&(W_2\otimes B(-1))^{\oplus2}\ar[r]^(0.55){\partial_{1}^{\oplus2}} &(W_1\otimes B(-1))^{\oplus2}\ar[r]^(0.6){{\partial_{1}^{\oplus2}}} &B^{\oplus2}(-1).
}
}
$$
\end{scriptsize}
Hence, ${s_i}_{\mid W_i}=\mathbb{J}\bullet
    \left(\delta_{i,r}-\delta'_{i,r}\right)-{s_{i-1}}_{\mid W_{i-1}}\boxtimes\id_V$ and $ \Image {s_i}_{\mid W_i}\subseteq W_{i+1}^{\oplus2},$
for any $i\geq 1$. Note that ${s_{1}}_{\mid W_1}=
    \mathbb{J}\bullet (\delta_{1,r}-\delta'_{1,r})$, we have
$$
{s_i}_{\mid W_i}=\sum_{j=0}^{i-1}(-1)^j\mathbb{J}\bullet\left(
    (\delta_{i-j,r}-\delta'_{i-j,r})\boxtimes\id_V^{\boxtimes j}\right).
$$
The result follows.

(b)  By Lemma \ref{lemma: construction of delta_r and delta_l}, one obtains 
\begin{align*}
&\sum_{j=1}^i(-1)^{i-j}
    \left(\delta_{j,r}-\delta'_{j,r}\right)\boxtimes\id_V^{\boxtimes i-j}
-\sum_{j=1}^i(-1)^{i-j}\sigma\boxtimes\left(
    \delta_{j-1,r}-\delta'_{j-1,r}\right)\boxtimes\id_V^{\boxtimes i-j}=
    (-1)^{i+1}\left(
    \delta_{i,l}-\delta'_{i,l}\right).
\end{align*}
Then
$$
\mathbb{J}^{-1}\bullet {s_i}_{\mid W_i}=\sum_{j=0}^{i-1}(-1)^{j}\left(
    \delta_{i-j,r}-\delta'_{i-j,r}\right)\boxtimes \id_V^{\boxtimes j}=(-1)^{i+1}\left(
    \delta_{i,l}-\delta'_{i,l}\right)-\sum_{j=1}^{i-1}(-1)^{j}\sigma\boxtimes\left(
    \delta_{i-j,r}-\delta'_{i-j,r}\right)\boxtimes \id_V^{\boxtimes j-1}.
$$
Then the result holds by (a).

The last consequence can be obtained by (a, b) and Proposition \ref{prop: relation between delta_r and delta_l}(a) immediately.
\end{proof}

\section{Nakayama automorphisms of graded double Ore extensions of Koszul AS-regular algebras}
In this section, $A=T(V)/(R)$ is a Koszul AS-regular algebra of dimension $d$ with the Nakayama automorphism $\mu_A$, where $V$ is a vector space with a basis $\{x_1,x_2,\cdots,x_n\}$. Write $\{\eta_1,\eta_2,\cdots,\eta_n\}$ for a basis of $W_{d-1}$ and  $\omega$ for a basis of $W_d$.

We assume $B=A_P[y_1,y_2;\varsigma,\nu]$ is graded double Ore extension of $A$ where $p_{12}\neq 0$, and so it is a Koszul AS-regular algebra of dimension $d+2$ by Theorem \ref{prop: invertible iff double ore extension} and Corollary \ref{cor: double ore extension preserve Koszul}. Let  $\sigma=\varsigma_{\mid V}$ and $(\{\delta_{i,r}\}, \{\delta_{i,l}\},\{\upsilon_{i,r}\},\{\upsilon_{i,l}\})$ be a quadruple of the $\varsigma$-derivation $\nu$ constructed in the last section.

\subsection{Skew divergence}
Since $\dim W_d=1$ with a basis $\omega$, there exists a unique pair $(\delta_{r}=(\delta_{r;j}),\delta_l=(\delta_{l;j}))$ of elements in $V^{\oplus 2}$ with respect to $(\{\delta_{i,r}\}, \{\delta_{i,l}\},\{\upsilon_{i,r}\},\{\upsilon_{i,l}\})$ such that
\begin{align}
  \delta_{d,r}(\omega)=\begin{pmatrix}
    \omega\otimes\delta_{r;1}\\\omega\otimes \delta_{r;2}
\end{pmatrix},\quad
\delta_{d,l}(\omega)=\begin{pmatrix}
    \delta_{l;1}\otimes\omega\\\delta_{l;2}\otimes\omega 
\end{pmatrix}.  
\end{align}

Although quadruple is not unique and the pair $(\delta_{r},\delta_l)$ is not unique for $\nu$, there exists an invariant for $\nu$ arisen by such pairs $(\delta_{r},\delta_l)$ of elements in $V^{\oplus 2}$. 
\begin{proposition}\label{coro: relation between delta_r and delta l}
Let $(\delta_{r}=(\delta_{r;j}),\delta_l=(\delta_{l;j}))$ and $(\delta'_{r}=(\delta'_{r;j}),\delta'_l=(\delta'_{l;j}))$ be two pairs of elements in $V^{\oplus 2}$ with respect to two quadruples $(\{\delta_{i,r}\}, \{\delta_{i,l}\},\{\upsilon_{i,r}\},\{\upsilon_{i,l}\})$ and $(\{\delta'_{i,r}\}, \{\delta'_{i,l}\},\{\upsilon'_{i,r}\},\{\upsilon'_{i,l}\})$ for $\nu$ respectively. Then
$$
\delta_r+{\mu_A^{\oplus2}}(\sigma^{-T}\cdot\delta_l)=\delta'_r+{\mu_A^{\oplus2}}(\sigma^{-T}\cdot\delta'_l).
$$
\end{proposition}
\begin{proof} Write $\widetilde{\delta}_{r}=(\widetilde{\delta}_{r;j})=\delta_{r}-\delta'_r$ and $\widetilde{\delta}_{l}=(\widetilde{\delta}_{l;j})=\mu_A^{\oplus 2}\left(\sigma^{-T}\cdot(\delta_{l}-\delta'_l)\right)$. We have
\begin{align*}
\begin{pmatrix}
  \omega\otimes   \widetilde{\delta}_{l;1}\\
  \omega\otimes \widetilde{\delta}_{l;2}
\end{pmatrix}=&
\left(\tau_{d+1}^{d}(\mu_A\otimes\id^{\otimes d})\right)^{\oplus 2}\left(\left(\sigma^{-T}\boxtimes \mathcal{I}^d\right) \cdot\left((\delta_{d,l}-\delta'_{d,l})(\omega) \right)\right)\\
=&(-1)^{d-1}\left(\tau_{d+1}^{d}(\mu_A\otimes\id^{\otimes d})\right)^{\oplus 2}\left(\left(\sigma^{-T}\boxtimes\mathcal{I}^d\right) \cdot\left( \sum_{j=1}^{d-1}(-1)^{j}\left(\sigma\boxtimes(\delta_{d-j,r}-\delta'
_{d-j,r})\boxtimes \id_V^{\boxtimes j-1}\right)(\omega)\right)\right)\\
=&(-1)^{d-1}\left(\tau_{d+1}^{d}(\mu_A\otimes\id^{\otimes d})\right)^{\oplus 2}\left( \sum_{j=1}^{d-1}(-1)^{j}\left(\mathcal{I}^1\boxtimes(\delta_{d-j,r}-\delta'
_{d-j,r})\boxtimes \id_V^{\boxtimes j-1}\right)(\omega)\right)\\
=&\left(\sum_{j=1}^{d-1}(-1)^{j}(\delta_{d-j,r}-\delta'
_{d-j,r})\boxtimes \id_V^{\boxtimes j}\right)
\left((-1)^{d-1}\tau_d^{d-1}(\mu_A\otimes\id_V^{\otimes d-1})(\omega)\right)
\\
=&-(\delta_{d,r}-\delta'_{d,r})
(\omega)\\
=&-\begin{pmatrix}
  \omega\otimes   \widetilde{\delta}_{r;1}\\
  \omega\otimes \widetilde{\delta}_{r;2}
\end{pmatrix},
\end{align*}
where the second equality holds by \eqref{eq: delta r and delta l in W_d+1=0}, the third equality holds by Remark \ref{rem: tensor is not unusual}, the fifth equality holds by \eqref{eq: delta r in W_d+1=0} and Theorem \ref{thm: properties of Koszul regular algebras}(c), and $\mathcal{I}^j=\begin{pmatrix}
    \id^{\otimes j} &\\
    &\id^{\otimes j}
\end{pmatrix}$ for $j\geq 1$.
\end{proof}

It is a generalization of \cite[Corollary 3.1]{SG} which is for skew derivations in graded Ore extensions of Koszul AS-regular algebras. We give the following definition as an analogue to \cite[Definition 3.2]{SG}.

\begin{definition}\label{def: skew divergence}
Let $A=T(V)/(R)$ be a Koszul AS-regular algebra and   $B=A_P[y_1,y_2;\varsigma,\nu]$ be a graded double Ore extension of $A$ with $p_{12}\neq 0$. Let $(\delta_r,\delta_l)$ be a pair of elements in $V^{\oplus 2}$ with respect to some quadruple $(\{\delta_{i,r}\}, \{\delta_{i,l}\},\{\upsilon_{i,r}\},\{\upsilon_{i,l}\})$ for $\nu$. Then the element $\delta_r+{\mu_A^{\oplus2}}(\sigma^{-T}\cdot\delta_l)$ is called the \emph{$\varsigma$-divergence} of $\nu$, denoted by $\mathrm{div}_{\varsigma} \nu$.
\end{definition}

\subsection{Yoneda product} For a fixed quadruple $(\{\delta_{i,r}\}, \{\delta_{i,l}\},\{\upsilon_{i,r}\},\{\upsilon_{i,l}\})$, the minimal free resolution $F_\centerdot$ of $k_B$ in Theorem \ref{thm: minimial free resolution} becomes
{\footnotesize
\begin{equation*}
    \begin{aligned}
0&\xlongrightarrow{}
W_{d}\otimes B(-2)
\xlongrightarrow{
\left(
\begin{array}{cc}
\partial_{d} \\ -f_d 
\end{array}
\right)}
W_{d-1}\otimes B(-2)\oplus (W_d\otimes B(-1))^{\oplus2}
\xlongrightarrow{
\left(
\begin{array}{cc}
 \partial_{d-1} & 0 \\
 -f_{d-1} & -\partial_d^{\oplus2}\\
 h_{d-1}&g_d
 \end{array}
 \right)
 }
 \cdots\cdots\\
&\quad\xlongrightarrow{}
W_1\otimes B(-2)\oplus (W_{2}\otimes B(-1))^{\oplus2}\oplus W_3\otimes B 
\xlongrightarrow{
\left(
\begin{array}{ccc}
\partial_{1}& 0&  0\\ 
-f_1& -\partial_{2}^{\oplus2}&  0\\
h_1 &   g_2&   \partial_3
\end{array}
\right)
}
B(-2)\oplus (W_1\otimes B(-1))^{\oplus2}\oplus W_2\otimes B\\
&\qquad\qquad\qquad\qquad\xlongrightarrow{
\left(
\begin{array}{ccc}
-f_0&  -\partial_1^{\oplus2}&  0\\
0&   g_1&  \partial_2
\end{array}
\right)
}
B(-1)^{\oplus2}\oplus W_1\otimes B
\xlongrightarrow{
\left(
\begin{array}{cc}
g_0& \partial_1
\end{array}
\right)
} 
B\xlongrightarrow{} 0.
\end{aligned}
\end{equation*}
}
Hence, 
\begin{align*}
&E^{1}(B)=\uHom_{B}(B(-1)^{\oplus2}\oplus W_1\otimes B,k)\cong k^{\oplus2}(1)\oplus W_1^*,\\
&E^{d+1}(B)=\uHom_{B}(W_{d-1}\otimes B(-2)\oplus (W_{d}\otimes B(-1))^{\oplus2},k)\cong W_{d-1}^*(2)\oplus W_d^*(1)^{\oplus2},\\
&E^{d+2}(B)=\uHom_{B}(W_{d}\otimes B(-2),k)\cong W_{d}^*(2).
\end{align*}
Then 
\begin{itemize}
    \item $E^{1}(B)$ has a basis $\xi_1,\xi_2,x_1^*,x_2^*,\cdots,x^*_n$, where 
    $\{\xi_1,\xi_2\}$ is the canonical basis of $k^{\oplus 2}(1)$. We also identify 
    
$\xi_1$ and $\xi_2$ with $(\varepsilon_B,0)$ and $(0,\varepsilon_B)$ from $B(-1)^{\oplus 2}$ to $k$ respectively, and 

$x_i^*$ with the map $x^*_i\otimes\varepsilon_B:W_1\otimes B\to k$ for $i=1,\cdots,n$;

\item $E^{d+1}(B)$ has a basis $\eta_1^*,\eta_2^*,\cdots,\eta^*_n,\omega_1^*,\omega_2^*$.  We also identify

$\eta_i^*$ with the map $\eta^*_i\otimes\varepsilon_B:W_{d-1}\otimes B(-2)\to k$ for $i=1,\cdots,n$, and

$\omega_1^*$ and $\omega_2^*$ with $(\omega^*\otimes\varepsilon_B,0)$ and $(0,\omega^*\otimes\varepsilon_B)$ from $(W_d\otimes B(-1))^{\oplus 2}$ to $k$ respectively;

\item $E^{d+2}(B)$ has a basis $\widetilde{\omega}^*$. We also identify $\widetilde{\omega}^*$ with the map $\omega^*\otimes\varepsilon_B:W_d\otimes B(-2)\to k$. 
\end{itemize}

Before computing the Yoneda product of $E(B)$, we show that the elements $\xi_1$ and $\xi_2$ in $E^1(B)$ are also dual of $y_1,y_2$. 

Recall that $J=k\langle y_1,y_2\rangle /(y_2y_1-p_{12}y_1y_2-p_{11}y_1^2)$ and there is a surjective graded algebra homomorphism
$p_J:B\to J$ (see \eqref{eq: algebra homomorphism B to J}). By \cite[Theorem 1]{SWZ}, we have an injective graded algebra homomorphism $E(p_J):E(J)\to E(B)$. Clearly, $k_J$ has a minimal resolution as below
$$
0\to k\{y_2y_1-p_{12}y_1y_2-p_{11}y_1^2\}\otimes J\xlongrightarrow{\partial^J_2} 
k\{y_1,y_2\}\otimes J\xlongrightarrow{\partial^J_1} 
J\xlongrightarrow{}0,
$$
where $k\{y_2y_1-p_{12}y_1y_2-p_{11}y_1^2\}$ and $k\{y_1,y_2\}$ are vector spaces spanned by $y_2y_1-p_{12}y_1y_2-p_{11}y_1^2$ and $y_1,y_2$ respectively, and 
\begin{align*}
    &\partial^J_1(y_i\otimes 1)=y_i,\qquad i=1,2,\\
    &\partial^J_2((y_2y_1-p_{12}y_1y_2-p_{11}y_1^2)\otimes 1)=y_2\otimes y_1-p_{12}y_1\otimes y_2-p_{11}y_1\otimes y_1.
\end{align*}
So $E^1(J)\cong k\{y_1,y_2\}^*$ with a basis $\{y_1^*,y_2^*\}$.

\begin{lemma}
For any $i=1,2,$ $E(p_J)(y_i^*)=\xi_i$.
\end{lemma}
\begin{proof}
It is straightforward to check that the following diagram is commutative
{\footnotesize
\begin{align*}
\hspace{-1cm} \xymatrix@C=65pt{
 (W_1\otimes B(-2)\oplus (W_{2}\otimes B(-1))^{\oplus2}\oplus W_3\otimes B \ar[r]^{\begin{pmatrix}
\partial_{1}& 0&  0\\ 
-f_1& -\partial_{2}^{\oplus2}&  0\\
h_1 &   g_2&   \partial_3
\end{pmatrix}}\ar[d]^{0}
&
B(-2)\oplus (W_1\otimes B(-1))^{\oplus2}\oplus W_2\otimes B\ar[r]^(0.7){
\begin{pmatrix}
-f_0&  -\partial_1^{\oplus2}&  0\\
0&   g_1&  \partial_2
\end{pmatrix}
}\ar[d]^{\mathscr{H}_2}&\\
 ( 0\ar[r] & k\{y_2y_1-p_{12}y_1y_2-p_{11}y_1^2\}\otimes J\ar[r]^(0.63){\partial^J_2}&~\\ 
}\\
\xymatrix@C=37pt{
B(-1)^{\oplus2}\oplus W_1\otimes B\ar[d]^{\mathscr{H}_1}
\ar[r]^(0.65){
\begin{pmatrix}
g_0& \partial_1
\end{pmatrix}
} &
B\ar[r]\ar[d]^{p_J}& 0)\quad\ar[r]^{\varepsilon_B}\ar[d]^0& \\
k\{y_1,y_2\}\otimes J\ar[r]^{\partial^J_1} &
J\ar[r]&0)\quad\ar[r]^{\varepsilon_J}&
}
\xymatrix{
k_B\ar@{=}[d]\\
k_J,
}
\end{align*}
}
where for any $b,b_1,b_2,b_3\in B$, $\alpha,\alpha_1,\alpha_2\in W_1$ and $\beta\in W_2$, 
\begin{align*}
\mathscr{H}_1
\begin{pmatrix}
    \begin{pmatrix}
        b_1\\
        b_2
    \end{pmatrix}\\
    \alpha\otimes b
\end{pmatrix}
=y_1\otimes p_J(b_1)+  y_2\otimes p_J(b_2), \qquad 
\mathscr{H}_2 
\begin{pmatrix}
    b\\
    \begin{pmatrix}
       \alpha_1\otimes b_1\\
        \alpha_2\otimes b_2
    \end{pmatrix}\\
    \beta\otimes b_3
\end{pmatrix}
=(y_2y_1-p_{12}y_1y_2-p_{11}y_1^2)\otimes p_J(b).
\end{align*}
Then
\begin{align*}
    E(p_Y)(y_1^*)\begin{pmatrix}
    \begin{pmatrix}
        b_1\\
        b_2
    \end{pmatrix}\\
    \alpha\otimes b
\end{pmatrix}&=(y_1^*\otimes\varepsilon_J)\mathscr{H}_1 \begin{pmatrix}\begin{pmatrix}
        b_1\\
        b_2
    \end{pmatrix}\\
    \alpha\otimes b
\end{pmatrix}=\varepsilon_Jp_J(b_1)=\left((\varepsilon_B,0),0\right)\begin{pmatrix}\begin{pmatrix}
        b_1\\
        b_2
    \end{pmatrix}\\
    \alpha\otimes b
\end{pmatrix},\\
   E(p_Y)(y_2^*)\begin{pmatrix}
    \begin{pmatrix}
        b_1\\
        b_2
    \end{pmatrix}\\
    \alpha\otimes b
\end{pmatrix}
&=(y_2^*\otimes\varepsilon_J)\mathscr{H}_1 \begin{pmatrix}\begin{pmatrix}
        b_1\\
        b_2
    \end{pmatrix}\\
    \alpha\otimes b
\end{pmatrix}=\varepsilon_Jp_J(b_2)=\left((0,\varepsilon_B),0\right)\begin{pmatrix}\begin{pmatrix}
        b_1\\
        b_2
    \end{pmatrix}\\
    \alpha\otimes b
\end{pmatrix}.
\end{align*}
The result follows.
\end{proof}

Let $(\delta_{r},\delta_l)$ be a pair of elements in $V^{\oplus 2}$ with respect to the fixed quadruple $(\{\delta_{i,r}\}, \{\delta_{i,l}\},\{\upsilon_{i,r}\},\{\upsilon_{i,l}\})$. We have a result for the Yoneda product of $E(B)$ as below.

\begin{lemma}\label{lemma: Yoneda product}
In $E(B)$, for any $i,j=1,\cdots,n$,
\begin{align*}
\begin{pmatrix}
\xi_1\ast\omega_1^* & \xi_2\ast\omega_1^* &x^*_i\ast\omega_1^* \\
\xi_1\ast\omega_2^* & \xi_2\ast\omega_2^* & x^*_i\ast\omega_2^*\\
\xi_1\ast\eta_j^* & \xi_2\ast\eta_j^* & x^*_i\ast\eta_j^*    
\end{pmatrix}&=
(-1)^d
\begin{pmatrix}
\mathbb{J}\hdet\varsigma& \mathbb{J}\left((x_i^*)^{\oplus 2}\delta_{r}\right)\\
0
&-(\eta_j^*\otimes x^*_i)(\omega)
\end{pmatrix}\widetilde{\omega}^*,\\
\begin{pmatrix}
\omega_1^*\ast\xi_1 & \omega_1^*\ast\xi_2& \omega_1^*\ast x^*_i\\
\omega_2^*\ast\xi_1 & \omega_2^*\ast\xi_2 &\omega_2^*\ast x^*_i \\
\eta_j^*\ast\xi_1&\eta_j^*\ast \xi_2&\eta_j^*\ast x^*_i    
\end{pmatrix}&=
\begin{pmatrix}
\mathbb{J}^T& 
(x_i^*)^{\oplus2}\left(\sigma^T\cdot(\mathbb{J}\delta_l)\right)\\
0&(-1)^{d+1}(x^*_i\det\sigma\otimes\eta_j^*)(\omega)
\end{pmatrix}\widetilde{\omega}^*.
\end{align*}
\end{lemma}
The proof of the lemma above is tedious. To make reading more fluent, it will be addressed in Section \ref{sec: proof}.

\subsection{Nakayama automorphisms}
For the linear automorphism $\det \sigma$ of $V$ and the Nakayama automorphism $\mu_A$ of $A$, there exist two invertible $n\times n$ matrices $U$ and $L$ over $k$, such that
$$\det\sigma
\left(
\begin{array}{c}
x_1\\
x_2\\
\vdots\\
x_n
\end{array}
\right)=U\left(
\begin{array}{c}
x_1\\
x_2\\
\vdots\\
x_n
\end{array}
\right),\qquad
\mu_A
\left(
\begin{array}{c}
x_1\\
x_2\\
\vdots\\
x_n
\end{array}
\right)=L\left(
\begin{array}{c}
x_1\\
x_2\\
\vdots\\
x_n
\end{array}
\right).
$$

\begin{lemma}\cite[Lemma 3.8]{SG}\label{lemma: Yoneda product in E(A)}
For any $j=1,\cdots,n$,
$$
\left(
\begin{array}{c}
(\eta_j^*\otimes x^*_1)(\omega)\\
(\eta_j^*\otimes x^*_2)(\omega)\\
\vdots\\
(\eta_j^*\otimes x^*_n)(\omega)
\end{array}
\right)
=(-1)^{d-1}L^T\left(
\begin{array}{c}
 (x^*_1\otimes\eta^*_j)(\omega)\\
(x^*_2\otimes\eta^*_j)(\omega)\\
\vdots\\
(x^*_n\otimes\eta^*_j)(\omega)
\end{array}
\right).
$$
\end{lemma}

We prove one main result of this paper.

\begin{theorem}\label{thm: nakayama automorphism of double Ore extension} Let $B=A_P[y_1,y_2;\varsigma,\nu]$ be a graded double Ore extension of a Koszul AS-regular algebra $A$, where $P=\{p_{12}\neq0,p_{11}\}\subseteq k$, $\varsigma:A\to M_{2\times 2}(A)$ is an invertible graded algebra homomorphism and $\nu:A\to A^{\oplus2}$ is a degree one $\varsigma$-derivation. Then the Nakayama automorphism $\mu_B$ of $B$ satisfies
\begin{align*}
{\mu_{B}}_{\mid A}=(\det\varsigma)^{-1}\mu_A,\qquad
\mu_B\begin{pmatrix}
    y_1\\y_2
\end{pmatrix}
=-(\mathbb{J}^T)^{-1}\mathbb{J}\hdet\varsigma
\begin{pmatrix}
    y_1\\y_2
\end{pmatrix}
-(\mathbb{J}^T)^{-1}\mathbb{J} \mathrm{div}_{\varsigma}\nu.
\end{align*}
\end{theorem}
\begin{proof}
By Theorem \ref{thm: properties of Koszul regular algebras}, $E(B)$ is a graded Frobenius algebra with a Nakayama automorphism $\mu_{E(B)}$. We compute $\mu_{E(B)}(\xi_1),\mu_{E(B)}(\xi_2)$ firstly. By Lemma \ref{lemma: Yoneda product}, we have  
\begin{align*}
\begin{pmatrix}
\langle\xi_1,\omega_1^*\rangle & \langle\xi_2,\omega_1^*\rangle \\
 \langle\xi_1,\omega_2^*\rangle & \langle\xi_2,\omega_2^*\rangle\\
   \langle\xi_1,\eta_j^*\rangle  & \langle\xi_2,\eta_j^*\rangle
\end{pmatrix}=&(-1)^d
\begin{pmatrix}
    \mathbb{J}\hdet\varsigma\\
    0
\end{pmatrix}=(-1)^{d+1}
\begin{pmatrix}
    \mathbb{J}^T\\
    0
\end{pmatrix}\left(-\left(\mathbb{J}^{-1}\right)^T\mathbb{J}\right)\hdet\varsigma\\
=&(-1)^{d+1}
\begin{pmatrix}
    \langle\omega_1^*,\xi_1\rangle &  \langle\omega_1^*,\xi_2\rangle   \\
    \langle\omega_2^*,\xi_1\rangle & \langle\omega_2^*,\xi_2 \rangle \\
     \langle\eta_j^*,\xi_1\rangle &   \langle\eta_j^*,\xi_2\rangle
\end{pmatrix}\left(-\left(\mathbb{J}^{-1}\right)^T\mathbb{J}\right)\hdet\varsigma,
\end{align*}
So
$$\mu_{E(B)}(
    \xi_1\ \xi_2
)=(
    \xi_1\ \xi_2
)\left(-\left(\mathbb{J}^{-1}\right)^T\mathbb{J}\right)\hdet\varsigma.$$

It turns to consider $\mu_{E(B)}(x_i)$ for $i=1,\cdots,n$. By Proposition \ref{prop: invertible means T-invertible} and Lemma \ref{lemma: Yoneda product}, one obtains that 
$$
\begin{pmatrix}
\langle x_1^*, \omega_1^*\rangle &   \langle x_2^*,\omega_1^*\rangle &\cdots&\langle x_n^*,\omega_1^*\rangle \\
\langle x_1^*,\omega_2^*\rangle &   \langle x_2^*,\omega_2^*\rangle &\cdots&\langle x_n^*,\omega_2^*\rangle     
\end{pmatrix}=
(-1)^d\mathbb{J}
\begin{pmatrix}
(x_1^*)^{\oplus 2}(\delta_r)    &    (x_2^*)^{\oplus 2}(\delta_r)   &\cdots&    (x_n^*)^{\oplus 2}(\delta_r)   
\end{pmatrix},
$$
and
\begin{align*}
&\begin{pmatrix}
\langle \omega_1^*,x_1^*\rangle &   \langle \omega_1^*,x_2^*\rangle &\cdots&\langle \omega_1^*,x_n^*\rangle \\
\langle \omega_2^*,x_1^*\rangle &   \langle \omega_2^*,x_2^*\rangle &\cdots&\langle \omega_2^*,x_n^*\rangle     
\end{pmatrix}\\
=&
\left(
(x_1^*)^{\oplus2}\left(\sigma^T\cdot(\mathbb{J}\delta_l)\right) \ \  (x_2^*)^{\oplus2}\left(\sigma^T\cdot(\mathbb{J}\delta_l)\right)\ \ 
\cdots \ \ 
(x_n^*)^{\oplus2}\left(\sigma^T\cdot(\mathbb{J}\delta_l)\right)
\right)\\
=&\mathbb{J}
\left(
(x_1^*\det\sigma)^{\oplus2}(\sigma^{-T}\cdot\delta_l) \ \  (x_2^*\det\sigma)^{\oplus2}(\sigma^{-T}\cdot\delta_l))\ \ 
\cdots \ \ 
(x_n^*\det\sigma)^{\oplus2}(\sigma^{-T}\cdot\delta_l)
\right)\\
=&\mathbb{J}
\left(
(x_1^*)^{\oplus2}(\sigma^{-T}\cdot\delta_l) \ \  (x_2^*)^{\oplus2}(\sigma^{-T}\cdot\delta_l))\ \ 
\cdots \ \ 
(x_n^*)^{\oplus2}(\sigma^{-T}\cdot\delta_l)
\right)U.
\end{align*}

Write
$
\begin{pmatrix}
c_{11}  &   c_{21} &    \cdots  &   c_{n1}\\
c_{12}  &   c_{22} &    \cdots  &   c_{n2}
\end{pmatrix}=\left(
(x_1^*)^{\oplus 2}(\mathrm{div}_{\varsigma}\nu) \ \  (x_2^*)^{\oplus 2}(\mathrm{div}_{\varsigma}\nu)\ \ \cdots\ \  (x_n^*)^{\oplus 2}(\mathrm{div}_{\varsigma}\nu)
\right),
$ that is,
\begin{align*}
\begin{pmatrix}
c_{11}  &   c_{21} &    \cdots  &   c_{n1}\\
c_{12}  &   c_{22} &    \cdots  &   c_{n2}
\end{pmatrix}=&
\left(
   (x_1^*)^{\oplus 2}(\delta_{r}) \ \  (x_2^*)^{\oplus 2}(\delta_{r}) \ \  \cdots  \ \  (x_n^*)^{\oplus 2}(\delta_{r}) 
\right)\\
&+
\begin{pmatrix}
   (x_1^*)^{\oplus 2}(\sigma^{-T}\cdot\delta_{l})   \ \     (x_2^*)^{\oplus 2}(\sigma^{-T}\cdot\delta_{l}) \ \   \cdots \ \  (x_n^*)^{\oplus 2}(\sigma^{-T}\cdot\delta_{l})
\end{pmatrix}L.
\end{align*}

By Lemma \ref{lemma: Yoneda product}, we have
\begin{align*}
\begin{pmatrix}
\langle x_1^*, \omega_1^*\rangle &   \langle x_2^*,\omega_1^*\rangle &\cdots&\langle x_n^*,\omega_1^*\rangle \\
\langle x_1^*,\omega_2^*\rangle &   \langle x_2^*,\omega_2^*\rangle &\cdots&\langle x_n^*,\omega_2^*\rangle     
\end{pmatrix}
=&
(-1)^{d+1}\begin{pmatrix}
\langle \omega_1^*,x_1^*\rangle &   \langle \omega_1^*,x_2^*\rangle &\cdots&\langle \omega_1^*,x_n^*\rangle \\
\langle \omega_2^*,x_1^*\rangle &   \langle \omega_2^*,x_2^*\rangle &\cdots&\langle \omega_2^*,x_n^*\rangle     
\end{pmatrix}U^{-1}L\\
&+(-1)^d\mathbb{J}^T(\mathbb{J}^T)^{-1}\mathbb{J}\begin{pmatrix}
c_{11}  &   c_{21} &    \cdots  &   c_{n1}\\
c_{12}  &   c_{22} &    \cdots  &   c_{n2}
\end{pmatrix}\\
=&(-1)^{d+1}\begin{pmatrix}
\langle \omega_1^*,x_1^*\rangle &   \langle \omega_1^*,x_2^*\rangle &\cdots&\langle \omega_1^*,x_n^*\rangle \\
\langle \omega_2^*,x_1^*\rangle &   \langle \omega_2^*,x_2^*\rangle &\cdots&\langle \omega_2^*,x_n^*\rangle     
\end{pmatrix}U^{-1}L\\
&+(-1)^{d+1}\begin{pmatrix}
     \langle \omega_1^*,\xi_1\rangle & \langle \omega_1^*,\xi_2\rangle \\
      \langle \omega_2^*,\xi_1\rangle & \langle \omega_2^*,\xi_2\rangle 
 \end{pmatrix}(-(\mathbb{J}^T)^{-1}\mathbb{J}) \begin{pmatrix}
c_{11}  &   c_{21} &    \cdots  &   c_{n1}\\
c_{12}  &   c_{22} &    \cdots  &   c_{n2}
\end{pmatrix}.
\end{align*}
Using Lemma \ref{lemma: Yoneda product}, it is not hard to obtain that
\begin{align*}
\begin{pmatrix}
\langle x_1^*, \eta_j^*\rangle &   \langle x_2^*,\eta_j^*\rangle &\cdots&\langle x_n^*,\eta_j^*\rangle   
\end{pmatrix}
=&
(-1)^{d+1}\begin{pmatrix}
\langle \eta_j^*,x_1^*\rangle &   \langle \eta_j^*,x_2^*\rangle &\cdots&\langle \eta_j^*,x_n^*\rangle 
\end{pmatrix}U^{-1}L\\
&+(-1)^{d+1}\begin{pmatrix}
     \langle \eta_j^*,\xi_1\rangle & \langle \eta_j^*,\xi_2\rangle
 \end{pmatrix}(-(\mathbb{J}^T)^{-1}\mathbb{J})   \begin{pmatrix}
c_{11}  &   c_{21} &    \cdots  &   c_{n1}\\
c_{12}  &   c_{22} &    \cdots  &   c_{n2}
\end{pmatrix}.
\end{align*}

Hence,
$$
\mu_{E(B)}(x_1^*,x_2^*,\cdots,x_n^*,\xi_1,\xi_2)=(x_1^*,x_2^*,\cdots,x_n^*,\xi_1,\xi_2)
\begin{pmatrix}
    U^{-1}L &  0 \\
    -\left(\mathbb{J}^T\right)^{-1}\mathbb{J}\begin{pmatrix}
        c_{11}&\cdots&c_{n1}\\
        c_{12}&\cdots&c_{n2}
        \end{pmatrix} & -\left(\mathbb{J}^T\right)^{-1}\mathbb{J}\hdet\varsigma
\end{pmatrix}.
$$
By Theorem \ref{thm: properties of Koszul regular algebras}, we have 
$$
\mu_{B}\begin{pmatrix}
   x_1\\
    x_2\\
    \vdots\\
    x_n\\
    y_1\\
    y_2
\end{pmatrix}=
\begin{pmatrix}
    U^{-1}L &  0 \\
    -\left(\mathbb{J}^T\right)^{-1}\mathbb{J}\begin{pmatrix}
        c_{11}&\cdots&c_{n1}\\
        c_{12}&\cdots&c_{n2}
        \end{pmatrix} & -\left(\mathbb{J}^T\right)^{-1}\mathbb{J}\hdet\varsigma
\end{pmatrix}
\begin{pmatrix}
   x_1\\
    x_2\\
    \vdots\\
    x_n\\
    y_1\\
    y_2
\end{pmatrix}.
$$
The proof is completed.
\end{proof}

\begin{remark}
The matrix 
$$
-\left(\mathbb{J}^T\right)^{-1}\mathbb{J}=
\begin{pmatrix}
    p_{12}^{-1} & 0\\
    p_{11}(1+p_{12}^{-1})& p_{12}
\end{pmatrix},
$$
appearing in the Nakayama automorphism $\mu_B$ of graded double Ore extension $B=A_P[y_1,y_2;\sigma,\delta]$, is exactly the one corresponding to the Nakayama automorphism of $J=k\langle Y\rangle/(Y^T\mathbb{J}Y)$.
\end{remark}

\subsection{Twisted superpotential}

By Corollary \ref{cor: double ore extension preserve Koszul}, Proposition \ref{prop: invertible iff double ore extension} and Theorem \ref{thm: properties of Koszul regular algebras}, graded double Ore extension $B=A_P[y_1,y_2;\varsigma,\nu]$ of Koszul AS-regular algebra $A$ is defined by a twisted superpotential. In this subsection, we describe such twisted superpotentials precisely. It generalizes the result for graded Ore extensions of Koszul AS-regular algebras (see \cite[Theorem 3.11]{SG}), and it is much more complicated than the one for graded Ore extensions.

\begin{theorem}\label{thm: twisted superpotential}
Let $A=T(V)/(R)$ be a Koszul AS-regular algebra of dimension $d$ and $\omega$ be a basis of $W_d$. Let $B=A_P[y_1,y_2;\varsigma,\nu]$ be a graded double Ore extension of $A$ with the Nakayama automorphism $\mu_B$, where $P=\{p_{12}\neq0,p_{11}\}\subseteq k$, $\varsigma:A\to M_{2\times 2}(A)$ is an invertible graded algebra homomorphism and $\nu:A\to A^{\oplus2}$ is a degree one $\varsigma$-derivation. Suppose $(\{\delta_{i,r}\}, \{\delta_{i,l}\},\{\upsilon_{i,r}\},\{\upsilon_{i,l}\})$ is a quadruple for $\nu$,
then
    \begin{align*}
\hat{\omega}=&\sum_{i=0}^{d}\sum_{j=0}^{d-i}(-1)^j\left((\det\sigma)^{\otimes i}\otimes\left( \lambda^T_Y\boxtimes(\mathbb{J}\bullet\sigma^{\boxtimes j})\boxtimes\lambda_Y\right)\otimes\id^{\otimes d-i-j}\right) \left(\tau^{i}_{d+2}(\id\otimes\tau^{i+j}_{d+1})\right)(1\otimes 1\otimes\omega)\\
&+
\sum_{j=0}^{d+1}\sum_{i=1}^{d}(-1)^{i+j} \tau^{j}_{d+2}\left(\lambda^T_Y\boxtimes\left((\sigma^{\boxtimes j}\boxtimes
\mathcal{I}^{d+1-j})^T\bullet\mathbb{J}\bullet(\delta_{i,r}\boxtimes \id^{\boxtimes d-i})\right)\right) (1\otimes\omega)\\
&+\left(\sum_{i=1}^{d}\sum_{j=1}^i(-1)^{i+j}\underline{\delta_{j,r}\boxtimes\id^{\boxtimes d+1-j}}\left(\mathbb{J}\bullet\left(\delta_{i,r}\boxtimes \id^{\boxtimes d-i}\right)\right)-\sum_{i=1}^{d-1}\left(\upsilon_{i,r}\otimes \id^{\otimes d-i}\right)\right)(\omega),
\end{align*}
is a ${\mu_B}_{\mid \hat{V}}$-twisted superpotential, 
where $\mathcal{I}^{d+1-j}=\begin{pmatrix}
    \id^{\otimes d+1-j} & \\
    &\id^{\otimes d+1-j}
\end{pmatrix}$, $\sigma={\varsigma}_{\mid V}$ and $\hat{V}=V\oplus k\{y_1,y_2\}$.  Moreover,
$$
B\cong  \mathcal{D}(\hat{\omega},d).
$$

\end{theorem}

One can also obtain many other representations of $\hat{\omega}$ by Proposition \ref{prop: relation between delta_r and delta_l} and Lemma \ref{lem: realtions between Gamma}. The proof of the theorem above will be also addressed in Section \ref{sec: proof} which is complicated and long. 

\begin{remark}
We give an explanation to understand the construction of $\hat{\omega}$. As mentioned earlier, the graded double Ore extension  $B=A_P[y_1,y_2;\varsigma,\nu]$ is a twisted tensor product of $A$ and $J=k\langle y_1,y_2\rangle /(y_2y_1-p_{12}y_1y_2-p_{11}y_1^2)$. Note that $J$ is also a Koszul AS-regular algebra, and so it has a twisted superpotential $y_2y_1-p_{12}y_1y_2-p_{11}y_1^2$. The prototype of $\hat{\omega}$ is just the element $\left(y_2y_1-p_{12}y_1y_2-p_{11}y_1^2\right)\otimes\omega$ where $\omega$ is a twisted superpotential for $A$, and then $\hat{\omega}$ is obtained by adding all terms by exchanging elements between $\{y_1,y_2\}$ and the elements in $\omega$ with respect to the relations in $B$ making use of $\sigma,\{\delta_{i,r}\}$ and $\{\upsilon_{i,r}\}$.
\end{remark} 

We close this section by a continuous work of Example \ref{exa: example 1}, which may help us to understand Theorem \ref{thm: nakayama automorphism of double Ore extension} and Theorem \ref{thm: twisted superpotential}.

\begin{example}
We compute the Nakayama automorphism of $B$ occurring in Example \ref{exa: example 1} first. The Nakayama automorphism $\mu_A$ of $A$ is 
$$
\mu_A
\begin{pmatrix}
    x_1\\
    x_2
\end{pmatrix}
=
\begin{pmatrix}
    p^{-1}&0\\
    0&p
\end{pmatrix}
\begin{pmatrix}
    x_1\\
    x_2
\end{pmatrix}.
$$

Clearly, $\omega=x_2\otimes x_1-px_1\otimes x_2$ is a basis of  $W_2$, and is also a ${\mu_A}_{\mid V}$-twisted. By Definition \ref{def: determinant of varsigma}, Theorem \ref{thm: hdet is sigma d}, Definition \ref{def: skew divergence} and \eqref{eq: t-inverse of sigma}, one obtains that 
$$
\begin{array}{ll}
\begin{aligned}
\det\varsigma(x_1)&=-px_1,\\
\hdet(\varsigma)&=\begin{pmatrix}
    p & 0\\
    0&-p
\end{pmatrix},
\end{aligned}&
\begin{aligned}
    \det\varsigma(x_2)&=px_2,  \\
\mathrm{div}_{\varsigma}\nu&=\begin{pmatrix}
0\\
-px_2
\end{pmatrix}+\mu_A^{\oplus}\sigma^{-T}\cdot\begin{pmatrix}
px_1\\
0
\end{pmatrix}=\begin{pmatrix}
0\\
0
\end{pmatrix}.
\end{aligned}
\end{array}
$$
By Theorem \ref{thm: nakayama automorphism of double Ore extension},
$$
\mu_B=\id.
$$
That is, $B$ is a graded Calabi-Yau algebra. By Theorem \ref{thm: twisted superpotential}, we have an $\id$-twisted superpotential (i.e., superpotential) $\hat{\omega}=\hat{\omega}_1+\hat{\omega}_2+\hat{\omega}_3$ for $B$, where 
\begin{align*}
\hat{\omega}_1=&\sum_{i=0}^2\sum_{j=0}^{2-i}(-1)^j\left((\det\sigma)^{\otimes i}\otimes \left( \lambda^T_Y\boxtimes(\mathbb{J}\bullet\sigma^{\boxtimes j})\boxtimes\lambda_Y\right)\otimes\id^{\otimes 2-i-j}\right) \left(\tau^{i}_{4}(\id\otimes\tau^{i+j}_{3})\right)(1\otimes 1\otimes\omega)\\
=&(y_2\otimes y_1-py_1\otimes y_2)\otimes (x_2\otimes x_1-px_1\otimes x_2)\\
&-y_2\otimes (x_2\otimes y_2\otimes x_2+px_1\otimes y_2\otimes x_1)-y_1\otimes(x_1\otimes y_1\otimes x_1+px_2\otimes y_1\otimes x_2)\\
&+y_1\otimes(px_1\otimes x_2-x_2\otimes x_1)\otimes y_2+y_2\otimes (px_2\otimes x_1+x_1\otimes x_2)\otimes y_1
\\
&+px_2\otimes (y_2\otimes y_1-py_1\otimes y_2)\otimes x_1-x_1\otimes (y_2\otimes y_1-py_1\otimes y_2)\otimes x_2\\
&+x_2\otimes (py_1\otimes x_2\otimes y_1+y_2\otimes x_2\otimes y_2)+x_1\otimes (y_1\otimes x_1\otimes y_1+py_2\otimes x_1\otimes y_2)\\
&+(x_2\otimes x_1-px_1\otimes x_2)\otimes (y_2\otimes y_1-py_1\otimes y_2),\\
\hat{\omega}_2=&\sum_{j=0}^{3}\sum_{i=1}^{2}(-1)^{i+j} \tau^{j}_{4}\left(\lambda^T_Y\boxtimes\left((\sigma^{\boxtimes j}\boxtimes
\mathcal{I}^{3-j})^T\bullet\mathbb{J}\bullet(\delta_{i,r}\boxtimes \id^{\boxtimes 2-i})\right)\right) (1\otimes\omega)\\
=&-py_2\otimes x_1\otimes x_2\otimes x_1-y_1\otimes  x_2\otimes x_1\otimes x_2+x_2\otimes y_1\otimes x_2\otimes x_1+px_1\otimes y_2\otimes x_1\otimes x_2\\
&-x_1\otimes x_2\otimes y_1\otimes x_2-px_2\otimes x_1\otimes y_2\otimes x_1+x_2\otimes x_1\otimes x_2\otimes y_1+px_1\otimes x_2\otimes x_1\otimes y_2,
\\
\hat{\omega}_3=&\left(\sum_{i=1}^{2}\sum_{j=1}^i(-1)^{i+j}\underline{\delta_{j,r}\boxtimes\id^{\boxtimes 3-j}}\left(\mathbb{J}\bullet\left(\delta_{i,r}\boxtimes \id_V^{\boxtimes 2-i}\right)\right)-\upsilon_{1,r}\otimes \id\right)(\omega)\\
=&x_1\otimes x_2\otimes x_2\otimes x_1+px_1\otimes x_2\otimes x_1\otimes x_2-p(x_2\otimes x_1-px_1\otimes x_2)\otimes x_2\otimes x_1\\
=&px_1\otimes x_2\otimes x_1\otimes x_2-px_2\otimes x_1\otimes x_2\otimes x_1.
\end{align*}

\end{example}

\section{Proof of Lemma \ref{lemma: Yoneda product} and Theorem \ref{thm: twisted superpotential}}\label{sec: proof}
In order to prove Lemma \ref{lemma: Yoneda product}, we show two commutative diagrams first.  

\begin{lemma}\label{lem: E^1(B) E^{d+1}(B)}
For any $j=1,\cdots,n$ and $a_0,a_1,a_2\in k$, the following diagram is commutative:
$$
\xymatrix{
 W_d\otimes B(-2)\ar[r]^(0.38){\begin{pmatrix} \partial_d \\ { -f_d}\end{pmatrix}}\ar[d]^{\varphi_{2;a_0a_1a_2}}
 &
 W_{d-1}\otimes B(-2)\bigoplus (W_d\otimes B(-1))^{\oplus2}\ar[d]^{\varphi_{1;a_0a_1a_2}} \ar[rd]^(0.6){\qquad \left(a_0\eta_j^*\otimes\varepsilon_B,\  \left(a_1\omega^*\otimes\varepsilon_B, \ a_2\omega^*\otimes\varepsilon_B\right)\right)}\\
B(-d)^{\oplus2}\bigoplus W_1\otimes B(-d-1)\ar[r]^(0.63){(-1)^{d+1}\left( g_0,\   \partial_1\right)}
&
B(-d-1)\ar[r]^{\varepsilon_B}&
k_B(-d-1),
}
$$
where 
\begin{align*}
\varphi_{1;a_0a_1a_2}&=\left(a_0\eta_j^*\otimes\id_B,\ \left(a_1\omega^*\otimes\id_B,a_2\omega^*\otimes\id_B\right)\right),\\
\varphi_{2;a_0a_1a_2}&=(-1)^{d+1}a_0\begin{pmatrix}
0\\
(\eta_j^*\otimes \id_V)\otimes \id_B
\end{pmatrix}+
(-1)^d
\begin{pmatrix}
(\mathbb{J}\hdet\varsigma)^T\bullet\begin{pmatrix}
a_1\omega^*\otimes \id_B\\a_2\omega^*\otimes \id_B
\end{pmatrix}\\
\left((\omega^*\otimes\id_V)\left((a_1,a_2)\bullet
\mathbb{J}\bullet\delta_{d,r}\right)\right)\otimes \id_B
\end{pmatrix}.
\end{align*}
\end{lemma}
\begin{proof}
Obviously,  $\varepsilon_B\varphi_{1;a_0a_1a_2}=\left(a_0\eta_j^*\otimes\varepsilon_B,\ \left(a_1\omega^*\otimes\varepsilon_B,a_2\omega^*\otimes\varepsilon_B\right)\right)$ and
\begin{align*}
f_d
=\mathbb{J}\bullet\left(\sigma^{\boxtimes d}\boxtimes \lambda_Y
+({\id}_V^{\otimes d}\otimes m_B)^{\oplus2}(\delta_{d,r}\boxtimes \id_B)
\right)=(\mathbb{J}\hdet \varsigma)\boxtimes \lambda_Y+  (\id^{\otimes d}_V\otimes m_B)^{\oplus 2}((\mathbb{J}\bullet\delta_{d,r})\boxtimes \id_B),
\end{align*} 
by Theorem \ref{thm: hdet is sigma d}. Then
\begin{align*}
\varphi_{1;a_0a_1a_2}\begin{pmatrix}
    \partial_d\\-f_d
\end{pmatrix}=&a_0\eta_j^*\otimes m_B-\left((a_1\omega^*,a_2\omega^*)\bullet(\mathbb{J}\hdet \varsigma)\right) \boxtimes \lambda_Y-(a_1\omega^*\otimes m_B,a_2\omega^*\otimes m_B)\bullet((\mathbb{J}\bullet\delta_{d,r})\boxtimes \id_B)\\
=&a_0\eta_j^*\otimes m_B-\left((\mathbb{J}\hdet \varsigma)^T\bullet\begin{pmatrix}a_1\omega^*\\a_2\omega^*\end{pmatrix}\right)^T \boxtimes \lambda_Y-m_B\left((\omega^*\otimes\id_V)\left((a_1,a_2)\bullet
\mathbb{J}\bullet\delta_{d,r}\right)\otimes \id_B\right)\\
=&a_0\eta_j^*\otimes m_B-\underline{\lambda_Y}\left(
(\mathbb{J}\hdet\varsigma)^T\bullet\begin{pmatrix}
a_1\omega^*\otimes \id_B\\a_2\omega^*\otimes \id_B
\end{pmatrix}\right)
-m_B\left((\omega^*\otimes\id_V)\left((a_1,a_2)\bullet
\mathbb{J}\bullet\delta_{d,r}\right)\otimes \id_B\right)\\
 =&(-1)^{d+1}(g_0,\ \partial_1)\varphi_{2;a_0a_1a_2}.
\end{align*}
The proof is completed.
\end{proof}

\begin{lemma}\label{lem: E{d+1}(B) ast E1(B)}
For any $a_0,a_1,a_2\in k$ and  $i=1,\cdots,n$, the following diagram is commutative.
{\footnotesize
$$
\xymatrix{
W_d\otimes B(-2) \ar[rr]^(0.4){ 
\left(
{\begin{array}{cc} 
\partial_d \\ -f_d
\end{array}
}
\right)}
\ar[d]^{\psi_{d+1;a_0a_1a_2}}
&& W_{d-1}\otimes B(-2)\bigoplus (W_d\otimes B(-1))^{\oplus2}\ar[rr]^(0.65){
   \begin{pmatrix}
       \partial_{d-1}&0\\
       -f_{d-1}&-\partial_d^{\oplus2}\\
       h_{d-1}&g_d
   \end{pmatrix}
}\ar[d]^{\psi_{d;a_0a_1a_2}}
&&\cdots
\\
W_{d-1}\otimes B(-3)\bigoplus (W_d\otimes B(-2))^{\oplus2}\ar[rr]^(0.43){
    \begin{pmatrix}
        -\partial_{d-1}&0\\
        f_{d-1}&\partial_d^{\oplus2}\\
        -h_{d-1}&-g_d
    \end{pmatrix}
    }
&&
W_{d-2}\otimes B(-3)\bigoplus (W_{d-1}\otimes B(-2))^{\oplus2}\bigoplus W_d\otimes B(-1)
 \ar[rr]^(0.85){
    \begin{pmatrix}
         -\partial_{d-2}&0&0\\
     f_{d-2}&\partial_{d-1}^{\oplus2}&0\\
         -h_{d-2}&-g_{d-1}&-\partial_d
    \end{pmatrix}}
&&\cdots
}
$$
$$
\qquad\qquad\qquad\xymatrix{
\cdots\ar[rr]^(0.3){
    \begin{pmatrix}
        \partial_1&0&0\\
        -f_1&-\partial_2^{\oplus2}&0\\
        h_1&g_2&\partial_3
    \end{pmatrix}
}
&&B(-2)\bigoplus(W_1\otimes B(-1))^{\oplus2}\bigoplus W_{2}\otimes B\ar[d]^{\psi_{1;a_0a_1a_2}}\ar[rr]^(0.6){
\begin{pmatrix}
-f_0&-\partial_1^{\oplus2}&0\\0&g_1&\partial_2
\end{pmatrix}
}
&&B(-1)^{\oplus2}\bigoplus W_1\otimes B\ar[d]^{\psi_{0;a_0a_1a_2}}\ar[rd]^{\ \ \ ((a_1\varepsilon_B,\ a_2\varepsilon_B),\ a_0x_i^*\otimes \varepsilon_B)}\\
\cdots\ar[rr]^(0.3){
    \begin{pmatrix}
        f_0&\partial_1^{\oplus2}&0\\
        0&-g_1&-\partial_2
    \end{pmatrix}}
&&B(-2)^{\oplus 2}\bigoplus W_1\otimes B(-1)\ar[rr]^(0.6){
    \begin{pmatrix}
        -g_0&-\partial_1
    \end{pmatrix}
}
&&B(-1)\ar[r]^{\varepsilon_{B}}
&
k_B(-1),
}
$$
}
where
{\scriptsize
\begin{align*}
\psi_{0;a_0a_1a_2}=&((a_1,\ a_2)\cdot-,\ 0)+(0,\ a_0x_i^*\otimes\id_B),\\
\psi_{1;a_0a_1a_2}=&\begin{pmatrix}
        \mathbb{J}^T\begin{pmatrix}a_1\\a_2\end{pmatrix}&0&0\\
        0&(a_1,a_2)\cdot-
        & 0
        \end{pmatrix}
        +a_0\begin{pmatrix}
        0&-\Lambda_1\cdot-&0\\
        0&
        -(x_i^*\otimes\id_V\otimes \id_B)\underline{\delta_{1,l}\boxtimes\id_B}
        & -(x_i^*\otimes \id_V)\otimes\id_B
        \end{pmatrix},\\
    \psi_{s;a_0a_1a_2}=&
    \begin{pmatrix}
    0&0&0\\
        \mathbb{J}^T\begin{pmatrix}a_1\\a_2\end{pmatrix} &0 &0\\
        0 & (a_1,a_2)\cdot-&0
    \end{pmatrix}+a_0{ \begin{pmatrix}
        (-1)^s(x_i^*\det\sigma\otimes\id_V^{\otimes s-2})\otimes\id_B&0&0\\
       \Lambda_{s}\bullet \mathbb{J}\bullet(\delta_{s-1,l}\boxtimes\id_B) &(-1)^s \Lambda_{s}\cdot- &0\\
        -((x_i^*\otimes\id_V^{\otimes s})\upsilon_{s-1,l})\otimes\id_B&-(x_i^*\otimes\id^{\otimes s}_V\otimes\id_B)\underline{\delta_{s,l}\boxtimes\id_B}&
      (-1)^s  (x_i^*\otimes\id_V^{\otimes s})\otimes\id_B
    \end{pmatrix}
    },
    \\
    \psi_{d;a_0a_1a_2}=&
    \begin{pmatrix}
    0&0\\
        \mathbb{J}^T\begin{pmatrix}a_1\\a_2\end{pmatrix}  &0 \\
        0 & (a_1,a_2)\cdot-
    \end{pmatrix}+a_0{ \begin{pmatrix}
        (-1)^d(x_i^*\det\sigma\otimes\id_V^{\otimes d-2})\otimes\id_B&0\\
        \Lambda_{d}\bullet\mathbb{J}\bullet(\delta_{d-1,l}\boxtimes\id_B)&(-1)^d \Lambda_{d}\cdot- \\
        -((x_i^*\otimes\id_V^{\otimes d})\upsilon_{d-1,l})\otimes\id_B&-(x_i^*\otimes\id^{\otimes d}_V\otimes\id_B)\underline{\delta_{d,l}\boxtimes\id_B}
    \end{pmatrix},
    }
    \\
        \psi_{d+1;a_0a_1a_2}=&
    \begin{pmatrix}
    0\\
        \mathbb{J}^T\begin{pmatrix}a_1\\a_2\end{pmatrix} 
    \end{pmatrix}+a_0\begin{pmatrix}
        (-1)^{d+1}(x_i^*\det\sigma\otimes\id_V^{\otimes d-1})\otimes\id_B\\
        \Lambda_{d+1}\bullet \mathbb{J}\bullet(\delta_{d,l}\boxtimes\id_B)
    \end{pmatrix},\\
    \Lambda_j=&\begin{pmatrix}
            x_i^*\sigma_{11}\otimes \id_V^{\otimes j-1}\otimes\id_B &x_i^*\sigma_{21}\otimes \id_V^{\otimes j-1}\otimes\id_B\\
            x_i^*\sigma_{12}\otimes \id_V^{\otimes j-1}\otimes\id_B &x_i^*\sigma_{22}\otimes \id_V^{\otimes j-1}\otimes\id_B
        \end{pmatrix}:V^{\otimes j}\otimes B\to M_{2\times 2}\left(V^{\otimes j-1}\otimes B(-1)\right),
\end{align*}
}
for all $s=2,\cdots,d-1$ and  $j=1,\cdots,d+1$.
\end{lemma}

\begin{proof}
Obviously, $\varepsilon_B\psi_{0;a_0a_1a_2}=((a_1\varepsilon_B,\ a_2\varepsilon_B),\ a_0x_i^*\otimes \varepsilon_B)$. We Write $\partial_{\centerdot}^F$ for the differential of the free resolution  $F_\centerdot$. So the differential of top row is $\partial_{\centerdot}^F$ and the bottom one is $\partial_{\centerdot}^{F[-1]}=-\partial_{\centerdot-1}^F$.

It is not hard to check that
$\psi_{s-1;0a_1a_2}\partial^F_{s+1}=-\partial^F_{s}\psi_{s;0a_1a_2}$ since $g_{s-1}\left(
\mathbb{J}^T
\begin{pmatrix}
    a_1\\a_2
\end{pmatrix}
\right)=(a_1,a_2)\bullet f_{s-1}
$ for any $a_1,a_2\in k$ and $s=1,\cdots,d+1$.

To prove 
$\psi_{s-1;a_000}\partial^F_{s+1}=-\partial^F_{s}\psi_{s;a_000}$ for any $a_0\in k$ and $s=1,\cdots,d+1$, we need some preparation. 
By Lemma \ref{lem: relation between Theta^{(i)}}, we have the following equalities immediately
\begin{align}
    g_{s-1}(\Lambda_s\cdot-)&=(x_i^*\otimes \id_V^{\otimes s-1}\otimes \id_B)\underline{\sigma^{\boxtimes s}\boxtimes\lambda_Y}
+
(x_i^*\otimes\id_V^{\otimes s-1}\otimes m_B)\underline{\sigma\boxtimes \delta_{s-1,r}\boxtimes\id_B},\label{eq: g Lambda}\\
\partial_{s-1}^{\oplus 2} (\Lambda_s\cdot-) &=(\Lambda_{s-1}\cdot-)\partial_{s}^{\oplus 2},\label{eq: partial Lambda}
\\
\Lambda_{s}\bullet\mathbb{J}\bullet\left(\sigma\boxtimes \alpha\right)
&=(\mathbb{J}\bullet \alpha)(x_i^*\det \sigma\otimes \id^{\otimes s-1}_V\otimes\id_B)\label{eq: Lambda J sigma},
\end{align}
for some  $\alpha:V^{\otimes s-1 }\otimes B\to (V^{\otimes s-1}\otimes B)^{\oplus 2}$ and $s=2,\cdots,d+1$. Then we have the following equalities
\begin{enumerate}[label=(\roman*)]
\item For any $s=3,\cdots,d+1$,
$$
((x^*_i\det\sigma\otimes\id^{\otimes s-3}_V)\otimes\id_B)\partial_{s-1}=x^*_i\det\sigma\otimes\id^{\otimes s-3}_V\otimes m_B=\partial_{s-2}((x^*_i\det\sigma\otimes\id^{\otimes s-2}_V)\otimes\id_B).
$$

\item For any $s=1,\cdots,d-1$,
$$
((x^*_i\otimes\id^{\otimes s-1}_V)\otimes\id_B)\partial_{s+1}=x^*_i\otimes\id^{\otimes s-1}_V\otimes m_B=\partial_{s}((x^*_i\otimes\id^{\otimes s}_V)\otimes\id_B).
$$

\item For any $s=2,\cdots,d+1$,
\begin{align*}
    &(\Lambda_{s-1}\bullet \mathbb{J}\bullet (\delta_{s-2,l}\boxtimes \id_B))\partial_{s-1}+(-1)^s\Lambda_{s-1}\bullet f_{s-1}\\
    =&\Lambda_{s-1}\bullet\left((\id^{\otimes s-1}_V\otimes m_B)^{\oplus 2}(\mathbb{J}\bullet (\delta_{s-2,l}\boxtimes\id_V\boxtimes \id_B))\right)+(-1)^{s}\Lambda_{s-1}\bullet\left((\id^{\otimes s-2}_V\otimes m_B)^{\oplus 2}\left(\mathbb{J}\bullet(
\delta_{s-1,r}\boxtimes \id_B)\right)\right)\\
&+(-1)^{s}\Lambda_{s-1}\bullet\left(\mathbb{J}\bullet(\sigma\boxtimes\sigma^{\boxtimes s-2}\boxtimes\lambda_Y)\right)\\
    =&\Lambda_{s-1}\bullet\left((\id^{\otimes s-1}_V\otimes m_B)^{\oplus 2}(\mathbb{J}\bullet (\delta_{s-2,l}\boxtimes\id_V\boxtimes \id_B))\right)+(-1)^{s}\Lambda_{s-1}\bullet\left((\id^{\otimes s-2}_V\otimes m_B)^{\oplus 2}\left(\mathbb{J}\bullet(
\delta_{s-1,r}\boxtimes \id_B)\right)\right)\\
&+(-1)^{s}\left(
\mathbb{J}\bullet (\sigma^{\boxtimes s-2}\boxtimes\lambda_Y)\right)
(x_i^*\det\sigma\otimes\id_V^{\otimes s-2}\otimes \id_B)\\
=&\Lambda_{s-1}\bullet\left((\id^{\otimes s-1}_V\otimes m_B)^{\oplus 2}(\mathbb{J}\bullet (\delta_{s-1,l}\boxtimes \id_B))\right)+(-1)^{s}\Lambda_{s-1}\bullet\left((\id^{\otimes s-2}_V\otimes m_B)^{\oplus 2}\left(\mathbb{J}\bullet(\sigma\boxtimes
\delta_{s-2,r}\boxtimes \id_B)\right)\right)
\\
&+(-1)^{s}\left(
\mathbb{J}\bullet (\sigma^{\boxtimes s-2}\boxtimes\lambda_Y)\right)
(x_i^*\det\sigma\otimes\id_V^{\otimes s-2}\otimes \id_B)\\
=&(\id^{\otimes s-2}_V\otimes m_B)^{\oplus 2}\left(\Lambda_{s-1}\bullet\mathbb{J}\bullet (\delta_{s-1,l}\boxtimes \id_B))\right)+(-1)^{s}\left(
\mathbb{J}\bullet (\sigma^{\boxtimes s-2}\boxtimes\lambda_Y)\right)
(x_i^*\det\sigma\otimes\id_V^{\otimes s-2}\otimes \id_B)
\\
&+(-1)^{s}(\id^{\otimes s-2}_V\otimes m_B)^{\oplus 2}\left(\mathbb{J}\bullet(
\delta_{s-2,r}\boxtimes \id_B)\right)(x_i^*\det\sigma\otimes\id_V^{\otimes s-2}\otimes \id_B)\\
=&(-1)^sf_{s-2}(
 (x_i^*\det\sigma\otimes\id_V^{\otimes s-2})\otimes \id_B
 )+\partial_{s-1}^{\oplus 2}\left(\Lambda_{s-1}\bullet
\mathbb{J}\bullet(\delta_{s-1,l}\boxtimes\id_B)\right),
\end{align*}
where the second and fourth equalities hold by \eqref{eq: Lambda J sigma}, and the third equality holds by Lemma \ref{lemma: construction of delta_r and delta_l}.

\item  For any $s=1,\cdots, d$,
\begin{align*}
    &(-1)^{s+1}g_{s-1}(\Lambda_{s}\cdot-)+\partial_s
    \left(x_i^*\otimes\id_V^{\otimes s}\otimes\id_B\right)\underline{\delta_{s,l}\boxtimes\id_B}\\
    =&(-1)^{s+1}(x_i^*\otimes\id_V^{\otimes s-1}\otimes\id_B)\underline{\sigma^{\boxtimes s}\boxtimes\lambda_Y}+(-1)^{s+1}
    (x_i^*\otimes\id_V^{\otimes s-1}\otimes m_B)\underline{\sigma\boxtimes\delta_{s-1,r}\boxtimes\id_B}\\
     &+(x_i^*\otimes\id_V^{\otimes s-1}\otimes m_B)\underline{\delta_{s,l}\boxtimes\id_B}\\
     =&(-1)^{s+1}(x_i^*\otimes\id_V^{\otimes s-1}\otimes\id_B)\underline{\sigma^{\boxtimes s}\boxtimes\lambda_Y}+(x_i^*\otimes\id_V^{\otimes s-1}\otimes m_B)\underline{\left(\delta_{s-1,l}\boxtimes\id_B+(-1)^{s+1}\delta_{s,r}\right)\boxtimes\id_B}\\
     =&(-1)^{s-1}\left((x_i^*\otimes\id_V^{\otimes s-1})\otimes\id_B)\right)g_s+\left(
(x_i^*\otimes\id_V^{\otimes s-1}\otimes\id_B)\underline{\delta_{s-1,l}\boxtimes\id_B}
     \right)\partial_s^{\oplus 2},
\end{align*}
where the first equality holds by \eqref{eq: g Lambda}, and the second equality holds by Lemma \ref{lemma: construction of delta_r and delta_l}.

\item For any $s=2,\cdots,d+1$,
\begin{align*}
&h_{s-2}\left(
(x_i^*\det\sigma\otimes\id_V^{\otimes s-2})\otimes\id_B
    \right)-g_{s-1}\left(\Lambda_{s}\bullet\mathbb{J}\bullet (\delta_{s-1,l}\boxtimes\id_B)\right)+\partial_s\left(\left(
    (x_i^*\otimes\id_V^{\otimes s})\upsilon_{s-1,l}\right)\otimes\id_B
    \right)\\
=&(x_i^*\otimes\id_V^{\otimes s-1}\otimes\id_B)\left((-1)^{s+1}\sum_{u=1}^{s-2}(-1)^{u+1}(\det\sigma)^{\otimes s-u-1}\boxtimes \Gamma_{u,l}\boxtimes\lambda_Y
-\left(
(\sigma^{\boxtimes s})^T\bullet \mathbb{J}\bullet \delta_{s-1,l}
\right)^T\boxtimes \lambda_Y
\right)\\
&+(x_i^*\otimes\id_V^{\otimes s-1}\otimes m_B)
\left(
\left(
\upsilon_{s-1,l}+(-1)^{s+1}\det\sigma\otimes\upsilon_{s-2,r}-\underline{\sigma\boxtimes\delta_{s-1,r}}(\mathbb{J}\bullet\delta_{s-1,l})
\right)\otimes\id_B
\right)\\
=&(x_i^*\otimes\id_V^{\otimes s-1}\otimes\id_B)\left((-1)^{s+1}\sum_{u=1}^{s-1}(-1)^{u+1}(\det\sigma)^{\otimes s-u-1}\boxtimes \Gamma_{u,l}\boxtimes\lambda_Y
+
\left(\delta_{s-1,l}^T\bullet\mathbb{J}\bullet\sigma^{\boxtimes s-1}\right)\boxtimes\lambda_Y
\right)\\
&+(x_i^*\otimes\id_V^{\otimes s-1}\otimes m_B)
\left(
\left(
(-1)^{s-1}\upsilon_{s-1,r}-\upsilon_{s-2,l}\otimes\id_V+\underline{\delta_{s-1,l}\boxtimes\id_V}(\mathbb{J}\bullet\delta_{s-1,r})
\right)\otimes\id_B
\right)\\
=&(-1)^{s-1}\left((x_i^*\otimes\id_V^{\otimes s-1})\otimes\id_B\right)h_{s-1}
+\left(
(x_i^*\otimes\id_V^{\otimes s-1}\otimes\id_B)\underline{\delta_{s-1,l}\boxtimes \id_B}
\right)f_{s-1}\\
&-
\left(
((x_i^*\otimes\id_V^{\otimes s-1})\upsilon_{s-2,l})\otimes \id_B
\right)\partial_{s-1},
\end{align*}
where the first equality is deduced by \eqref{eq: g Lambda}, the second equality holds by the definition of $\Gamma_{s-1,l}$ and Lemma \ref{lem: construction of upsilon l}.
\end{enumerate}

Then $\psi_{s-1;a_000}\partial^F_{s+1}=-\partial^F_{s}\psi_{s;a_000}$ for any $=1,\cdots,d+1$ by  \eqref{eq: partial Lambda} and (i)--(v).
\end{proof}

\begin{proof}[Proof of Lemma \ref{lemma: Yoneda product}]
By Lemma \ref{lem: E^1(B) E^{d+1}(B)}, on obtains that
\begin{align*}
\begin{pmatrix}
    \xi^*_1\ast \omega^*_1 &\xi^*_2\ast \omega^*_1   &x^*_i\ast\omega^*_1\\
    \xi^*_1\ast \omega^*_2&\xi^*_2\ast \omega^*_2& x_i^*\ast\omega^*_2\\
    \xi^*_1\ast \eta^*_j  & \xi^*_2\ast \eta^*_j&   x_i^*\ast \eta^*_j
\end{pmatrix}&=
\begin{pmatrix}
    ((\varepsilon_B,0),0) \varphi_{2;010} &((0,\varepsilon_B),0) \varphi_{2;010}&\left((0,0),x_i^*\otimes \varepsilon_B\right)\varphi_{2;010}\\
   ((\varepsilon_B,0),0) \varphi_{2;001}  &((0,\varepsilon_B),0) \varphi_{2;001}&\left((0,0),x_i^*\otimes \varepsilon_B\right)\varphi_{2;001}\\
   ((\varepsilon_B,0),0)\varphi_{2;100} &((0,\varepsilon_B),0)\varphi_{2;100} &\left((0,0),x_i^*\otimes \varepsilon_B\right)\varphi_{2;100}
\end{pmatrix}\\
&=(-1)^d
\begin{pmatrix}(\mathbb{J}\hdet\varsigma)(\omega^*\otimes\varepsilon_B)&  \left(\mathbb{J}\left((x_i^*)^{\oplus 2}\delta_r\right)\right)(\omega^*\otimes \varepsilon_B)\\
   0& -(\eta_j^*\otimes x_i^*)(\omega)(\omega^*\otimes\varepsilon_B)
\end{pmatrix}\\
&=(-1)^d
\begin{pmatrix}\mathbb{J}\hdet\varsigma&  \mathbb{J}\left((x_i^*)^{\oplus 2}\delta_r\right)\\
 0&   -(\eta_j^*\otimes x_i^*)(\omega)
\end{pmatrix}\widetilde{\omega}^*.
\end{align*}

By Lemma \ref{lem: E{d+1}(B) ast E1(B)}, we have
\begin{align*}
&\begin{pmatrix}
\omega_1^*\ast\xi_1 & \omega_1^*\ast\xi_2& \omega_1^*\ast x^*_i\\
\omega_2^*\ast\xi_1 & \omega_2^*\ast\xi_2 &\omega_2^*\ast x^*_i \\
\eta_j^*\ast\xi_1&\eta_j^*\ast \xi_2&\eta_j^*\ast x^*_i    
\end{pmatrix}\\
=&
\begin{pmatrix}
    (0,(\omega^*\otimes\varepsilon_B,0)) \psi_{d+1;010} &(0,(\omega^*\otimes\varepsilon_B,0)) \psi_{d+1;001}&(0,(\omega^*\otimes\varepsilon_B,0))\psi_{d+1;100}\\
      (0,(0,\omega^*\otimes\varepsilon_B)) \psi_{d+1;010} &(0,(0,\omega^*\otimes\varepsilon_B)) \psi_{d+1;001}&(0,(0,\omega^*\otimes\varepsilon_B))\psi_{d+1;100}\\
   (\eta_j^*\otimes\varepsilon_B,(0,0))\psi_{d+1;010} & (\eta_j^*\otimes\varepsilon_B,(0,0))\psi_{d+1;001} & (\eta_j^*\otimes\varepsilon_B,(0,0))\psi_{d+1;100} 
\end{pmatrix}\\
=&\begin{pmatrix}
\mathbb{J}^T(\omega^*\otimes\varepsilon_B)& 
(x_i^*)^{\oplus2}\left(\sigma^T\cdot(\mathbb{J}\delta_l)\right)(\omega^*\otimes\varepsilon_B)\\
0&(-1)^{d+1}(x^*_i\det\sigma\otimes\eta_j^*)(\omega)(\omega^*\otimes\varepsilon_B)
\end{pmatrix}\\
=&\begin{pmatrix}
\mathbb{J}^T& 
(x_i^*)^{\oplus2}\left(\sigma^T\cdot(\mathbb{J}\delta_l)\right)\\
0&(-1)^{d+1}(x^*_i\det\sigma\otimes\eta_j^*)(\omega)
\end{pmatrix}\widetilde{\omega}^*.\qedhere
\end{align*}
\end{proof}

It turns to prove Theorem \ref{thm: twisted superpotential}. There is a useful equality in the proof of Theorem \ref{thm: twisted superpotential}. 

\begin{lemma}\label{lem: transfer for the last part of twisted superpotential}
\begin{align*}
    &\sum_{i=1}^d(-1)^i\left(\underline{\sigma\boxtimes \delta_{i,r}\boxtimes\id^{\boxtimes d-i}_V} \left(\mathbb{J}\bullet (\delta_{i,l}\boxtimes\id^{\boxtimes d-i}_V)\right)+\underline{\delta_{i,l}\boxtimes\id^{\boxtimes d-i+1}_V}\left(\mathbb{J}\bullet(\delta_{{i},r}\boxtimes\id^{\boxtimes d-i}_V)\right)\right)\\
&+\sum_{i=1}^{d}\sum_{j=1}^i(-1)^{i+j}\underline{\delta_{j,r}\boxtimes\id^{\boxtimes d+1-j}}\left(\mathbb{J}\bullet\left(\delta_{i,r}\boxtimes \id^{\boxtimes d-i}\right)\right)\\
=& \sum_{j=1}^{d}(-1)^{j}\underline{\sigma\boxtimes\delta_{j,r}\boxtimes\id^{\boxtimes d-j}}\left(\mathbb{J}\bullet\delta_{d,l}\right)+\sum_{i=1}^{d-1}\sum_{j=1}^{i}(-1)^{i+j}\det\sigma\otimes\left(\underline{\delta_{j,r}\boxtimes\id^{\boxtimes d-j}}\left(\mathbb{J}\bullet\left(\delta_{i,r}\boxtimes \id^{\boxtimes d-i-1}\right)\right)\right).
\end{align*} 
\end{lemma}
\begin{proof}
\begin{align*}
&\sum_{i=1}^d(-1)^i\left(\underline{\sigma\boxtimes \delta_{i,r}\boxtimes\id^{\boxtimes d-i}_V} \left(\mathbb{J}\bullet (\delta_{i,l}\boxtimes\id^{\boxtimes d-i}_V)\right)+\underline{\delta_{i,l}\boxtimes\id^{\boxtimes d-i+1}_V}\left(\mathbb{J}\bullet(\delta_{{i},r}\boxtimes\id^{\boxtimes d-i}_V)\right)\right)\\
&+\sum_{i=1}^{d}\sum_{j=1}^i(-1)^{i+j}\underline{\delta_{j,r}\boxtimes\id^{\boxtimes d+1-j}}\left(\mathbb{J}\bullet\left(\delta_{i,r}\boxtimes \id^{\boxtimes d-i}\right)\right)\\
=&  \sum_{i=1}^d(-1)^i\left(\underline{\sigma\boxtimes \delta_{i,r}\boxtimes\id^{\boxtimes d-i}_V} \left(\mathbb{J}\bullet (\delta_{i,l}\boxtimes\id^{\boxtimes d-i}_V)\right)+\underline{\delta_{i,l}\boxtimes\id^{\boxtimes d-i+1}_V}\left(\mathbb{J}\bullet(\delta_{{i},r}\boxtimes\id^{\boxtimes d-i}_V)\right)\right.\\
&\left.+(-1)^{i+1}\sum_{j=1}^i(-1)^{j}\underline{\sigma^{\boxtimes i-j}\boxtimes\delta_{j,l}\boxtimes\id^{\boxtimes d+1-i}}\left(\mathbb{J}\bullet\left(\delta_{i,r}\boxtimes \id^{\boxtimes d-i}\right)\right)\right)\\
=& \sum_{i=1}^d(-1)^i\left(\underline{\sigma\boxtimes \delta_{i,r}\boxtimes\id^{\boxtimes d-i}_V} \left(\mathbb{J}\bullet (\delta_{i,l}\boxtimes\id^{\boxtimes d-i}_V)\right)-\sum_{j=1}^{i-1}(-1)^{j}\underline{\sigma\boxtimes\delta_{j,r}\boxtimes\id^{\boxtimes d-j}}\left(\mathbb{J}\bullet\left(\delta_{i,r}\boxtimes \id^{\boxtimes d-i}\right)\right)\right)\\
=& \sum_{i=1}^d(-1)^i\underline{\sigma\boxtimes \delta_{i,r}\boxtimes\id^{\boxtimes d-i}_V} \left(\mathbb{J}\bullet (\delta_{i,l}\boxtimes\id^{\boxtimes d-i}_V)\right)+\sum_{i=1}^d\sum_{j=1}^{i-1}(-1)^{j}\underline{\sigma\boxtimes\delta_{j,r}\boxtimes\id^{\boxtimes d-j}}\left(\mathbb{J}\bullet\left(\delta_{i,l}\boxtimes \id^{\boxtimes d-i}\right)\right)\\
&-\sum_{i=1}^d\sum_{j=1}^{i-1}(-1)^{j}\underline{\sigma\boxtimes\delta_{j,r}\boxtimes\id^{\boxtimes d-j}}\left(\mathbb{J}\bullet\left(\delta_{i-1,l}\boxtimes \id^{\boxtimes d-i+1}\right)\right)\\
&-\sum_{i=1}^d\sum_{j=1}^{i-1}(-1)^{i+j}\underline{\sigma\boxtimes\delta_{j,r}\boxtimes\id^{\boxtimes d-j}}\left(\mathbb{J}\bullet\left(\sigma\boxtimes \delta_{i-1,r}\boxtimes \id^{\boxtimes d-i}\right)\right)\\
=& \sum_{i=1}^d(-1)^i\underline{\sigma\boxtimes \delta_{i,r}\boxtimes\id^{\boxtimes d-i}_V} \left(\mathbb{J}\bullet (\delta_{i,l}\boxtimes\id^{\boxtimes d-i}_V)\right)+\sum_{i=1}^d\sum_{j=1}^{i-1}(-1)^{j}\underline{\sigma\boxtimes\delta_{j,r}\boxtimes\id^{\boxtimes d-j}}\left(\mathbb{J}\bullet\left(\delta_{i,l}\boxtimes \id^{\boxtimes d-i}\right)\right)\\
&-\sum_{i=1}^{d-1}\sum_{j=1}^{i}(-1)^{j}\underline{\sigma\boxtimes\delta_{j,r}\boxtimes\id^{\boxtimes d-j}}\left(\mathbb{J}\bullet\left(\delta_{i,l}\boxtimes \id^{\boxtimes d-i}\right)\right)\\
&-\sum_{i=1}^d\sum_{j=1}^{i-1}(-1)^{i+j}\underline{\sigma\boxtimes\delta_{j,r}\boxtimes\id^{\boxtimes d-j}}\left(\mathbb{J}\bullet\left(\sigma\boxtimes \delta_{i-1,r}\boxtimes \id^{\boxtimes d-i}\right)\right)\\
=& \sum_{j=1}^{d}(-1)^{j}\underline{\sigma\boxtimes\delta_{j,r}\boxtimes\id^{\boxtimes d-j}}\left(\mathbb{J}\bullet\delta_{d,l}\right)+\sum_{i=1}^{d-1}\sum_{j=1}^{i}(-1)^{i+j}\det\sigma\otimes\left(\underline{\delta_{j,r}\boxtimes\id^{\boxtimes d-j}}\left(\mathbb{J}\bullet\left(\delta_{i,r}\boxtimes \id^{\boxtimes d-i-1}\right)\right)\right),
\end{align*}
where the first and second equality hold by Proposition \ref{prop: relation between delta_r and delta_l}(a,b), the third equality holds by Lemma \ref{lemma: construction of delta_r and delta_l} and the fifth equality holds by Lemma \ref{lem: relation between Theta^{(i)}}.
\end{proof}

\begin{proof}[Proof of Theorem \ref{thm: twisted superpotential}]
We divide $\hat{\omega}$ into $6$ parts to show it is a ${\mu_B}_{\mid V}$-twisted superpotential.

\noindent$\bullet$ The first part:
\begin{align*}
&\tau_{d+2}^{d+1}(\mu_B\otimes\id^{\otimes d+1})\sum_{j=0}^{d}(-1)^j\left(\left( \lambda^T_Y\boxtimes(\mathbb{J}\bullet\sigma^{\boxtimes j})\boxtimes\lambda_Y\right)\otimes\id^{\otimes d-j}\right) \left(\id\otimes\tau^{j}_{d+1}\right)(1\otimes 1\otimes\omega)\\
=&\tau_{d+2}^{d+1}\sum_{j=0}^{d}(-1)^{j+1}\left(\left((\mathbb{J}\bullet \hdet\varsigma \bullet \lambda_Y)^T\boxtimes \sigma^{\boxtimes j} \boxtimes\lambda_Y\right)\otimes\id^{\otimes d-j}\right) \left(\id\otimes\tau^{j}_{d+1}\right)(1\otimes 1\otimes\omega)\\
&+\tau_{d+2}^{d+1}\sum_{j=0}^{d}(-1)^{j+1}\left(\left((\mathbb{J}\mathrm{div}_{\varsigma}\nu)^T\boxtimes \sigma^{\boxtimes j} \boxtimes\lambda_Y\right)\otimes\id^{\otimes d-j}\right) \left(\id\otimes\tau^{j}_{d+1}\right)(1\otimes 1\otimes\omega)
\\
=&\sum_{j=0}^{d}(-1)^{j+1}\tau_{d+2}^j\left(\lambda^T_Y\boxtimes\left(\left(\sigma^{\boxtimes j} \boxtimes \mathcal{I}^{d-j}\right)^T\bullet \mathbb{J}\bullet \sigma^{\boxtimes d} \right)\boxtimes \lambda_Y\right)\tau^{d+1}_{d+2}(1\otimes 1\otimes\omega)\\
&+\sum_{j=0}^{d}(-1)^{j+1}\tau_{d+2}^j\left(\lambda^T_Y\boxtimes\left(\left(\sigma^{\boxtimes j} \boxtimes \mathcal{I}^{d-j}\right)^T\bullet \mathbb{J}\right)\boxtimes \mathrm{div
}_\varsigma\nu\right)\tau^{d+1}_{d+2}(1\otimes 1\otimes\omega)
\\
=&(-1)^{d+1}\sum_{i=0}^{d}(-1)^{d-i}\left(\det\sigma^{\otimes i}\boxtimes\lambda^T_Y\boxtimes \left(\mathbb{J}\bullet\sigma^{\boxtimes d-i}\right)\boxtimes \lambda_Y\right)\left(\tau_{d+2}^i(\id\otimes\tau^{d}_{d+1})\right)(1\otimes 1\otimes\omega)\\
&+(-1)^{d+1}\sum_{j=0}^{d}(-1)^{d+j}\tau_{d+2}^j\left(\lambda^T_Y\boxtimes\left(\left(\sigma^{\boxtimes j} \boxtimes \mathcal{I}^{d-j+1}\right)^T\bullet \mathbb{J}\bullet \delta_{d,r}\right)\right)(1\otimes\omega)\\
&+(-1)^{d+1}\sum_{j=0}^{d}(-1)^{d+j}\tau^{j}_{d+2}(\id^{\otimes d+1}\otimes\mu_A)\underline{\left(\lambda^T_Y\boxtimes
(\sigma^{\boxtimes j}\boxtimes
\mathcal{I}^{d-j})^T\boxtimes \left( \mathbb{J}\bullet\sigma^{-T}\right)\right)^T}
\begin{pmatrix}
    1\otimes \omega\otimes\delta_{l;1}\\
    1\otimes\omega\otimes \delta_{l;2}
\end{pmatrix},
\end{align*}

\noindent where the second equality holds by Theorem \ref{thm: hdet is sigma d} and the third equality holds by Lemma \ref{lem: relation between Theta^{(i)}}.

\noindent$\bullet$  The second part:
\begin{align*}
&\tau_{d+2}^{d+1}(\mu_B\otimes\id^{\otimes d+1})\sum_{i=1}^{d}\sum_{j=0}^{d-i}(-1)^j\left((\det\sigma)^{\otimes i}\otimes\left( \lambda^T_Y\boxtimes(\mathbb{J}\bullet\sigma^{\boxtimes j})\boxtimes\lambda_Y\right)\otimes\id^{\otimes d-i-j}\right) \left(\tau^{i}_{d+2}(\id\otimes\tau^{i+j}_{d+1})\right)(1\otimes 1\otimes\omega)\\
=&\sum_{i=1}^{d}\sum_{j=0}^{d-i}(-1)^j\left((\det\sigma)^{\otimes i-1}\otimes\left( \lambda^T_Y\boxtimes(\mathbb{J}\bullet\sigma^{\boxtimes j})\boxtimes\lambda_Y\right)\otimes\id^{\otimes d-i-j+1}\right) \left(\tau^{i-1}_{d+2}(\id\otimes\tau^{i+j-1}_{d+1})\right)\\
&\left(\id^{\otimes 2}\otimes\tau^{d-1}_{d}(\mu_A\otimes \id^{\otimes d-1})\right)(1\otimes 1\otimes\omega)\\
=&(-1)^{d+1}\sum_{i=0}^{d-1}\sum_{j=0}^{d-i-1}(-1)^j\left((\det\sigma)^{\otimes i}\otimes\left( \lambda^T_Y\boxtimes(\mathbb{J}\bullet\sigma^{\boxtimes j})\boxtimes\lambda_Y\right)\otimes\id^{\otimes d-i-j}\right) \left(\tau^{i}_{d+2}(\id\otimes\tau^{i+j}_{d+1})\right)(1\otimes 1\otimes\omega),
\end{align*}

\noindent where the second equality holds by Theorem \ref{thm: properties of Koszul regular algebras} and re-indexing.

\noindent$\bullet$  The third part:
\begin{align*}
&\tau_{d+2}^{d+1}(\mu_B\otimes\id^{\otimes d+1})\sum_{j=1}^{d+1}\sum_{i=1}^{d}(-1)^{i+j} \tau^{j}_{d+2}\left(\lambda^T_Y\boxtimes\left((\sigma^{\boxtimes j}\boxtimes
\mathcal{I}^{d+1-j})^T\bullet\mathbb{J}\bullet(\delta_{i,r}\boxtimes \id^{\boxtimes d-i})\right)\right) (1\otimes\omega)\\
=&(-1)^{d+1}\tau_{d+2}^{d+1}(\mu_B\otimes\id^{\otimes d+1})\sum_{j=1}^{d+1}\sum_{i=1}^{d}(-1)^{i+j} \tau^{j}_{d+2}\left(\lambda^T_Y\boxtimes\left((\sigma^{\boxtimes j}\boxtimes
\mathcal{I}^{d+1-j})^T\bullet\mathbb{J}\bullet(\sigma^{\boxtimes d-i}\boxtimes\delta_{i,l})\right)\right) (1\otimes\omega)\\
=&(-1)^{d+1}\tau_{d+2}^{d+1}(\mu_B\otimes\id^{\otimes d+1})\sum_{j=1}^{d+1}\left(\sum_{i=1}^{d-j}(-1)^{i+j} \tau^{j}_{d+2}\left(\lambda^T_Y\boxtimes\left((\sigma^{\boxtimes j}\boxtimes
\mathcal{I}^{d+1-j})^T\bullet\mathbb{J}\bullet(\sigma^{\boxtimes d-i}\boxtimes\delta_{i,l})\right)\right) \right.\\
&+\left.\sum_{i=d-j+1}^{d-1}(-1)^{i+j} \tau^{j}_{d+2}\left(\lambda^T_Y\boxtimes\left((\sigma^{\boxtimes j}\boxtimes
\mathcal{I}^{d+1-j})^T\bullet\mathbb{J}\bullet(\sigma^{\boxtimes d-i}\boxtimes\delta_{i,l})\right)\right)\right) (1\otimes\omega)\\
&+(-1)^{d+1}\tau_{d+2}^{d+1}(\mu_B\otimes\id^{\otimes d+1})\sum_{j=1}^{d+1}(-1)^{d+j} \tau^{j}_{d+2}\left(\lambda^T_Y\boxtimes\left((\sigma^{\boxtimes j}\boxtimes
\mathcal{I}^{d+1-j})^T\bullet\mathbb{J}\bullet \delta_{d,l}\right)\right) (1\otimes\omega)\\
=&(-1)^{d+1}\sum_{j=1}^{d+1}(-1)^j\tau_{d+2}^{j-1}\left(\sum_{i=1}^{d-j}(-1)^{i} \left(\lambda^T_Y\boxtimes\scalebox{0.7}{$
    \begin{pmatrix}
    (\det\sigma)^{\otimes j-1} &\\
    &(\det\sigma)^{\otimes j-1}
\end{pmatrix}
$}\boxtimes\left(\mathbb{J}\bullet(\sigma^{\boxtimes d-j-i}\boxtimes\delta_{i,l})\right)\boxtimes\id\right)+\sum_{i=d-j+1}^{d-1}(-1)^{i}  \right.\\
&\left.\left(\lambda_Y^T\boxtimes
\scalebox{0.7}{$
\begin{pmatrix}
    (\det\sigma)^{\otimes d-i-1} &\\
    &(\det\sigma)^{\otimes d-i-1}
\end{pmatrix}
$}\boxtimes \left((\sigma^{\boxtimes j+i-d}\boxtimes
\mathcal{I}^{d+1-j})^T\bullet\mathbb{J}\bullet\delta_{i,l}\right)\boxtimes\id\right)\right)\left(1\otimes\tau_{d}^{d-1}(\mu_A\otimes\id^{\otimes d-1})(\omega)\right)\\
&+(-1)^{d+1}
\sum_{j=1}^{d+1}(-1)^{d+j} 
\tau^{j-1}_{d+2}(\id^{\otimes d+1}\otimes\mu_A)
\underline{\left(\lambda^T_Y\boxtimes
(\sigma^{\boxtimes j}\boxtimes
\mathcal{I}^{d-j})^T\boxtimes \left(\scalebox{0.7}{$\begin{pmatrix}
(\det\sigma)^{-1}&\\
&(\det\sigma)^{-1}
\end{pmatrix}$}\bullet\sigma^{T}\bullet \mathbb{J}\right)\right)^T}
\begin{pmatrix}
    1\otimes \omega\otimes\delta_{l;1}\\
    1\otimes\omega\otimes \delta_{l;2}
\end{pmatrix}
\\
=&\sum_{j=0}^{d}\sum_{i=1}^{d-1}(-1)^{i+j+1}\tau_{d+2}^{j} \left(\lambda^T_Y\boxtimes\left((\sigma^{\boxtimes j}\boxtimes
\mathcal{I}^{d-j+1})^T\bullet\mathbb{J}\bullet(\sigma^{\boxtimes d-i-1}\boxtimes\delta_{i,l}\boxtimes \id)\right)\right) (1\otimes\omega)\\
&+(-1)^{d+1}
\sum_{j=0}^{d}(-1)^{d+j+1} 
\tau^{j}_{d+2}(\id^{\otimes d+1}\otimes\mu_A)
\underline{\left(\lambda^T_Y\boxtimes
(\sigma^{\boxtimes j+1}\boxtimes
\mathcal{I}^{d-j-1})^T\boxtimes \left(\mathbb{J}\bullet\sigma^{-T}\right)\right)^T}
\begin{pmatrix}
    1\otimes \omega\otimes\delta_{l;1}\\
    1\otimes\omega\otimes \delta_{l;2}
\end{pmatrix}\\
=&(-1)^{d+1}\sum_{j=0}^{d}\sum_{i=1}^{d-1}(-1)^{i+j}\tau_{d+2}^{j} \left(\lambda^T_Y\boxtimes\left((\sigma^{\boxtimes j}\boxtimes
\mathcal{I}^{d-j+1})^T\bullet\mathbb{J}\bullet(\delta_{i,r}\boxtimes \id^{\boxtimes d-i})\right)\right) (1\otimes\omega)\\
&+(-1)^{d+1}
\sum_{j=0}^{d}(-1)^{d+j+1} 
\tau^{j}_{d+2}(\id^{\otimes d+1}\otimes\mu_A)
\underline{\left(\lambda^T_Y\boxtimes
(\sigma^{\boxtimes j+1}\boxtimes
\mathcal{I}^{d-j-1})^T\boxtimes \left(\mathbb{J}\bullet\sigma^{-T}\right)\right)^T}
\begin{pmatrix}
    1\otimes \omega\otimes\delta_{l;1}\\
    1\otimes\omega\otimes \delta_{l;2}
\end{pmatrix},
\end{align*}

\noindent where the first equality holds by Proposition \ref{prop: relation between delta_r and delta_l}(a), the third one holds by Lemma \ref{lem: relation between Theta^{(i)}}, the fourth one holds by  Theorem \ref{thm: properties of Koszul regular algebras}, Lemma \ref{lem: relation between Theta^{(i)}} and \eqref{eq: t-inverse of sigma}, and the fifth one holds by Proposition \ref{prop: relation between delta_r and delta_l}(b).

\noindent$\bullet$  The fourth part:
\begin{align*}
& \tau_{d+2}^{d+1}(\mu_B\otimes\id^{\otimes d+1}) \sum_{i=1}^{d}(-1)^{i} \left(\lambda^T_Y\boxtimes\left(\mathbb{J}\bullet(\delta_{i,r}\boxtimes \id^{\boxtimes d-i})\right)\right) (1\otimes\omega)\\
=&\sum_{i=1}^{d}(-1)^{i+1} \left(\left((\delta_{i,r}\boxtimes \id^{ \boxtimes d-i})^T\bullet \mathbb{J}\bullet\sigma^{\boxtimes d}\right)\boxtimes \lambda_Y\right) (\omega\otimes 1)+\sum_{i=1}^{d}(-1)^{i+1} \left(\left((\delta_{i,r}\boxtimes \id^{ \boxtimes d-i})^T\bullet \mathbb{J}\right)\boxtimes \mathrm{div}_\varsigma\nu\right) (\omega\otimes 1)\\\
=&\sum_{i=1}^d(-1)^i\left(\left(\left(\sigma^{\boxtimes d+1}\right)^T\bullet\mathbb{J}\bullet(\delta_{i,r}\boxtimes\id^{\boxtimes d-i})\right)\boxtimes \lambda_Y\right) (\omega\otimes 1)+\sum_{i=1}^d(-1)^{i+1}\underline{\delta_{i,r}\boxtimes\id^{\boxtimes d-i+1}}(\mathbb{J}\bullet\delta_{d,r})(\omega)\\
     &+\sum_{i=1}^d(-1)^{i+1}(\id^{\otimes d+1}\otimes\mu_A)\underline{\delta_{i,r}\boxtimes\id^{\boxtimes d-i+1}}\left(
     \left(\mathcal{I}^{d}\boxtimes \left(\mathbb{J}\bullet \sigma^{-T}\right)\right)\cdot
     \begin{pmatrix}
        \omega\otimes\delta_{l;1}\\
         \omega\otimes\delta_{l;2}
     \end{pmatrix}
     \right),
\end{align*}

\noindent where the first equality holds by Theorem \ref{thm: hdet is sigma d},  and the second one holds by Lemma \ref{lem: realtions between Gamma}.

We use Proposition \ref{prop: relation between delta_r and delta_l}(c) to replace $-\sum_{i=1}^{d-1}\upsilon_{i,r}\otimes\id^{\otimes d-i}$, and there are two remaining parts.

\noindent$\bullet$  The fifth part:
\begin{align*}
&\tau_{d+2}^{d+1}(\mu_B\otimes\id^{\otimes d+1})(-1)
\left(\sum_{i=1}^{d-1}(-1)^i(\det\sigma)^{\otimes d-i}\otimes\upsilon_{i,l}-\sum_{i=1}^{d-1}\sum_{u=1}^i(-1)^u(\det\sigma)^{\otimes d-i}\otimes\right.\\
&\qquad\qquad\qquad\quad\left.\underline{\sigma\boxtimes \delta_{u,r}\boxtimes\id^{\boxtimes i-u}_V} \left(\mathbb{J}\bullet (\delta_{u,l}\boxtimes\id^{\boxtimes i-u}_V)\right)+\underline{\delta_{u,l}\boxtimes\id^{\boxtimes i-u+1}_V}\left(\mathbb{J}\bullet(\delta_{{u},r}\boxtimes\id^{\boxtimes i-u}_V)\right)\right)(\omega)\\
=&-\left(\sum_{i=1}^{d-1}(-1)^i(\det\sigma)^{\otimes d-i-1}\otimes\upsilon_{i,l}\otimes \id -\sum_{i=1}^{d-1}\sum_{u=1}^i(-1)^u(\det\sigma)^{\otimes d-i-1}\otimes\right.\\
&\left.\underline{\sigma\boxtimes \delta_{u,r}\boxtimes\id^{\boxtimes i-u+1}_V} \left(\mathbb{J}\bullet (\delta_{u,l}\boxtimes\id^{\boxtimes i-u+1}_V)\right)+\underline{\delta_{u,l}\boxtimes\id^{\boxtimes i-u+2}_V}\left(\mathbb{J}\bullet(\delta_{{u},r}\boxtimes\id^{\boxtimes i-u+1}_V)\right)\right)\tau_{d}^{d-1}(\mu_A\otimes\id^{\otimes d-1})(\omega)\\
=&(-1)^{d+1}\left(-\sum_{i=1}^{d-1}\upsilon_{i,r}\otimes\id^{\otimes d-i}(\omega)\right).
\end{align*}
\noindent It holds by Proposition \ref{prop: relation between delta_r and delta_l}(d).

\noindent$\bullet$  The sixth part:
\begin{align*}
&\tau_{d+2}^{d+1}(\mu_B\otimes\id^{\otimes d+1})\left(\sum_{u=1}^d(-1)^u\left(\underline{\sigma\boxtimes \delta_{u,r}\boxtimes\id^{\boxtimes d-u}_V} \left(\mathbb{J}\bullet (\delta_{u,l}\boxtimes\id^{\boxtimes d-u}_V)\right)+\underline{\delta_{u,l}\boxtimes\id^{\boxtimes d-u+1}_V}\left(\mathbb{J}\bullet(\delta_{{u},r}\boxtimes\id^{\boxtimes d-u}_V)\right)\right)\right.\\    &\left.+\sum_{i=1}^{d}\sum_{j=1}^i(-1)^{i+j}\underline{\delta_{j,r}\boxtimes\id^{\boxtimes d+1-j}}\left(\mathbb{J}\bullet\left(\delta_{i,r}\boxtimes \id^{\boxtimes d-i}\right)\right)\right)(\omega)\\
    =&\tau_{d+2}^{d+1}(\mu_B\otimes\id^{\otimes d+1})\sum_{j=1}^{d}(-1)^{j}\underline{\sigma\boxtimes\delta_{j,r}\boxtimes\id^{\boxtimes d-j}}\left(\mathbb{J}\bullet\delta_{d,l}\right)(\omega)\\
    &+\tau_{d+2}^{d+1}(\mu_B\otimes\id^{\otimes d+1})\sum_{i=1}^{d-1}\sum_{j=1}^{i}(-1)^{i+j}\det\sigma\otimes\left(\underline{\delta_{j,r}\boxtimes\id^{\boxtimes d-j}}\left(\mathbb{J}\bullet\left(\delta_{i,r}\boxtimes \id^{\boxtimes d-i-1}\right)\right)\right)(\omega)\\
    =&\tau_{d+2}^{d+1}(\mu_B\otimes\id^{\otimes d+1})\sum_{j=1}^{d}(-1)^{j}\underline{\sigma\boxtimes\delta_{j,r}\boxtimes\id^{\boxtimes d-j}}\left(\mathbb{J}\bullet\delta_{d,l}\right)(\omega)\\
    =&\sum_{j=1}^{d}(-1)^{j}(\id^{\otimes d+1}\otimes\mu_A)\underline{\delta_{j,r}\boxtimes\id^{\boxtimes d-j+1}}
    \left(\mathcal{I}^{d}\boxtimes
    \left(
    \begin{pmatrix}
        (\det\sigma)^{-1} & \\
            &    (\det\sigma)^{-1}
    \end{pmatrix}\bullet
    \sigma^T\bullet
    \mathbb{J}
    \right)\right)\cdot
    \begin{pmatrix}
        \omega\otimes \delta_{l;1}\\
        \omega\otimes \delta_{l;2}
    \end{pmatrix}
    \\
    &+\sum_{i=1}^{d-1}\sum_{j=1}^{i}(-1)^{i+j}\left(\underline{\delta_{j,r}\boxtimes\id^{\boxtimes d-j+1}}\left(\mathbb{J}\bullet\left(\delta_{i,r}\boxtimes \id^{\boxtimes d-i}\right)\right)\right)\tau_{d}^{d-1}(\mu_A\otimes\id^{\otimes d-1})(\omega)\\
    =&\sum_{j=1}^{d}(-1)^{j}(\id^{\otimes d+1}\otimes\mu_A)\underline{\delta_{j,r}\boxtimes\id^{\boxtimes d-j+1}}
    \left(\mathcal{I}^{d}\boxtimes
    \left(  \mathbb{J}^{-1}
    \bullet
    \sigma^{-T}  
    \right)\right)\cdot
    \begin{pmatrix}
        \omega\otimes \delta_{l;1}\\
        \omega\otimes \delta_{l;2}
    \end{pmatrix}
    \\
    &+(-1)^{d+1}\sum_{i=1}^{d-1}\sum_{j=1}^{i}(-1)^{i+j}\left(\underline{\delta_{j,r}\boxtimes\id^{\boxtimes d-j+1}}\left(\mathbb{J}\bullet\left(\delta_{i,r}\boxtimes \id^{\boxtimes d-i}\right)\right)\right)(\omega),
\end{align*}
\noindent where the first equality holds by Lemma \ref{lem: transfer for the last part of twisted superpotential}, and the third one holds by \eqref{eq: t-inverse of sigma} and Theorem \ref{thm: properties of Koszul regular algebras}.

Then $\hat{\omega}$ is a ${\mu_{B}}_{\mid \hat{V}}$-twisted superpotential by combining those six parts above.
 
The Koszul algebra $B\cong T(\hat{V})/(\hat{R})$ where 
$$\hat{R}=R\oplus k\{y_2y_1-p_{12}y_1y_2-p_{11}y_1^2\}\oplus k\{y_iv-\sigma_{i1}(v)y_1-\sigma_{i2}(v)y_2, \forall v\in V,i=1,2\}.$$
Write $\hat{W}_0=k$, $\hat{W}_1=\hat{V}$ and  
$\hat{W}_i=\bigcap_{j=0}^i\hat{V}^{\otimes j} \otimes \hat{R}\otimes \hat{V}^{\otimes i-j-2}$ for any $i\geq 2$. Using Lemma \ref{lem: prop for Gamma}, one obtains that 
$$
\hat{\omega}\in \hat{R}\otimes \hat{V}^{\otimes d}.
$$
Then we have $\hat{\omega}\in \hat{W}_{d+2}$ by the fact that $\hat{\omega}$ is a ${\mu_{B}}_{\mid \hat{V}}$-twisted superpotential. The last assertion follows by Theorem \ref{thm: properties of Koszul regular algebras}.
\end{proof}

\vskip7mm
\noindent {\bf Acknowledgments.} Yuan Shen is supported by Zhejiang Provincial Natural Science Foundation of China under Grant No. LY24A010006 and National Natural Science Foundation of China under Grant Nos. 11701515 and 12371101. Xin Wang is supported by Doctoral Research Fund of Shandong Jianzhu University under Grant No. XNBS1943.

\end{document}